\documentclass[a4paper,fleqn]{cas-sc}
\usepackage{mathrsfs}
\usepackage{amsmath,amssymb,amsfonts,amsthm}
\usepackage{overpic}
\usepackage{cases}
\theoremstyle{plain}
\newtheorem{theorem}{Theorem}[section]
\newtheorem{lemma}[theorem]{Lemma}
\newtheorem{coro}[theorem]{Corollary}
\newtheorem{prop}[theorem]{Proposition}
\theoremstyle{definition}

\newtheorem{exam}[theorem]{Example}

\newtheorem{rmk}[theorem]{Remark}
\theoremstyle{definition}

\allowdisplaybreaks
\usepackage{appendix}
\usepackage{subcaption}
\usepackage{rotating}
\usepackage{framed}
\usepackage{tikz-cd}
\usepackage{ulem}
\usepackage[authoryear,longnamesfirst]{natbib}

\def\tsc#1{\csdef{#1}{\textsc{\lowercase{#1}}\xspace}}
\tsc{WGM}
\tsc{QE}
\tsc{EP}
\tsc{PMS}
\tsc{BEC}
\tsc{DE}

\begin{document}
\let\WriteBookmarks\relax
\shorttitle{}
\shortauthors{M. Du et~al.}
\let\printorcid\relax

\title [mode = title]{Stochastic geometrically exact beams: Continuous-time geometric principles and variational discretizations}                    

\author[1]{{Minghang Du}}
\credit{Writing - review \& editing, Visualization, Software, Methodology, Conceptualization}

\author[2]{{Fran\c{c}ois Gay-Balmaz}}
\ead{francois.gb@ntu.edu.sg}
\credit{Writing - review \& editing, Supervision, Conceptualization, Funding acquisition}

\author[3]{{Tianzhi Li}}
\ead{tlee@stu.pku.edu.cn}
\credit{Writing - original draft, Validation, Methodology, Investigation, Software}

\author[1]{{Donghua Shi}}
\cormark[1]
\ead{dshi@bit.edu.cn}
\credit{Writing - review \& editing, Supervision, Funding acquisition, Conceptualization}

\affiliation[1]{organization={School of Mathematics and Statistics, Beijing Key Laboratory on MCAACI, Beijing Institute of Technology}, 
city={Beijing},
citysep={}, 
postcode={102488},
country={China}
}

\affiliation[2]{organization={Division of Mathematical Sciences, Nanyang Technological University}, 
city={Singapore},
citysep={}, 
postcode={637371},
country={Singapore}
}

\affiliation[3]{organization={School of Advanced Manufacturing and Robotics, Peking University}, 
city={Beijing},
citysep={}, 
postcode={100871},
country={China}
}

\cortext[cor1]{Corresponding author}

\begin{abstract}
Geometrically exact beams naturally arise in physical phenomena and engineering practice for modeling flexible structures, such as soft robotic manipulators and DNA structures. To the best of our knowledge, there appear to be no geometric integrators for stochastic geometrically exact beams in a variational formulation. In this paper, we introduce variational principles for the dynamics of stochastic geometrically exact beams, in both continuous- and discrete-time settings. By introducing perturbations via stochastic potentials and stochastic kinematics, we derive the equations of motion from a variational viewpoint. The resulting stochastic geometric integrators simultaneously discretize time and space in a parallel manner, which provides a distributed computational framework for solving the stochastic dynamics of complex beam systems. In particular, the proposed integrators are shown to be a stochastic extension of the well-known deterministic staggered leapfrog scheme for the wave equation, and also a stochastic extension of Yee's scheme for spatially one-dimensional computational electromagnetism. Numerical results are shown to demonstrate the effectiveness of the proposed approaches.
\end{abstract}




\maketitle

\section{Introduction}
\textbf{Motivation and related works.} Variational integrators have been a well-known family of geometric integrators that discretize variational principles that determine the equations of motion, rather than discretizing the continuous-time ODEs or PDEs. As such, the resulting numerical schemes are symplectic, preserve important invariants, and exhibit energy conservation over exponentially long time (see \cite{JEM_MW_2001}). For finite-dimensional systems, many remarkable works on variational integrators and related applications have been developed, such as Lie group variational integrators for full-body dynamics \citep{Lee2007}, variational integrators for optimal control of unconstrained/constrained systems \citep{Ober2011,Leyendecker2010,Huang2024}, discrete nonholonomic systems \citep{Zenkov2012,Ball_Zenkov_2015,Chang2019,Simoes2022,Shi2023}, variational discretization of thermodynamic systems \citep{FGB_2018,Cou2020}, variational accelerated optimization \citep{Duruisseaux2021,Duruisseaux2022,Campos2023}, filtering and estimation \citep{Fla2015,Schultz2017,li2023multisymplectic,LW_Filter_2025,li2025structure}, structure-preserving learning \citep{Eldred2024,li2023physics,Eldred2025,LW_2025,BW_AMB_2026}. For infinite-dimensional systems, variational integrators have been widely studied in a wide range of applications. Variational discretizations of fluids have been proposed in \citep{Pavlov2011,Gawlik2011,Gawlik2021}.
A Lie group variational integrator was constructed for geometrically exact beams in \cite{Demoures2015}. In addition, a variational integrator for the infinite-dimensional Chaplygin-Timoshenko sleigh was introduced in \cite{An2020} as a nonholonomic extension of \citep{Wang2016}. More recently, a constraint realization-based Hamel field variational integrator has been proposed for the geometrically exact planar Euler-Bernoulli beam \citep{WSY_2025}.

However, the above variational formulations are mathematical models that depict deterministic phenomena. Mechanical systems are inevitably subject to external noise and uncertainties. Therefore, it is of great interest in various applications to study mechanical systems with stochastic perturbations, such as stochastic geometrically exact beams. This necessitates the extension of deterministic formulations to their stochastic counterparts, in both continuous- and discrete-time settings.

To motivate our variational discretization technique, we begin by reviewing the literature on stochastic variational integrators. For finite-dimensional systems, this direction has been extensively studied, beginning with the seminal work of \cite{NBR_HO_2009}. Stochastic Hamiltonian variational integrators have been proposed in, e.g., \cite{DDH_TT_2018,Kraus2021}. Luesink et al. extended the deterministic Runge--Kutta--Munthe--Kaas methods \citep{luesink2111casimir} to the stochastic setting and constructed a variational integrator that preserves the Lie--Poisson structure. Gay-Balmaz and Wu \citep{MW_FGB_2026} generalized the midpoint integrator on vector spaces to stochastic Lie group integrators based on retraction maps, and they provided a theoretical convergence analysis. In \cite{LGSW_2025_1,LGSW_2025_2}, a stochastic nonholonomic integrator was introduced to simulate finite-dimensional stochastic nonholonomic mechanical systems.

From a field-theoretic or infinite-dimensional perspective, Hong et al. constructed a family of multisymplectic Hamiltonian integrators for stochastic Maxwell equations \citep{HoSu2022}. A structure-preserving discretization was introduced in \cite{HP_2025} for stochastic multisymplectic PDE systems. Additional discretization schemes were developed for stochastic fluid dynamics in \cite{CCHPS_2019,Franken2025,HSS_2026}. 

As a typical example of infinite-dimensional systems, the geometrically exact beam has been a well-known mathematical model for studying flexible mechanical systems with a single spatial variable, such as DNA structures and robot manipulators. In practical applications, beam dynamics are subject to uncertainties and fluctuations that are not captured by a purely deterministic model. However, to the best of our knowledge, neither variational principles for the stochastic geometrically exact beam, nor variational integrators capable of simulating its dynamics have been developed. In this paper, the stochasticity of the beam system will be modeled from the perspectives of both stochastic potentials and stochastic kinematics in a continuous-time variational manner, and further used to construct the associated stochastic variational integrators. Furthermore, we apply our stochastic integrators to the special case of deterministic systems, in order to validate the effectiveness of our approach in the deterministic setting as well. In particular, the proposed integrators exhibit good numerical momentum preservation and almost energy conservation in various examples in the deterministic setting, and Lie group structure preservation even in the stochastic setting.

\begin{figure}[pos=htbp]
\centering
\vspace{1em}
\begin{overpic}[width = 0.9\textwidth]{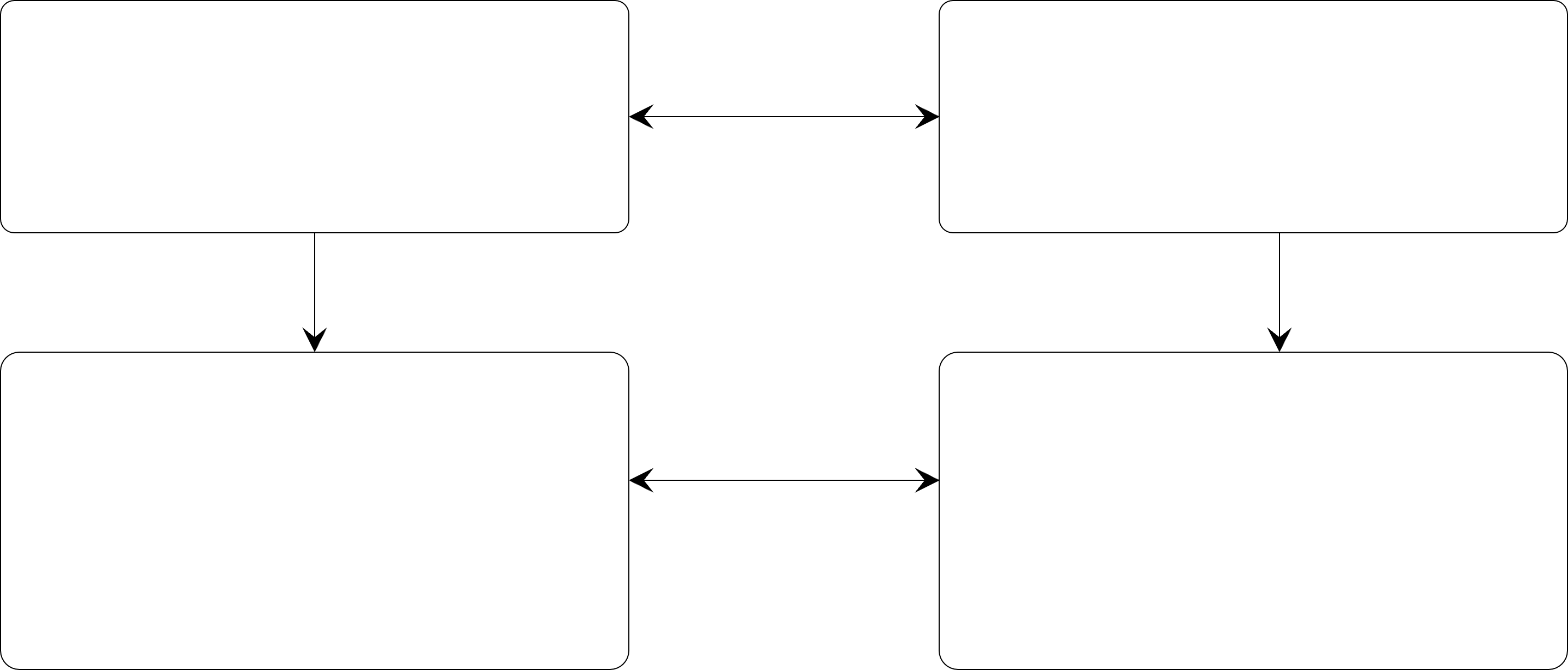}
\put(2,40.3){{\footnotesize \textbf{Stochastic Variational Principle Based on}}}
\put(11,37.8){{\footnotesize \textbf{Stochastic Potential}}}
\put(3,34){{\footnotesize
$
\delta\int_{0}^{S}\int_{a}^{b}
\left[
\ell(g,\zeta,\gamma)dt + \left<\mu,g^{-1}d_{t}g - \zeta dt\right>
\right.
$
}}
\put(3.5,30.5){{\footnotesize
$
\left.
+ \left<\sigma,g^{-1}\partial_{s}g - \gamma\right>dt
- \Phi(g,\mu)\circ dW
\right]ds = 0.
$
}}

\put(42,36.5){{\footnotesize Equivalence Cond.}}
\put(41.8,32.7){{\footnotesize
$
\Phi(g,\mu) = \left<\mu,\Theta\right> + C
$
}}

\put(63,40){{\footnotesize \textbf{Stochastic Variational Principle Based}}}
\put(68.3,37.5){{\footnotesize \textbf{on Stochastic Kinematics}}}
\put(69,33.5){{\footnotesize
$
\delta\int_{0}^{S}\int_{a}^{b}\ell(g,\zeta,\gamma)dtds = 0,
$
}}
\put(70.5,30.5){{\footnotesize
$
g^{-1}d_{t}g = \zeta dt + \Theta\circ dW.
$
}}

\put(4,17){{\footnotesize
$
d_t\mu =
\mathrm{ad}^{*}_{g^{-1}d_tg}\mu
- \partial_{s}\sigma dt
+ T^{*}_{e}L_{g}\frac{\delta\ell}{\delta g}dt
$
}}
\put(9,13){{\footnotesize
$
+ \mathrm{ad}^{*}_{g^{-1}\partial_{s}g}\sigma dt
- T_{e}^{*}L_{g}\frac{\delta\Phi}{\delta g}\circ dW,
$
}}
\put(4,9){{\footnotesize
$
d_t\gamma - \partial_{s}\zeta dt 
= 
\mathrm{ad}_{g^{-1}\partial_{s}g}g^{-1}d_tg
+ \partial_{s}\frac{\delta\Phi}{\delta\mu}\circ dW,
$
}}
\put(4,5.5){{\footnotesize
$
d_tg = g\zeta dt + g\frac{\delta\Phi}{\delta\mu}\circ dW,
\qquad
\partial_{s}g = g\gamma,
$
}}
\put(4,2){{\footnotesize
$
\mu = \delta\ell/\delta\zeta,
\qquad
\sigma = \delta\ell/\delta\gamma.
$
}}

\put(42,13){{\footnotesize Equivalence Cond.}}
\put(43,9.5){{\footnotesize
$
T_{e}^{*}L_{g}\frac{\delta\Phi}{\delta g}(g,\mu) = 0,
$
}}
\put(45,6){{\footnotesize
$
\frac{\delta\Phi}{\delta\mu}(g,\mu) = \Theta.
$
}}

\put(66.5,17){{\footnotesize
$
d_{t}\frac{\delta\ell}{\delta\zeta}
= 
\mathrm{ad}^{*}_{g^{-1}d_{t}g}\frac{\delta\ell}{\delta\zeta}
- \partial_{s}\frac{\delta\ell}{\delta\gamma}dt
$
}}
\put(74,12.5){{\footnotesize
$
+ \mathrm{ad}^{*}_{\gamma}\frac{\delta\ell}{\delta\gamma}dt
+ T^{*}_{e}L_{g}\frac{\delta\ell}{\delta g}dt,
$
}}
\put(66,9){{\footnotesize
$
d_{t}\gamma - \partial_{s}\zeta dt 
= 
\mathrm{ad}_{\gamma}g^{-1}d_{t}g
+ \partial_{s}\Theta\circ dW,
$
}}
\put(66,5.5){{\footnotesize
$
d_{t}g = g\zeta dt + g\Theta\circ dW,
$
}}
\put(66,2){{\footnotesize
$
\partial_{s}g = g\gamma.
$
}}

\put(10.2,24){{\footnotesize
Stochastic
}}
\put(20.5,24){{\footnotesize
Comp. Condition
}}

\put(66.2,24){{\footnotesize
Stochastic Comp.
}}
\put(82.2,24){{\footnotesize
Condition
}}
\end{overpic}
\caption{Variational formulations for stochastic geometrically exact beams.}
\label{continuous}
\end{figure}


\textbf{Goal of this paper}. This work focuses on developing a geometric framework for modeling and computing the dynamics of stochastic geometrically exact beams from a \textit{variational} perspective. Our approach is structured as follows:
\begin{itemize}
  \item[(a)] Introducing stochasticity into the beam system from two different \textit{variational} points of view: stochastic kinematics and stochastic potentials (see Fig.~\ref{continuous}). This variational framework helps to treat stochasticity and dynamics in a unified manner, rather than directly adding noise to the deterministic system.
  \item[(b)] Discretizing the two types of continuous-time variational principles to obtain the associated stochastic geometric integrators, instead of directly discretizing the continuous-time equations of motion. This helps to construct discrete-time dynamics that are physically more reliable.
  \item[(c)] Simulating the stochastic geometrically exact beam using the proposed stochastic numerical integrators. This helps to make it clearer the relation between the proposed methods and existing methods, such as the staggered leapfrog method and the well-known Yee's scheme in computational electromagnetism.
\end{itemize}
In this framework, the key idea is that stochastic perturbations are not directly added into the deterministic system, but are derived from stochastic variational principles, in both continuous- and discrete-time settings. As such, one can interpret the stochasticity from first principles and further analyze the properties of the stochastic beam system. Moreover, the proposed discrete stochastic variational formulations help to make the constructed integrators physically reliable, with structure-preserving properties, such as almost energy conservation, numerical momentum preservation, and Lie group structure preservation. Interestingly, several existing numerical methods, such as the staggered leapfrog method for the wave equation \citep{MoMa2005} and the well-known Yee's scheme for spatially one-dimensional computational electromagnetism \citep{Yee1966}, arise as special cases of the proposed stochastic integrators. These connections will be discussed in detail in the following sections.

\textbf{Organization of the paper}. In Section~\ref{Intro_Hamel}, we review the deterministic field-theoretic variational principle, which serves as the foundation for constructing stochastic variational principles throughout the paper. In Section~\ref{Continuous}, we present two types of field-theoretic stochastic variational principles for geometrically exact beams. The Type-I stochastic variational principle is based on the idea of stochastic kinematics, whereas the Type-II variational principle is based on a stochastic potential. It will be shown that these two variational principles are equivalent under certain conditions. In Section~\ref{Discrete}, we construct two types of stochastic variational integrators corresponding to the continuous-time variational formulations. Furthermore, we show that the two types of integrators are both stochastic extensions of the well-known staggered leapfrog method and Yee's scheme in computational electromagnetism. Finally, in Section~\ref{NR} we numerically demonstrate the effectiveness of the proposed integrators through examples involving the stochastic geometrically exact beam, the stochastic wave equation, and the stochastic flying spaghetti.

\section{Field-theoretic variational principle in Hamel's formalism}\label{Intro_Hamel}
This section presents an overview of the deterministic Hamel's formalism for mechanical field theories that involve a single spatial coordinate. A comprehensive exposition can be found in \cite{Shi_field,An2020}.

\paragraph{The moving frames.} This work studies the stochastic geometrically exact beam from a field-theoretic viewpoint. The physical fields are described by smooth sections of fiber bundles $\pi_{BE}:E\to B$, where the base manifold is spacetime $B = \mathbb{R}\times [0,S]$, with $S$ representing the spatial length of the system (for instance, the length of a geometrically exact beam). Consequently, the total space manifold is
\begin{align*}
E = B\times Q = \mathbb{R}\times [0,S]\times Q,
\end{align*}
where $Q$ is the fiber. The local sections of bundle $\pi_{BE}$ are defined as maps $\phi :U\to Q$, with $U\subset B$ being a subset of the base manifold with nonempty interior. Let $Y$ be an open subset of $Q$. The moving frames are defined to be a family of invertible linear operators
\begin{align*}
\Psi_{q}: V\longrightarrow T_{q}Q
\end{align*}
smoothly dependent on $q\in Y$ (making $Y$ be sufficiently small to guarantee the invertibility of $\Psi_{q}$, if necessary). Here $V$ is a vector space isomorphic to $T_{q}Q$. The dual map $\Psi_{q}^{*}: T_{q}^{*}Q\to V^{*}$ is defined by $\left<\Psi_q^* p,\xi \right> = \left< p,\Psi_q\xi \right>$, where $\left<\cdot,\cdot\right>$ on the right is the pairing between the vector space $T_{q}Q$ and its dual $T^{\ast}_{q}Q$. By construction, the map
\begin{align*}
Y\times V\ni (q,\xi)\mapsto (q,\Psi_q\xi)\in TQ
\end{align*}
is a fiber-preserving diffeomorphism that is linear in the second component. Hence, each $\xi\in V$ defines a smooth vector field $\Psi \xi := \bigcup_{q\in Y}\Psi_{q}\xi$ on $Y$. For any two vectors $\xi,\eta\in V$, an antisymmetric bilinear bracket operation $[\xi,\eta]_{q}$ on $V$ is defined by
\begin{align*}
\Psi_{q}\left[\xi,\eta\right]_{q} = \left[\Psi\xi,\Psi\eta\right]_{J-L}(q),
\end{align*}
where the bracket $\left[\cdot,\cdot\right]_{J-L}$ denotes the standard Jacobi-Lie bracket on $Q$. The Jacobi identity for the Jacobi-Lie bracket and the invertibility of $\Psi_q$ imply that $[\cdot,\cdot]_q$ satisfies the Jacobi identity on $V$. Hence, $(V,[\cdot,\cdot]_q)$ forms a Lie algebra for each $q\in Y$. For maps $\xi,\eta:U\to V$ and $\phi:U\to Y$, define the bracket operation
\begin{align*}
[\xi,\eta]_{\phi}(s,t):= [\xi(s,t),\eta(s,t)]_{\phi(s,t)},
\quad 
(s,t)\in U.
\end{align*}
The dual  bracket $[\xi,\alpha]_{\phi}^{*}$ is defined by $\langle [\xi,\alpha]_{\phi}^{*},\eta\rangle = \langle\alpha,[\xi,\eta]_{\phi}\rangle$, where $\alpha: U\rightarrow V^{*}$. Given a local section $\phi:U\to Y$, define a $V$-valued one-form $\xi := \Psi^{-1}\mathbf{d}\phi$. Equivalently, the form $\mathbf{d}\phi$ can be written as
\begin{align*}
\mathbf{d}\phi = \phi_s\,\mathrm{d}s + \phi_t\,\mathrm{d}t = \Psi_{\phi}\gamma\,\mathrm{d}s + \Psi_{\phi}\zeta\,\mathrm{d}t,
\end{align*}
where $\mathbf{d}$ is the exterior derivative, $\phi_{s}$ and $\phi_{t}$ are partial derivatives of \(\phi\) with respect to \(s\) and \(t\), respectively, and $\zeta$ and $\gamma$ are $V$-valued functions, i.e.,
\begin{align*}
\phi_s = \Psi_{\phi}\gamma,
\quad 
\phi_t = \Psi_{\phi}\zeta,
\quad 
\text{and}
\quad
\xi = \gamma\,\mathrm{d}s + \zeta\,\mathrm{d}t.
\end{align*}
Since variables $\gamma$ and $\zeta$ are derivatives measured using moving frames $\Psi_{q}$, they are
referred to as nonmaterial strain and nonmaterial velocity, respectively.

\paragraph{Hamilton's variational principle in Hamel's formalism.} Let $L:J^{1}E\to\mathbb{R}$ be a Lagrangian, where $J^{1}E$ is the first jet bundle, the space of all 1-jets of local sections of $E$. For a local section \(\phi\) of $\pi_{BE}: B\rightarrow E$, the local representations of the first jet extension reads $(\phi,\mathbf{d}\phi)$, and therefore the Lagrangian reads $L(\phi,\mathbf{d}\phi)$. Using moving frames $\Psi$ and treating $\xi$ as an independent variable, the Lagrangian reads
\begin{equation*}
\ell(\phi,\zeta,\gamma) = \ell(\phi,\xi) := L(\phi,\Psi\xi).
\end{equation*}
Based on this Lagrangian, one obtains the following variational principles for field-theoretic mechanical systems.

\begin{theorem}[\cite{Shi_field}]
\label{vp_hamel}
The following statements are equivalent:
	\begin{enumerate}
		\item[(i)] The map \(\phi\) extremizes the action functional
		\begin{align*}
		    \int_U L(\phi,\mathbf{d}\phi)\,\mathrm{d}s\,\mathrm{d}t,
		\end{align*}
		for all variations \(\delta\phi\) vanishing on \(\partial U\).
        
		\item[(ii)] The map \(\phi\) satisfies the Euler-Lagrange equations
		\begin{align*}
		  \frac{\partial}{\partial t}\frac{\partial L}{\partial\phi_t} + \frac{\partial}{\partial s}\frac{\partial L}{\partial\phi_s} - \frac{\partial L}{\partial\phi} = 0.
		\end{align*}
        
		\item[(iii)] The map \((\phi,\xi)\) extremizess the functional
		\begin{align*}
		  \int_U l(\phi,\xi)\,\mathrm{d}s\,\mathrm{d}t
		\end{align*}
		with respect to variations $\delta\xi$, induced by the variations $\delta\phi = \Psi_{\phi}\eta$, and given by $\delta\xi = \mathbf{d}\eta + [\xi,\eta]_{\phi}$, where $\eta: U\rightarrow V$ vanishes on $\partial U$. 

		\item[(iv)] The map $(\phi,\xi)$ satisfies field-theoretic Hamel's equations
		\begin{align*}
		\frac{\partial}{\partial t}\frac{\partial l}{\partial\zeta}
		+ \frac{\partial}{\partial s}\frac{\partial l}{\partial\gamma}
		- \left[\zeta,\frac{\partial l}{\partial\zeta}\right]_{\phi}^{*}
		- \left[\gamma,\frac{\partial l}{\partial\gamma}\right]_{\phi}^{*}
		- \Psi_{\phi}^{*}\frac{\partial l}{\partial\phi} = 0
		\end{align*}
		coupled with the kinematic equation $\mathbf{d}\phi = \Psi_{\phi}\xi$.
	\end{enumerate}
\end{theorem}

\paragraph{Compatibility condition.} 
It is noted that Theorem~\ref{vp_hamel} is not about solving either Euler-Lagrange or Hamel's equations. When seeking the solutions, an additional condition is required to ensure that the solution sets of field-theoretic Hamel's equations and the Euler-Lagrange equations are the same. Indeed, the kinematic equation  defines a generically nonintegrable distribution. Thus, to ensure the integrability of the kinematic equation, one has to use the compatibility condition
\begin{align*}
d\xi + [\xi, \xi]_{\phi} = 0,
\end{align*}
where $\xi$ is a $V$-valued one-form and \([\xi, \xi]_{\phi}\) is the two-form defined by $[\xi, \xi]_{\phi}(u_1, u_2) := [\xi u_1, \xi u_2]_{\phi}$.

\section{Stochastic variational principles for geometrically exact beams}\label{Continuous}
Based on the field-theoretic viewpoint above, the covariant configuration bundle of a geometrically exact beam is the fiber bundle $\pi_{E}: E = B\times SE(3)\rightarrow B$, where $B=\mathbb{R}\times[0,S]$ is the base manifold consisting of time and space, with $S$ representing the length of the undeformed beam. The configuration of the beam is the space of local sections of $\pi_{E}$, that is,
\begin{align}
\mathcal{H}=
\left\{g: B\rightarrow E \bigg|
g(t,s) =
\left(
  \begin{array}{cc}
    R(t,s) & p(t,s) \\
    0 & 1 \\
  \end{array}
\right)\in SE(3)
\right\},
\end{align}
where $R(t,s)$ and $p(t,s)$ represent rotation matrices and translation vectors, respectively. The Lagrangian of the geometrically exact beam is the kinetic energy minus the potential energy, that is,
\begin{equation}
\ell\left(g,\zeta,\gamma\right) 
= 
\frac{1}{2}
\left(
\left<K\zeta,\zeta\right> - \left<P\gamma,\gamma\right>
\right),
\label{beam_lag}
\end{equation}
where $\zeta,\gamma\in\mathfrak{se}(3)$ are the velocity and the strain of the beam measured by moving frames $T_{e}L_{g}: \mathfrak{se}(3)\rightarrow T_{g}SE(3)$, respectively. With these moving frames, the inertial coordinates $\left\{E_{1}, E_{2}, E_{3}\right\}$ are transformed to body frames $\left\{e_{1}(t,s), e_{2}(t,s), e_{3}(t,s)\right\}$ for $(t,s)\in B$ (see Fig.~\ref{beam-pic}). The matrices
\begin{equation*}
K = \mathrm{diag}\{J,M\}
\quad
\text{and}
\quad
P = \mathrm{diag}\{C_{1},C_{2}\},
\end{equation*}
with $M=\rho AI_{3\times 3}$, $J$ being the inertia matrix, $C_{1} = \mathrm{diag}\{EI,2GI,EI\}$, $C_{2} = \mathrm{diag}\{GA,EA,GA\}$,
where $\rho$ is the mass density of the beam, $I_{3\times 3}$ is the identity matrix, $I$ is the principal inertial axis of the cross section, $E$ is Young's modulus, $G=E/2(1+\nu)$ is the shear modulus, and $\nu$ is the Poisson's ratio of the beam. See \cite{Shi_field} for the derivation of the Lagrangian of the geometrically exact beam.
\begin{figure}[pos=t]
\centering
\begin{overpic}[width = 0.45\textwidth]{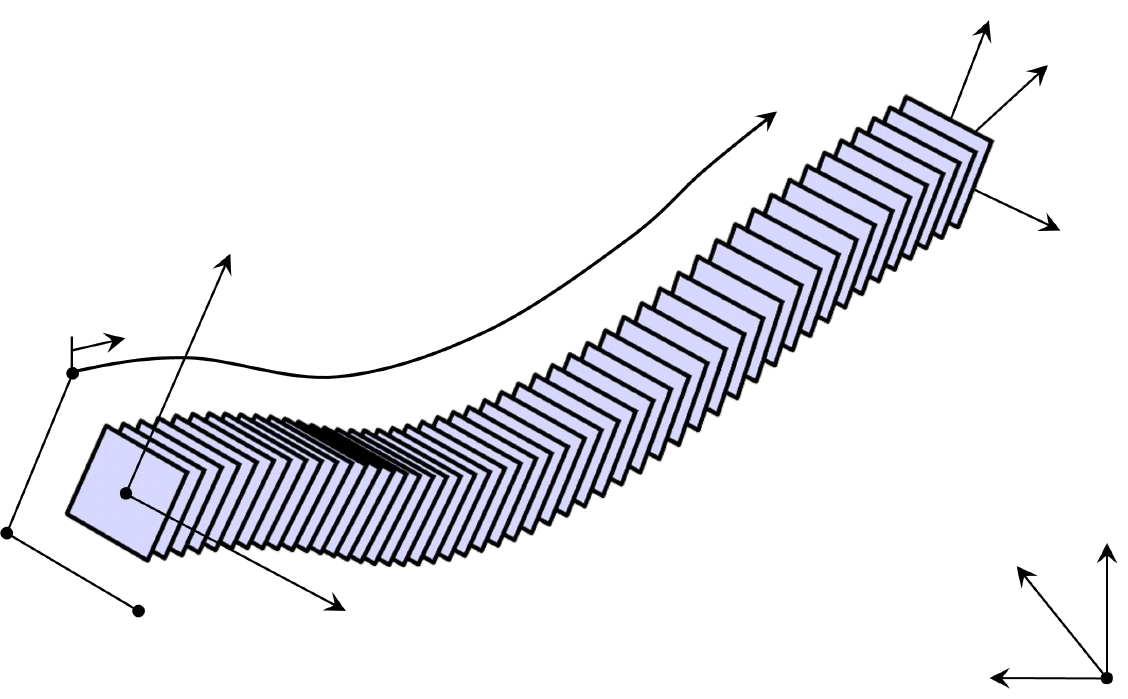}
\put(6,33){$s=0$}
\put(65,53){$s=S$}
\put(-1,22){$a$}
\put(3,7.5){$a$}
\put(85.2,13.8){$E_{1}$}
\put(82,1){$E_{2}$}
\put(95,15.5){$E_{3}$}
\put(28,4.5){$e_{1}(t,0)$}
\put(16,42){$e_{3}(t,0)$}
\put(81,62.5){$e_{3}(t,S)$}
\put(94,56){$e_{2}(t,S)$}
\put(93,38){$e_{1}(t,S)$}
\end{overpic}
\caption{Kinematics of a geometrically exact beam.}
\label{beam-pic}
\vspace{-1em}
\end{figure}


\subsection{Stochasticity introduced by kinematics}\label{Hamel_HP}
In practical applications, the beam dynamics are subject to uncertainties and fluctuations that are not captured by a purely deterministic model. The first way of introducing stochasticity to the beam system is to introduce stochastic perturbations directly into kinematic equations, so that the form of perturbations looks like what one encounters in Euclidean space. Recall that the deterministic kinematic equations of a beam read
\begin{equation}
g^{-1}\dot{g} = \zeta
\quad
\text{and}
\quad
g^{-1}g' = \gamma,
\label{det_kine}
\end{equation}
where $\dot{g}$ and $g'$ represent taking the temporal and spatial derivatives, respectively, and $\zeta$ and $\gamma$ are the convective velocity and constrain. Hence, the stochastic counterpart of the temporal kinematic equation reads
\begin{equation}
\label{tem_kin}
g^{-1}d_{t}g = \zeta dt + \Theta\circ dW,
\end{equation}
where $\Theta:B=\mathbb{R}\times[0,S]\rightarrow \mathfrak{se}(3)$ is a $\mathfrak{se}(3)$-valued vector that characterizes diffusivity of the stochastic process, the symbol $\circ$ implies that the integral is in the Stratonovich sense, and $W$ is a $L^{2}(0,S)$-valued Q-Wiener process. The spatial kinematic equation reads
\begin{equation}
\label{spa_kin}
g' = g\gamma.
\end{equation}
In order to present our variational approach, we make the following preparations. Taking the variation of the left-hand side of \eqref{tem_kin} yields
\begin{equation}
\begin{aligned}
\delta\left(g^{-1}d_{t}g\right) &= - g^{-1}\delta g g^{-1}d_{t}g + g^{-1}\delta\left(d_{t}g\right) \\
&= -\eta g^{-1}d_{t}g + d_{t}\left(g^{-1}\delta g\right) - d_{t}\left(g^{-1}\right)\delta g \\
&= d_{t}\eta 
+ \mathrm{ad}_{g^{-1}d_{t}g}\eta,
\label{lhs}
\end{aligned}
\end{equation}
where $\eta = g^{-1}\delta g$ and $\mathrm{ad}_{\xi}: \mathfrak{se}(3)\rightarrow\mathfrak{se}(3)$ is the adjoint action. The variation of the right-hand side of \eqref{tem_kin} reads
\begin{equation}
\begin{aligned}
\delta \left(\zeta dt + \Theta\circ dW\right)
= \delta\zeta dt.
\label{rhs}
\end{aligned}
\end{equation}
Combining \eqref{lhs} and \eqref{rhs}, we obtain the constrained variation formula
\begin{equation}
\begin{aligned}
d_{t}\eta
=
\delta\zeta dt
- \mathrm{ad}_{g^{-1}d_{t}g}\eta.
\label{tem}
\end{aligned}
\end{equation}
Similarly, equation \eqref{spa_kin} implies
\begin{equation}
\delta\gamma = \delta\left(g^{-1}g'\right) 
= \partial_{s}\eta + 
\mathrm{ad}_{\gamma}\eta,
\label{delta_gamma}
\end{equation}
where $\partial_{s} = \frac{\partial}{\partial s}$ represents taking the spatial derivative. Define the action functional associated to the Lagrangian \eqref{beam_lag} as
\begin{equation}
\label{action_functional}
\mathcal{A} = \int_{0}^{S}\int_{a}^{b}\ell(g,\zeta,\gamma)dtds,
\end{equation}
which is considered on the following space of curves 
\begin{equation*}
\mathcal{C} = \left\{(g,\zeta,\gamma)|g^{-1}d_{t}g = \zeta dt + \Theta\circ dW, \, g' = g\gamma\right\}.
\end{equation*}
Using the constrained variations \eqref{tem} and \eqref{delta_gamma}, we have
\begin{eqnarray*}
\delta\mathcal{A} &=& \delta\int_{0}^{S}\int_{a}^{b}\ell(g,\zeta,\gamma)dtds \\
&=& \int_{0}^{S}\int_{a}^{b}\left[\left<\frac{\delta\ell}{\delta g},\delta g\right> + \left<\frac{\delta\ell}{\delta\zeta},\delta\zeta\right> + \left<\frac{\delta\ell}{\delta\gamma},\delta\gamma\right>\right]dtds  \\
&=& \int_{0}^{S}\int_{a}^{b}\left[\left<T^{*}_{e}L_{g}\frac{\delta\ell}{\delta g}, \eta\right>dt + \left<\frac{\delta\ell}{\delta\zeta}, d_{t}\eta
+ \mathrm{ad}_{g^{-1}d_{t}g}\eta
\right> \right]ds
+ \int_{0}^{S}\int_{a}^{b}\left<\frac{\delta\ell}{\delta\gamma}, \partial_{s}\eta 
+ \mathrm{ad}_{\gamma}\eta
\right> dtds \\
&=&  \int_{0}^{S}\int_{a}^{b}\left<
-d_{t}\frac{\delta\ell}{\delta\zeta}
+ \mathrm{ad}^{\ast}_{g^{-1}d_{t}g}\frac{\delta\ell}{\delta\zeta}
- \frac{\partial}{\partial s}\frac{\delta\ell}{\delta\gamma}dt,
\eta\right> ds
+
\int_{0}^{S}\int_{a}^{b}\left<
\mathrm{ad}^{\ast}_{\gamma}\frac{\delta\ell}{\delta\gamma}
+ T^{*}_{e}L_{g}\frac{\delta\ell}{\delta g},
\eta
\right>dtds \\
&&\quad +
\underbrace{
\int_{0}^{S}\left<\frac{\delta\ell}{\delta\zeta},\eta\right>\bigg|_{t=a}^{t=b}ds
+ \int_{a}^{b}\left<\frac{\delta\ell}{\delta\gamma},\eta\right>\bigg|_{s=0}^{s=S}dt
}_{\text{boundary terms}}.
\end{eqnarray*}
By the arbitrariness of $\eta$, we obtain the stochastic equations of motion of the stochastic beam
\begin{equation*}
\begin{aligned}
d_{t}\frac{\delta\ell}{\delta\zeta}
- \mathrm{ad}^{*}_{g^{-1}d_{t}g}\frac{\delta\ell}{\delta\zeta}
+ \frac{\partial}{\partial s}\frac{\delta\ell}{\delta\gamma}dt
- \mathrm{ad}^{*}_{\gamma}\frac{\delta\ell}{\delta\gamma}dt
- T^{*}_{e}L_{g}\frac{\delta\ell}{\delta g}dt = 0,
\end{aligned}
\end{equation*}
together with the boundary condition
\begin{align*}
\frac{\delta\ell}{\delta\gamma}\bigg|_{s=0}
=
\frac{\delta\ell}{\delta\gamma}\bigg|_{s=S}
= 0.
\end{align*}
To look for solutions of equations of motion, we need to establish the stochastic counterpart of the compatibility condition. The idea is to use the commutativity of the temporal and spatial evolutions. Specifically, taking the partial derivative of \eqref{tem_kin} with respect to the spatial variable $s$ yields
\begin{equation}
-g^{-1}\partial_{s}gg^{-1}d_{t}g + g^{-1}\frac{\partial}{\partial s}d_{t}g
= \partial_{s}\zeta dt + \partial_{s}\Theta\circ dW.
\label{eq1}
\end{equation}
Similarly, taking the stochastic differential $d_{t}$ of \eqref{spa_kin} gives
\begin{equation}
-g^{-1}d_{t}gg^{-1}\partial_{s}g + g^{-1}d_{t}\frac{\partial}{\partial s}g = d_{t}\gamma.
\label{eq2}
\end{equation}
Since the partial derivative $\frac{\partial}{\partial s}$ and the stochastic derivative $d_{t}$ commute, the above \eqref{eq1} and \eqref{eq2} imply
\begin{equation*}
d_{t}\gamma - \partial_{s}\zeta dt
= \mathrm{ad}_{\gamma}g^{-1}d_{t}g
+ \partial_{s}\Theta\circ dW,
\end{equation*}
which is the \textit{stochastic compatibility condition}. Finally, the full dynamics of the stochastic geometrically exact beam reads
\begin{numcases}{}
\label{sto_Hamel}\text{\sc Stochastic field-theoretic equation:} &  \\
\hspace{0.8em}
d_{t}\frac{\delta\ell}{\delta\zeta}
= \mathrm{ad}^{*}_{g^{-1}d_{t}g}\frac{\delta\ell}{\delta\zeta}
- \frac{\partial}{\partial s}\frac{\delta\ell}{\delta\gamma}dt
+ \mathrm{ad}^{\ast}_{\gamma}\frac{\delta\ell}{\delta\gamma}dt
+ T^{*}_{e}L_{g}\frac{\delta\ell}{\delta g}dt,  & \notag \\
\label{sto_tem}\text{\sc Stochastic temporal kinematic equation:}&  \\
\hspace{0.8em} d_{t}g = g\left(\zeta dt + \Theta\circ dW\right),  & \notag \\
\label{sto_spa}\text{\sc Stochastic spatial kinematic equation:}&  \\
\hspace{0.8em} \partial_{s}g = g\gamma,  & \notag\\
\label{sto_com}\text{\sc Stochastic compatibility condition:}&  \\
\hspace{0.8em}
d_{t}\gamma - \partial_{s}\zeta dt = 
\mathrm{ad}_{\gamma}g^{-1}d_{t}g
+ \partial_{s}\Theta\circ dW.  & \notag
\end{numcases}
Equivalently, the above \eqref{sto_Hamel}-\eqref{sto_com} can be written in the following SPDE form
\begin{equation}
\begin{aligned}
d_{t}
\left(
  \begin{array}{c}
    g \\
    \frac{\delta\ell}{\delta\zeta} \\
    \gamma \\
  \end{array}
\right)
=
\left(
  \begin{array}{c}
    g\zeta \\
    \mathrm{ad}^{*}_{\zeta}\frac{\delta\ell}{\delta\zeta}
    + \mathrm{ad}^{*}_{\gamma}\frac{\delta\ell}{\delta\gamma}
    - \partial_{s}\frac{\delta\ell}{\delta\gamma} + T_{e}^{\ast}L_{g}\frac{\delta\ell}{\delta g} \\
    \partial_{s}\zeta 
    + \mathrm{ad}_{\gamma}\zeta  \\
  \end{array}
\right)dt
+
\left(
  \begin{array}{c}
    g\Theta \\
    \mathrm{ad}^{*}_{\Theta}\frac{\delta\ell}{\delta\zeta}  \\
    \partial_{s}\Theta 
    + \mathrm{ad}_{\gamma}\Theta  \\
  \end{array}
\right)\circ dW,
\end{aligned}
\end{equation}
together with the spatial kinematic equation $\partial_{s}g=g\gamma$.

The above variational formulation is motivated by geometrically exact beams, while in fact the proposed approach applies to stochastic field-theoretic mechanical systems $\pi_{BE}:B\times G\rightarrow B$ with Lie-group-valued fiber $G$. One representative example is the wave equation, as shown in the following.
\begin{exam}[Stochastic wave equation]
Consider the linear wave equation, with state  $\phi\in\mathbb{R}$, $t\in\mathbb{R}$ being the temporal variable and $s\in[0,S]$ being the spatial variable. The covariant configuration bundle is therefore $\pi_{BE}: B\times G\rightarrow B$, where the base space $B=\mathbb{R}\times[0,S]$ and the fiber $G=\mathbb{R}$. The Lagrangian of the wave motion is the kinetic energy minus the potential energy
\begin{equation}
\ell(g,\zeta,\gamma) = \frac{1}{2}\left(\zeta^{2} - \gamma^{2}\right).
\label{L_1}
\end{equation}
Similar to the geometrically exact beam, the fiber of this system is also a Lie group, but the trivial one $Q=\mathbb{R}$. By selecting moving frames to be the identity map $T_{e}L_{\phi} = \mathrm{Id}_{\mathbb{R}}$, the operators $\mathrm{ad}$ and $\mathrm{ad}^{*}$ are trivial, that is, $\mathrm{ad}_{\zeta} = 0$ and $\mathrm{ad}^{*}_{\zeta} = 0$. Hence, the stochastic equations of motion \eqref{sto_Hamel}-\eqref{sto_com} for the wave system read
\begin{equation}
\begin{aligned}
d_{t}\zeta &= \gamma_{s}dt, \\
d_{t}\phi &= \zeta dt + \Theta\circ dW,  \\
\phi_{s} &= \gamma,  \\
d_{t}\gamma &= \zeta_{s}dt + \Theta_{s}\circ dW,
\label{wave1}
\end{aligned}
\end{equation}
where $\Theta: \mathbb{R}\times[0,S]\rightarrow\mathbb{R}$ is the diffusivity vector field, and $(\cdot)_{s} = \partial/\partial s$ represents taking the spatial derivative. If $\Theta$ is chosen to be zero, the above equations \eqref{wave1} specialize to
\begin{align*}
\zeta_{t} = \gamma_{s}, 
\quad
\phi_{t} = \zeta, 
\quad
\phi_{s} = \gamma, 
\quad
\gamma_{t} = \zeta_{s},
\end{align*}
which recovers the deterministic wave equation $\left(-\frac{\partial^{2}}{\partial t^{2}} + \triangle\right)\phi = 0$ in \cite{evans2022partial}.
\end{exam}

\subsection{Stochasticity introduced by potential functions}\label{Hamel_HPH}
An alternative way of introducing the stochasticity is to use a stochastic potential function to construct a mixed variational principle. The first way of introducing stochasticity via kinematics provides us with the viewpoint of treating perturbations in a \textit{local} manner, which results from a Euclidean-like stochastic kinematic equation \eqref{tem_kin}. While the second way of introducing stochasticity in this section describes the equations of motion under the influence of a stochastic potential function. Specifically, we define the action functional as
\begin{equation}
\begin{aligned}
\label{action_fun}
\mathcal{A} &= \int_{0}^{S}\int_{a}^{b}
\left[
\ell(g,\zeta,\gamma)dt + \left<\mu,g^{-1}d_{t}g - \zeta dt\right>  
+ \left<\sigma,g^{-1}\partial_{s}g - \gamma\right>dt
- \Phi(g,\mu)\circ dW
\right]ds,
\end{aligned}
\end{equation}
where $W$ is an $L^{2}(0,S)$-valued Q-Wiener process, $\zeta,\gamma\in\mathfrak{se}(3)$ and $\mu,\sigma\in\mathfrak{se}^*(3)$ are multipliers that impose the stochastic kinematic constraints. The stochasticity is described by the stochastic potential $\Phi:SE(3)\times\mathfrak{se}^{*}(3)\rightarrow\mathbb{R}$. Taking the variation of \eqref{action_fun} yields
\begin{align*}
\delta\mathcal{A} &= \int_{0}^{S}\int_{a}^{b}\left[\left(\left<\frac{\delta\ell}{\delta g},\delta g\right> + \left<\frac{\delta\ell}{\delta \zeta},\delta \zeta\right> + \left<\frac{\delta\ell}{\delta \gamma},\delta \gamma\right>\right)dt 
+ \left<\delta\mu, g^{-1}d_tg - \zeta dt\right>
\right. \\
&\left. \qquad\qquad
+ \left<\mu, d_t\eta + \mathrm{ad}_{g^{-1}d_tg}\eta - \delta\zeta dt\right> 
+ \left<\delta\sigma, g^{-1}\partial_{s}g - \gamma \right>dt
+ \left<\sigma, \partial_{s}\eta + \mathrm{ad}_{g^{-1}\partial_{s}g}\eta - \delta\gamma \right>dt
\right.\\
&\qquad\qquad \left.
- \left(
\left<\frac{\delta\Phi}{\delta g}, \delta g\right>
+ \left<\frac{\delta\Phi}{\delta \mu}, \delta\mu\right>
\right)\circ dW\right]ds \\
&= \int_{0}^{S}\int_{a}^{b}
\left[
\left<
T^{\ast}_{e}L_{g}\frac{\delta\ell}{\delta g}
- d_{t}\mu
+ \mathrm{ad}^{*}_{g^{-1}d_{t}g}\mu
- \sigma_{s}
+ \mathrm{ad}^{*}_{g^{-1}\partial_{s}g}\sigma
,
\eta
\right>dt 
- \left<
T_{e}^{*}L_{g}\frac{\delta\Phi}{\delta g},
\eta\right>
\circ dW
\right. \\
&\qquad\qquad \left. + 
\left(
\left<
\frac{\delta\ell}{\delta\zeta}
- \mu,
\delta\zeta
\right>
+
\left<
\frac{\delta\ell}{\delta\gamma}
- \sigma,
\delta\gamma
\right>
\right)dt
+ 
\left<
\delta\mu,
g^{-1}d_{t}g - \zeta dt
\right>
- \left<\delta\mu, \frac{\delta\Phi}{\delta\mu}\right>\circ dW
\right] ds \\
&\quad  + \int_{0}^{S}\int_{a}^{b}
\left<
\delta\sigma,
g^{-1}\partial_{s}g - \gamma
\right>dtds 
+
\underbrace{
\int_{0}^{S}\left<\mu,\eta\right>\Big|_{t=a}^{t=b}ds
+ \int_{a}^{b}\left<\sigma,\eta\right>\Big|_{s=0}^{s=S}dt
}_{\text{boundary terms}}.
\end{align*}
By the boundary condition $\eta(a,s) = \eta(b,s) = 0$, for $s\in [0,S]$, and the arbitrariness of variations $\delta\zeta$, $\delta\gamma$, $\delta\mu$, and $\delta\sigma$, we obtain the stochastic equations of motion
\begin{numcases}{}
\text{\sc Stochastic implicit field-theoretic equation:} & \label{e8}\\
\hspace{0.8em} d_t\mu = \mathrm{ad}^{*}_{g^{-1}d_tg}\mu
- \partial_{s}\sigma dt
+ T^{*}_{e}L_{g}\frac{\delta\ell}{\delta g}dt
+ \mathrm{ad}^{*}_{g^{-1}\partial_{s}g}\sigma dt
- T_{e}^{*}L_{g}\frac{\delta\Phi}{\delta g}\circ dW, & \notag \\
\text{\sc Stochastic temporal kinematic equation:} & \label{e4}\\
\hspace{0.8em} d_tg = g\zeta dt + g\frac{\delta\Phi}{\delta\mu}\circ dW, & \notag\\
\text{\sc Stochastic spatial kinematic equation:} & \label{e5}\\
\hspace{0.8em} \partial_{s}g = g\gamma, & \notag \\
\text{\sc Stochastic temporal Legendre transformation:} & \label{e6}\\
\hspace{0.8em} \mu = \frac{\delta\ell}{\delta\zeta}, & \notag \\
\text{\sc Stochastic spatial Legendre transformation:} & \label{e7}\\
\hspace{0.8em} \sigma = \frac{\delta\ell}{\delta\gamma}. & \notag
\end{numcases}

In order to derive the stochastic compatibility condition, we take the spatial derivative and the temporal stochastic derivative of \eqref{e4} and \eqref{e5}, respectively, which yields
\begin{align}
\label{com1}
&- g^{-1}\partial_{s}gg^{-1}d_tg + g^{-1}\partial_{s}(d_tg)
- \partial_{s}\zeta dt
- \partial_{s}\left(\frac{\delta\Phi}{\delta\mu}\right)\circ dW = 0,\\
\label{com2}
&-g^{-1}d_{t}gg^{-1}\partial_{s}g + g^{-1}d_{t}(\partial_{s}g)
- d_{t}\gamma = 0.
\end{align}
Combining \eqref{com1} and \eqref{com2} gives the stochastic compatibility condition
\begin{eqnarray}
\label{e9}
d_t\gamma - \partial_{s}\zeta dt = \mathrm{ad}_{g^{-1}\partial_{s}g}g^{-1}d_tg
+ \partial_{s}\left(\frac{\delta\Phi}{\delta\mu}\right)\circ dW.
\end{eqnarray}
Equivalently, the above \eqref{e8}-\eqref{e7} and \eqref{e9} can be written in the following SPDE form
\begin{equation*}
\begin{aligned}
d_{t}
\left(
  \begin{array}{c}
    g \\
    \mu \\
    \gamma \\
  \end{array}
\right)
=
\left(
  \begin{array}{c}
    g\zeta \\
    \mathrm{ad}^{*}_{\zeta}\mu 
    + \mathrm{ad}^{*}_{\gamma}\sigma 
    - \partial_{s}\sigma 
    + T_{e}^{\ast}L_{g}\frac{\delta\ell}{\delta g} \\
    \partial_{s}\zeta + \mathrm{ad}_{\gamma}\zeta  \\
  \end{array}
\right)dt
+
\left(
  \begin{array}{c}
    g\frac{\delta\Phi}{\delta\mu} \\
    \mathrm{ad}^{*}_{\delta\Phi/\delta\mu}\mu
    - T_{e}^{\ast}L_{g}\frac{\delta\Phi}{\delta g} \\
    \mathrm{ad}_{\gamma}\frac{\delta\Phi}{\delta\mu}
    + \partial_{s}\left(\frac{\delta\Phi}{\delta\mu}\right) \\
  \end{array}
\right)\circ dW,
\end{aligned}
\end{equation*}
together with the intrinsic relations \eqref{e5}-\eqref{e7}.

\begin{rmk}
\label{equi_continuous}
The first and the second way of introducing the stochasticity are closely related. By comparing the coefficients, one observes that \eqref{sto_Hamel}-\eqref{sto_com} and \eqref{e8}-\eqref{e9} are equivalent if and only if
\begin{align}
\frac{\delta\Phi}{\delta\mu}(g,\mu) = \Theta
\qquad
\text{and}
\qquad
T_{e}^{*}L_{g}\frac{\delta\Phi}{\delta g}(g,\mu) = 0,
\label{equivalence}
\end{align}
which is equivalent to that the stochastic potential is of the following form
\begin{align*}
\Phi(g,\mu) = \left<\mu,\Theta\right> + C,
\label{equivalence}
\end{align*}
where $C\in\mathbb{R}$ is a constant.
\end{rmk}

\begin{exam}[Stochastic wave equation, continued]
\label{sto_wave_phi}
For the stochastic wave motion, the stochastic field-theoretic equations \eqref{e8}-\eqref{e7} and \eqref{e9} specialize to
\begin{equation}
\begin{aligned}
&d_{t}\mu = - \sigma_{s}dt - \frac{\delta\Phi}{\delta \phi}\circ dW, 
\quad
d_{t}\phi = \zeta dt + \frac{\delta\Phi}{\delta\mu}\circ dW, 
\quad
\phi_{s} = \gamma,\\
&\mu = \zeta, 
\quad
\sigma = -\gamma, 
\quad
d_{t}\gamma = \zeta_{s}dt + \partial_{s}\left(\frac{\delta\Phi}{\delta\mu}\right)\circ dW.
\label{wave_2_1}
\end{aligned}
\end{equation}
In particular, if the stochastic potential $\Phi = \phi$, the stochastic field-theoretic equations \eqref{wave_2_1} become, in components,
\begin{align*}
\mathrm{d}_{t}\zeta = \Delta\phi \mathrm{d}t - \mathrm{d}W, 
\quad
\mathrm{d}_{t}\phi = \zeta \mathrm{d}t, 
\quad
\mathrm{d}_{t}\phi_{s} = \zeta_{s}\mathrm{d}t,
\end{align*}
which recovers the stochastic linear wave equations with additive noise in \cite{HoSu2022}. Furthermore, if $\Phi$ is chosen to be zero, equations \eqref{wave_2_1} become
\begin{align*}
\mu_{t} = - \sigma_{s}, 
\quad
\phi_{t} = \zeta, 
\quad
\phi_{s} = \gamma,
\quad
\mu = \zeta, 
\quad
\sigma = -\gamma, 
\quad
\gamma_{t} = \zeta_{s},
\end{align*}
which recovers the deterministic wave equation \citep{evans2022partial}.
\end{exam}

\subsection{Energy analysis}\label{Energy_Analysis}
The stochastic potential (or the stochastic diffusivity vector field) influences the total energy of the system. In order to quantitatively give an energy analysis of the system, it is more convenient to reformulate in the following stochastic Hamiltonian or Lie-Poisson forms. Given a system $\left(\ell,\Phi,W\right)$ with dynamics  \eqref{e8}-\eqref{e7}, the Hamiltonian $h$ is related to the Lagrangian $\ell(\zeta,\gamma)$ by
\begin{equation*}
h(\mu,\gamma) = \left<\mu,\zeta\right> - \ell(\zeta,\gamma),
\end{equation*}
where we assume that $\ell$ is hyperregular and with $\mu= \frac{\delta \ell}{\delta \zeta}$. If the Lagrangian $\ell$ and the stochastic potential $\Phi$ are $G$-invariant, the equations of motion \eqref{e8} and the stochastic compatibility condition \eqref{e9} can be rewritten in the following stochastic Hamiltonian form
\begin{numcases}{}
d_t\mu =
\left(
\mathrm{ad}^{*}_{\frac{\delta h}{\delta\mu}}\mu
+ \partial_{s}\frac{\delta h}{\delta\gamma}
- \mathrm{ad}^{*}_{\gamma}\frac{\delta h}{\delta\gamma}
\right)dt
+ \mathrm{ad}^{*}_{\frac{\delta\Phi}{\delta\mu}}\mu\circ dW, & \label{sto_Hamel_Ham}  \\
d_t\gamma =
\left(
\partial_{s}\frac{\delta h}{\delta\mu}
+ \mathrm{ad}_{\gamma}\frac{\delta h}{\delta\mu}
\right)dt
+ \left(
\mathrm{ad}_{\gamma}\frac{\delta\Phi}{\delta\mu}
+ \partial_{s}\frac{\delta\Phi}{\delta\mu}
\right)\circ dW,  &  \label{sto_com_Ham}
\end{numcases}
which is a stochastic extension of the affine Lie-Poisson formulation for geometrically exact beam as given in \cite{EllisGayBalmazHolmPutkaradzeRatiu2010}.

\begin{prop}\label{ene_evo}
The following two statements are equivalent:
\begin{itemize}
  \item[(a)] The equations of motion are given by \eqref{sto_Hamel_Ham}-\eqref{sto_com_Ham} in the stochastic Hamiltonian form;
  \item[(b)] For any smooth function $f$, its dynamics can be given in the following \textit{Stratonovich-Lie-Poisson} form
    \begin{equation}
    d_{t}f = \left\{f,h\right\}dt + \left\{f,\Phi\right\}\circ dW,
    \label{SLP}
    \end{equation}
    where the affine Lie-Poisson bracket $\left\{\cdot,\cdot\right\}$ is given by Lemma~\ref{ChenShi2025}.
\end{itemize}
\end{prop}

\begin{proof}
Please see Appendix~\ref{app_proof}.
\end{proof}
In particular, taking $f=h$ in the above \eqref{SLP} yields
\begin{equation*}
d_{t}h 
= \left\{h,h\right\}dt + \left\{h,\Phi\right\}\circ dW  
= \left\{h,\Phi\right\}\circ dW, 
\end{equation*}
which implies the following corollary.
\begin{coro}[Energy conservation]\label{Coro_Energy}
If the Hamiltonian $h$ Poisson-commutes with the stochastic potential $\Phi$, that is, $\left\{h,\Phi\right\} = 0$, then the total energy of the stochastic system is a conserved quantity.
\end{coro}

\begin{rmk}\label{Choice_Phi}
There are some specific choices of $\Phi$ such that the energy is conserved. (a) A trivial situation occurs when $\Phi$ is a Casimir function \citep{marsden1999}, which Poisson-commutes with $\mu$ and $\gamma$ and therefore makes the equations of motion implicitly stochastic according to \eqref{SLP}. (b) Other interesting but nontrivial candidates of $\Phi$ that conserve energy include smooth functions of the Hamiltonian $h$ (that is, $\Phi = \Psi(h)$, where $\Psi$ is a differentiable function), then the affine Lie-Poisson bracket $\left\{h,\Phi\right\}
= D\Psi(h)\left\{h, h\right\}
= 0$, which implies that the total energy is conserved.
\end{rmk}

\section{Stochastic field-theoretic variational integrators}\label{Discrete}
In the above sections, we proposed two continuous-time variational principles for stochastic field-theoretic systems from a geometric perspective. In this section, we introduce two variational discretizations (Type-I \& II) associated to the proposed two continuous-time formulations. In addition, the relationship between these two discretizations will be investigated in detail. Stochastic field-theoretic variational integrators will be constructed for the examples of the stochastic geometrically exact beam, the stochastic wave equation, and the stochastic flying spaghetti.

\subsection{Discrete stochastic Hamilton's principle (Type-I integrator)}\label{d_ham}
The first variational discretization corresponds to the variational principle proposed in Section~\ref{Hamel_HP}, which results in the so-called Type-I integrator. Let $\Delta t>0$ and $\Delta s>0$ be the temporal and spatial discretization steps, respectively. To start the variational discretization, we first discretize the state variables. The discrete curve approximating $\left\{\left(g,\zeta,\gamma,\right)\right\}\subseteq G\times \mathfrak{g}\times \mathfrak{g}$ is 
\begin{equation*}
c_{d} =
\left\{    \left(g^{0}_{j},\zeta^{1/2}_{j},\gamma^{0}_{j+1/2},g^{1}_{j},\cdots,g^{N-1}_{j},\zeta^{N-1/2}_{j},\gamma^{N-1}_{j+1/2},g^{N}_{j}\right)
\right\},
\end{equation*}
where $g^{n}_{j}$ ($1\leq j \leq J$, $0\leq n \leq N$) is defined by $g^{n}_{j} = g(t_{n}, s_{j})$, and $\zeta^{n+1/2}_{j}$ ($1\leq j \leq J$, $0\leq n \leq N-1$) is the approximation of $\zeta$ at $\left(t_{n+1/2}, s_{j}\right) = \left(\frac{t_{n}+t_{n+1}}{2}, s_{j}\right)$, that is, $\zeta^{n+1/2}_{j} \approx \zeta\left(\frac{t_{n}+t_{n+1}}{2}, s_{j}\right)$. Similarly, $\gamma^{n}_{j+1/2}$ ($0\leq j \leq J-1$, $1\leq n \leq N$) is the approximation of $\gamma$ at $\left(t_{n},s_{j+1/2}\right) := \left(t_{n}, \frac{s_{j} + s_{j+1}}{2}\right)$, i.e., $\gamma^{n}_{j+1/2} \approx \gamma\left(t_{n}, \frac{s_{j} + s_{j+1}}{2}\right)$. The discrete temporal kinematic equation is given by 
\begin{equation}
g^{n+1}_{j} = g^{n}_{j}\tau
\left(
\Delta t\zeta^{n+1/2}_{j} + \Theta^{n+1/2}_{j}
\left(W^{(M)}_{n+1,j} - W^{(M)}_{n,j}\right)
\right),
\label{dis_kin_eq}
\end{equation}
where $\tau:\mathfrak{g}\rightarrow G$ is the retraction map, which is used to locally flatten the Lie group. The discrete spatial kinematic equation is given by
\begin{equation}
g^{n}_{j+1} = g^{n}_{j}\tau
\left(
\Delta s\gamma^{n}_{j+1/2}
\right).
\label{dis_spatial_equation}
\end{equation}
Define the discrete action sum that approximates the action functional \eqref{action_functional} by
\begin{equation}
\begin{aligned}
A_{d}(c_{d}) = \Delta t\Delta s\sum_{n=0}^{N-1}\sum_{j=0}^{J-1}
\ell
\left(g^{n}_{j},\zeta^{n+1/2}_{j}, \gamma^{n}_{j+1/2}
\right),
\label{discrete_act_sum_1}
\end{aligned}
\end{equation}
where 
$c_{d} =
\left\{
\left(g^{0}_{j},\zeta^{1/2}_{j},\gamma^{0}_{j+1/2},g^{1}_{j},\cdots,g^{N-1}_{j},\zeta^{N-1/2}_{j},\gamma^{N-1}_{j+1/2},g^{N}_{j}\right)
\right\}$ is a curve in space
\begin{equation*}
C_{d}(N,J) = \left(G\times\left(\mathfrak{g}\oplus \mathfrak{g}\times G\right)^{N}\right)^{J}
\end{equation*}
subject to constraints \eqref{dis_kin_eq}-\eqref{dis_spatial_equation}. In the above \eqref{dis_kin_eq}, $W^{(M)}$ approximates the $Q$-Wiener process $W$ by truncating the Karhunen-Lo\`{e}ve expansion \citep{loeve1978}
\begin{equation*}
W(t) = \sum_{k=1}^{\infty}\sqrt{\eta_{k}}e_{k}B_{k}(t),
\quad
t \geq 0,
\end{equation*}
where $\left\{B_{k}\right\}_{k\in\mathbb{N}_{+}}$ is a family of independent standard Brownian motions and $\left(e_{k}, \eta_{k}\right)$ are the eigenvectors and eigenvalues of $Q$, respectively, with $\left\{e_{k}\right\}_{k\in\mathbb{N}_{+}}$ being an orthonormal basis of $L^{2}(0,S)$. We follow the idea of truncation in \cite{Hong2023}, that is, 
\begin{equation*}
W^{(M)} (t):= \sum_{k=1}^{M}\sqrt{\eta_{k}}e_{k}B_{k}(t),
\end{equation*}
with $M$ being a positive integer. Defining the discrete variable
$W^{(M)}_{n,j} := W^{(M)}\left(t^{n},s_{j}\right)$ yields the increment
\begin{equation*}
W^{(M)}_{n+1,j} - W^{(M)}_{n,j}
=
\sqrt{\Delta t}\sum_{k=1}^{M}\sqrt{\eta_{k}}e_{k}(s_{j})\beta^{n}_{k},
\end{equation*}
where $\beta^{n}_{k} = \frac{B_{k}(t^{n+1}) - B_{k}(t^{n})}{\sqrt{\Delta t}}\sim N(0,1)$ is distributed according to the standard Gaussian distribution and $s_{j} = jS/J$. In this paper, we take the eigenpairs $\left\{\left(\eta_{k},e_{k}\right)\right\}_{k\in\mathbb{N}}$ to be
\begin{equation*}
\eta_{k} = k^{-(2r+1+\varepsilon)},
\qquad
e_{k}(s) = \sqrt{\frac{2}{S}}\sin\left(\frac{k\pi s}{S}\right),
\label{eigenpair}
\end{equation*}
where $s\in [0,S]$, $k\in\mathbb{N}$, $r>0$, and $\varepsilon\geq 0$.







Now we are ready to present variational discretizations of stochastic geometrically exact beams by the following steps:
\begin{itemize}
    \item[(a)] We first present an \textbf{\textit{exact variational integrator}}, which is obtained by deriving exact variation formulae.
    \item[(b)] In order to improve computational efficiency, we keep the temporal/spatial accuracy up to certain orders to propose two ways of obtaining the \textbf{\textit{approximate variational integrator}}: 
        \begin{itemize}
            \item[(i)] Truncation-then-variation: First truncate the exact variations, and then perform the variation to derive the corresponding approximate variational integrator.
            \item[(ii)] Variation-then-truncation: First perform the variations to obtain the exact discrete equations, and then truncate these equations to obtain an approximate variational integrator.
        \end{itemize}
    \item[(c)] Finally, we show that the obtained two approximate variational integrators in the above (i) and (ii) are in fact the same, and both agree with the exact variational integrator up to temporal order $\mathcal{O}\left(\Delta t\right)$ and spatial order $\mathcal{O}\left(\Delta s^{2}\right)$.
\end{itemize}
The idea is shown in the following Fig.~\ref{Idea_Chap_4}.

\begin{figure}[pos=h]
\centering
\vspace{1em}
\begin{overpic}[width = 0.85\textwidth]{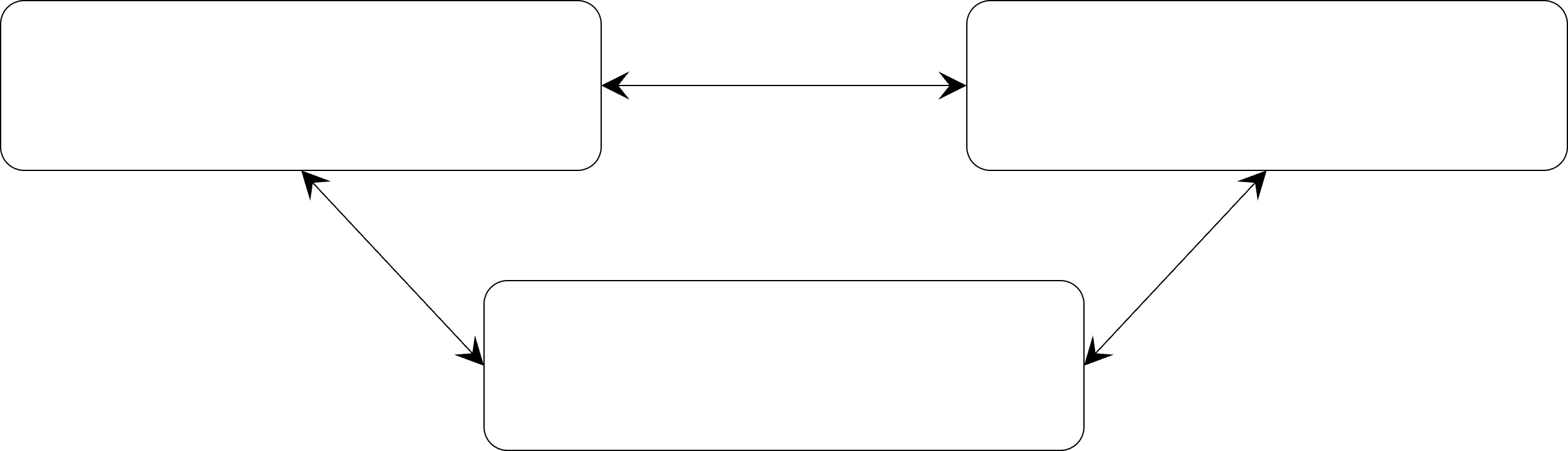}
\put(3,25.8){{\footnotesize \textbf{Approximate Variational Integrator}}}
\put(2.5,22.8){{\footnotesize obtained by truncation-then-variation}}
\put(12.5,19.8){{\footnotesize (Section~\ref{App_Int_Type1})}}
\put(47,24){Agree}
\put(49,17.5){$\circlearrowright$}
\put(28,14){equivalent up to accuracy order $\mathcal{O}\left(\Delta t + \Delta s^{2}\right)$}
\put(19.5,10){Equiv.}
\put(75.5,10){Equiv.}
\put(65,25.8){{\footnotesize \textbf{Approximate Variational Integrator}}}
\put(64.5,22.8){{\footnotesize obtained by variation-then-truncation}}
\put(74.5,19.8){{\footnotesize (Section~\ref{App_VI_VTT})}}
\put(32,8){{\footnotesize \textbf{Stochastic Exact Variational Integrator}}}
\put(33.5,5){{\footnotesize obtained by variation-then-truncation}}
\put(43.5,2){{\footnotesize (Section~\ref{Exact_VI_1})}}
\end{overpic}
\caption{The idea of constructing exact/approximate variational integrators.}
\label{Idea_Chap_4}
\vspace{-1em}
\end{figure}




\subsubsection{Stochastic exact variational integrator (Type-I)}\label{Exact_VI_1}

We first give the exact variation formulae for $\zeta_{j}^{n+1/2}$ and $\gamma_{j+1/2}^{n}$. Then, we use them to construct an exact variational integrator as illustrated in the bottom part of Fig.~\ref{Idea_Chap_4}.
\begin{prop}\label{exact_variations}
The exact variations $\delta\zeta^{n+1/2}_{j}$ and $\delta\gamma^{n}_{j+1/2}$ subject to constraints \eqref{dis_kin_eq} and \eqref{dis_spatial_equation} are given by
\begin{eqnarray}
\delta\zeta^{n+1/2}_{j} \Delta t 
&=& \delta\left(\tau^{-1}\left(\left(g^{n}_{j}\right)^{-1}g^{n+1}_{j}\right)\right)  
= \mathrm{d}_{\tau^{-1}\left(\left(g^{n}_{j}\right)^{-1}g^{n+1}_{j}\right)}\tau^{-1}
\cdot
\left(
-\eta^{n}_{j} + \mathrm{Ad}_{\left(g^{n}_{j}\right)^{-1}g^{n+1}_{j}}\eta^{n+1}_{j}
\right),  \label{var_zeta_2}\\
\delta\gamma^{n}_{j+1/2} \Delta s&=& 
\delta\left(\tau^{-1}\left(\left(g^{n}_{j}\right)^{-1}g^{n}_{j+1}\right)\right)
=
\mathrm{d}_{\tau^{-1}\left(\left(g^{n}_{j}\right)^{-1}g^{n}_{j+1}\right)}\tau^{-1}
\cdot
\left(
-\eta^{n}_{j} + \mathrm{Ad}_{\left(g^{n}_{j}\right)^{-1}g^{n}_{j+1}}\eta^{n}_{j+1}
\right),
\label{var_gamma_2}
\end{eqnarray}
where $\eta^{n}_{j} = \left(g^{n}_{j}\right)^{-1}\delta g^{n}_{j}$, 
$\mathrm{Ad}:G\times\mathfrak{g}\rightarrow\mathfrak{g}$ is the adjoint action, and $\mathrm{d}\tau^{-1}:\mathfrak{g}\times\mathfrak{g}\rightarrow\mathfrak{g}$ is the inverse right trivialized tangent \citep{Bou_Rabee_Marsden_2009}.
\end{prop}
\begin{proof}
Please see Appendix~\ref{App_var}.
\end{proof}

Then, we use the exact variation formulas established in Prop.~\ref{exact_variations} to derive the exact variational integrator below.

\begin{theorem}\label{theorem_exact_integrator}
The following statements are equivalent:
\begin{itemize}
    \item[(i)] The discrete curve
    \begin{equation*}
    c_{d} =
    \left\{
    \left(g^{0}_{j},\zeta^{1/2}_{j},\gamma^{0}_{j+1/2},g^{1}_{j},\cdots,g^{N-1}_{j},\zeta^{N-1/2}_{j},\gamma^{N-1}_{j+1/2},g^{N}_{j}\right)
    \right\}
    \end{equation*}
    in $C_{d}(N,J) = \left(G\times\left(\mathfrak{g}\oplus \mathfrak{g}\times G\right)^{N}\right)^{J}$ with kinematics constraints \eqref{dis_kin_eq}-\eqref{dis_spatial_equation} is a critical point of the stochastic action sum \eqref{discrete_act_sum_1}, that is, 
    $$\delta A_{d} = 0,$$
    for variations $\delta g^{n}_{j}$, $\delta\zeta^{n+1/2}_{j}$, and $\delta\gamma^{n}_{j+1/2}$ given by Prop.~\ref{exact_variations}
    and the fixed point condition $\delta g^{0}_{j} = \delta g^{N}_{j} = 0$, $j=1,\dots,J$.
    \item[(ii)] The discrete curve
    \begin{equation*}
    c_{d} =
    \left\{
    \left(g^{0}_{j},\zeta^{1/2}_{j},\gamma^{0}_{j+1/2},g^{1}_{j},\cdots,g^{N-1}_{j},\zeta^{N-1/2}_{j},\gamma^{N-1}_{j+1/2},g^{N}_{j}\right)
    \right\}
    \end{equation*}
    subject to constraints \eqref{dis_kin_eq}-\eqref{dis_spatial_equation} satisfies the discrete stochastic equations of motion
    \begin{eqnarray}
    \begin{aligned}
    &\Delta t\Delta s T_{e}^{\ast}L_{g^{n}_{j}}\frac{\delta\ell}{\delta g}\left(g^{n}_{j},\zeta^{n+1/2}_{j},\gamma^{n}_{j+1/2}\right)
    - \Delta s\left[\mathrm{d_{\tau^{-1}\left(\left(g^{n}_{j}\right)^{-1}g^{n+1}_{j}\right)}\tau^{-1}}\right]^{\ast}
    \frac{\delta\ell}{\delta\zeta}\left(g^{n}_{j},\zeta^{n+1/2}_{j},\gamma^{n}_{j+1/2}\right)  \\
    &\quad
    + \Delta s\mathrm{Ad}^{\ast}_{\left(g^{n-1}_{j}\right)^{-1}g^{n}_{j}}
    \left[\mathrm{d_{\tau^{-1}\left(\left(g^{n-1}_{j}\right)^{-1}g^{n}_{j}\right)}\tau^{-1}}\right]^{\ast}
    \frac{\delta\ell}{\delta\zeta}\left(g^{n-1}_{j},\zeta^{n-1/2}_{j},\gamma^{n-1}_{j+1/2}\right)  \\
    &\quad 
    + \Delta t\mathrm{Ad}^{\ast}_{\left(g^{n}_{j-1}\right)^{-1}g^{n}_{j}}
    \left[\mathrm{d_{\tau^{-1}\left(\left(g^{n}_{j-1}\right)^{-1}g^{n}_{j}\right)}\tau^{-1}}\right]^{\ast}
    \frac{\delta\ell}{\delta\gamma}\left(g^{n}_{j-1},\zeta^{n+1/2}_{j-1},\gamma^{n}_{j-1/2}\right)  \\
    &\quad
    - \Delta t\left[\mathrm{d_{\tau^{-1}\left(\left(g^{n}_{j}\right)^{-1}g^{n}_{j+1}\right)}\tau^{-1}}\right]^{\ast}
    \frac{\delta\ell}{\delta\gamma}\left(g^{n}_{j},\zeta^{n+1/2}_{j},\gamma^{n}_{j+1/2}\right) = 0,
    \label{without_truncation_exact}
    \end{aligned}
    \end{eqnarray}
    and the boundary conditions
    \begin{equation}
    \begin{aligned}    \left[\mathrm{d}_{\tau^{-1}\left(\left(g^{n}_{0}\right)^{-1}g^{n}_{1}\right)}\tau^{-1}\right]^{*}
    \frac{\delta\ell}{\delta\gamma^{n}_{1/2}} = 0
    \quad
    \text{and}
    \quad
    \mathrm{Ad}^{*}_{\left(g^{n}_{J-1}\right)^{-1}g^{n}_{J}}\left[\mathrm{d}_{\tau^{-1}\left(\left(g^{n}_{J-1}\right)^{-1}g^{n}_{J}\right)}\tau^{-1}\right]^{*}
    \frac{\delta\ell}{\delta\gamma^{n}_{J-1/2}} = 0.
    \label{Exact_bc}
    \end{aligned}
    \end{equation}
\end{itemize}
\end{theorem}

\begin{proof}
Please see Appendix~\ref{proof_exact_integrator}.
\end{proof}

\begin{rmk}
The obtained stochastic exact variational integrator in Theorem~\ref{theorem_exact_integrator} respects the discrete variational principle, and therefore preserves the discrete differential geometry structure of the system (such as momentum map and symplecticity as mentioned in \cite{MW_FGB_2026}). In practical implementations, we can achieve fast computation by constructing numerical integrators that agree with this exact variational integrator up to certain temporal/spatial accuracy orders. To this end, we propose the following \textit{approximate variational integrators}. 
\end{rmk}

\subsubsection{Approximate variational integrator (Type-I): Truncation-then-variation}\label{App_Int_Type1}

From the derivation process of the above exact variational integrator, we find that the computational complexity partially comes from the expression of discrete variations \eqref{var_zeta_2} and \eqref{var_gamma_2}, which contains the adjoint action map $\mathrm{Ad}$ and the inverse right trivialized tangent $\mathrm{d}\tau^{-1}$. Hence, here we propose an approximate variational integrator, which is obtained by approximating the variation formulae \eqref{var_zeta_2} and \eqref{var_gamma_2} up to temporal order $\mathcal{O}\left(\Delta t^{3/2}\right)$ and spatial order $\mathcal{O}\left(\Delta s^{2}\right)$, as illustrated in the upper-left part of Fig.~\ref{Idea_Chap_4}.

\begin{prop}\label{app_var_formulae}
The exact variation formulae \eqref{var_zeta_2} and \eqref{var_gamma_2} admit series expansions
\begin{equation}
\begin{aligned}
\delta \left[\tau^{-1}\left(\left(g^{n}_{j}\right)^{-1}g^{n+1}_{j}\right)\right]
&= \delta\zeta^{n+1/2}_{j}\Delta t 
= \eta^{n+1}_{j} - \eta^{n}_{j}
+ \mathrm{ad}_{\tau^{-1}\left(\left(g^{n}_{j}\right)^{-1}g^{n+1}_{j}\right)}\eta^{n+1/2}_{j}
+ \mathcal{O}\left(\Delta t^{3/2}\right),  \\
\delta \left[\tau^{-1}\left(\left(g^{n}_{j}\right)^{-1}g^{n}_{j+1}\right)\right]
&= \delta\gamma^{n}_{j+1/2}\Delta s 
= \eta^{n}_{j+1} - \eta^{n}_{j} + \mathrm{ad}_{\tau^{-1}\left(\left(g^{n}_{j}\right)^{-1}g^{n}_{j+1}\right)}\eta^{n}_{j+1/2}
+ \mathcal{O}\left(\Delta s^{2}\right).
\label{approximate_variation_for}
\end{aligned}
\end{equation}
\end{prop}

\begin{proof}
Please see Appendix~\ref{app_variation}.
\end{proof}

Hence, by keeping the temporal accuracy order up to $\mathcal{O}(\Delta t^{3/2})$ and the spatial accuracy order up to $\mathcal{O}(\Delta s^{2})$ in
\eqref{approximate_variation_for}, we obtain the approximate variation formulae
\begin{equation}
\begin{aligned}
\delta\zeta^{n+\frac{1}{2}}_{j}\Delta t = \eta^{n+1}_{j} - \eta^{n}_{j}
+ \mathrm{ad}_{\tau^{-1}\left(\left(g^{n}_{j}\right)^{-1}g^{n+1}_{j}\right)}\eta^{n+\frac{1}{2}}_{j},  
\qquad
\delta\gamma^{n}_{j+\frac{1}{2}}\Delta s = \eta^{n}_{j+1} - \eta^{n}_{j} + \mathrm{ad}_{\tau^{-1}\left(\left(g^{n}_{j}\right)^{-1}g^{n}_{j+1}\right)}\eta^{n}_{j+\frac{1}{2}},
\label{app_variation_formula_1}
\end{aligned}
\end{equation}
which induce the approximate variational integrator as the following theorem shows.

\begin{theorem}\label{Theorem_App_Int}
The approximate variation formulae \eqref{app_variation_formula_1} induce a stochastic variational integrator
\begin{equation}
\begin{aligned}
&\Delta t\Delta s T_{e}^{\ast}L_{g^{n}_{j}}\frac{\delta\ell}{\delta g}\left(g^{n}_{j},\zeta^{n+\frac{1}{2}}_{j},\gamma^{n}_{j+\frac{1}{2}}\right)
- \Delta s
\frac{\delta\ell}{\delta\zeta}\left(g^{n}_{j},\zeta^{n+\frac{1}{2}}_{j},\gamma^{n}_{j+\frac{1}{2}}\right)
+ \frac{\Delta s}{2}\mathrm{ad}^{\ast}_{\tau^{-1}\left(\left(g^{n}_{j}\right)^{-1}g^{n+1}_{j}\right)}
\frac{\delta\ell}{\delta\zeta}\left(g^{n}_{j},\zeta^{n+\frac{1}{2}}_{j},\gamma^{n}_{j+\frac{1}{2}}\right)\\
&\quad
+ \Delta s\frac{\delta\ell}{\delta\zeta}\left(g^{n-1}_{j},\zeta^{n-1/2}_{j},\gamma^{n-1}_{j+1/2}\right)
+ \frac{\Delta s}{2}\mathrm{ad}^{\ast}_{\tau^{-1}\left(\left(g^{n-1}_{j}\right)^{-1}g^{n}_{j}\right)}\frac{\delta\ell}{\delta\zeta}\left(g^{n-1}_{j},\zeta^{n-1/2}_{j},\gamma^{n-1}_{j+1/2}\right)  \\
&\quad
- \Delta t\frac{\delta\ell}{\delta\gamma}\left(g^{n}_{j},\zeta^{n+1/2}_{j},\gamma^{n}_{j+1/2}\right)
+ \frac{\Delta t}{2}\mathrm{ad}^{\ast}_{\tau^{-1}\left(\left(g^{n}_{j}\right)^{-1}g^{n}_{j+1}\right)}
\frac{\delta\ell}{\delta\gamma}\left(g^{n}_{j},\zeta^{n+1/2}_{j},\gamma^{n}_{j+1/2}\right)  \\
&\quad 
+ \Delta t\frac{\delta\ell}{\delta\gamma}\left(g^{n}_{j-1},\zeta^{n+1/2}_{j-1},\gamma^{n}_{j-1/2}\right)
+ \frac{\Delta t}{2}\mathrm{ad}^{\ast}_{\tau^{-1}\left(\left(g^{n}_{j-1}\right)^{-1}g^{n}_{j}\right)}\frac{\delta\ell}{\delta\gamma}\left(g^{n}_{j-1},\zeta^{n+1/2}_{j-1},\gamma^{n}_{j-1/2}\right)
= 0,
\label{App_Integrator_1}
\end{aligned}
\end{equation}
together with the discrete stochastic boundary condition
\begin{equation}
\begin{aligned}
&- \frac{\delta\ell}{\delta\gamma}
\left(
g^{n}_{0},\zeta^{n+1/2}_{0},\gamma^{n}_{1/2}
\right)
+
\frac{1}{2}
\mathrm{ad}^{*}_{
\tau^{-1}\left(\left(g^{n}_{0}\right)^{-1}g^{n}_{1}\right)
}
\frac{\delta\ell}{\delta\gamma}\left(g^{n}_{0},\zeta^{n+1/2}_{0},\gamma^{n}_{1/2}\right) = 0,  \\
&\frac{\delta\ell}{\delta\gamma}
\left(
g^{n}_{J-1},\zeta^{n+1/2}_{J-1},\gamma^{n}_{J-1/2}
\right)
+ \frac{1}{2}
\mathrm{ad}^{*}_{
\tau^{-1}\left(\left(g^{n}_{J-1}\right)^{-1}g^{n}_{J}\right)
}
\frac{\delta\ell}{\delta\gamma}\left(g^{n}_{J-1},\zeta^{n+1/2}_{J-1},\gamma^{n}_{J-1/2}\right)
= 0.
\label{App_boun_con_1}
\end{aligned}
\end{equation}
\end{theorem}

\begin{proof}
Please see Appendix~\ref{Proof_App_Int}.
\end{proof}
The advantage of the obtained approximate variational integrator is that it gives an elegant discretization scheme that looks similar to the continuous-time formulation \eqref{sto_Hamel}-\eqref{sto_com}, while maintaining the same performance as the exact integrator \eqref{without_truncation_exact} up to a certain accuracy order. Meanwhile, the approximate variational integrator avoids computationally complex operations $\mathrm{Ad}^{*}$ and $\mathrm{d}\tau^{-1}$, and also makes the compatibility condition easy to implement (see Section~\ref{Sec_comp_cond}). In some cases, the approximate integrator becomes the same as the exact integrator (see Example~\ref{Exam_wave}).

\subsubsection{Approximate variational integrator (Type-I): Variation-then-truncation}\label{App_VI_VTT}

The above approximate variational integrator in Theorem~\ref{Theorem_App_Int} results from first keeping the leading-order terms in the variational formulae and then taking variations in the discrete action sum. On the other hand, we can construct an approximate integrator in an opposite but direct way as illustrated in the upper-right part of Fig.~\ref{Idea_Chap_4}. This helps us figure out in which accuracy-order sense the approximate and exact integrators are equivalent. 

Consider maps in the exact variational integrator \eqref{without_truncation_exact}. First, the dual of the inverse trivialized tangent map and the coadjoint action admit series expansions \citep{Iserles2000lie,hairer2006,luesink2111casimir}
\begin{equation*}
\left[\mathrm{d}_{A}\tau^{-1}\right]^{\ast}
= I - \frac{1}{2}\mathrm{ad}^{\ast}_{A} + \mathcal{O}(|A|^{2})
\quad
\text{and}
\quad
\mathrm{Ad}^{\ast}_{\tau(B)} = I + \mathrm{ad}^{\ast}_{B} + \mathcal{O}(|B|^{2}).
\label{ser_exp}
\end{equation*}
Hence, in the exact variational integrator \eqref{without_truncation_exact} we have
\begin{align*}
\nonumber \left[\mathrm{d_{\tau^{-1}\left(\left(g^{n}_{j}\right)^{-1}g^{n+1}_{j}\right)}\tau^{-1}}\right]^{\ast}
= \left[\mathrm{d_{\Delta t\zeta^{n+1/2}_{j} + \Theta^{n+1/2}_{j}
\left(
W^{(M)}_{n+1,j} - W^{(M)}_{n,j}
\right)}
\tau^{-1}}\right]^{\ast}  
= I - \frac{1}{2}\mathrm{ad}^{\ast}_{\Delta t\zeta^{n+1/2}_{j} + \Theta^{n+1/2}_{j}
\left(
W^{(M)}_{n+1,j} - W^{(M)}_{n,j}
\right)}
+ \mathcal{O}(\Delta t).
\label{temp1}
\end{align*}
and
\begin{equation*}
\mathrm{Ad}^{\ast}_{\left(g^{n-1}_{j}\right)^{-1}g^{n}_{j}} 
= \mathrm{Ad}^{\ast}_{\tau\left(\Delta t\zeta^{n-1/2}_{j} + \Theta^{n-1/2}_{j}
\left(
W^{(M)}_{n,j} - W^{(M)}_{n-1,j}
\right)\right)}  
= I + \mathrm{ad}^{\ast}_{\Delta t\zeta^{n-1/2}_{j} + \Theta^{n-1/2}_{j}
\left(
W^{(M)}_{n,j} - W^{(M)}_{n-1,j}
\right)}
+ \mathcal{O}\left(\Delta t\right),
\label{temp2}
\end{equation*}
which implies that
\begin{eqnarray}
\mathrm{Ad}^{\ast}_{\left(g^{n-1}_{j}\right)^{-1}g^{n}_{j}}
\left[\mathrm{d_{\tau^{-1}\left(\left(g^{n-1}_{j}\right)^{-1}g^{n}_{j}\right)}\tau^{-1}}\right]^{\ast}
&=& I + \frac{1}{2}\mathrm{ad}^{\ast}_{\Delta t\zeta^{n-1/2}_{j} + \Theta^{n-1/2}_{j}
\left(
W^{(M)}_{n,j} - W^{(M)}_{n-1,j}
\right)}
+ \mathcal{O}\left(\Delta t\right).
\label{h1}
\end{eqnarray}
Similarly, for the spatial part we have
\begin{eqnarray}
\mathrm{Ad}^{\ast}_{\left(g^{n}_{j-1}\right)^{-1}g^{n}_{j}}
\left[\mathrm{d_{\tau^{-1}\left(\left(g^{n}_{j-1}\right)^{-1}g^{n}_{j}\right)}\tau^{-1}}\right]^{\ast}
=
I + \frac{1}{2}\mathrm{ad}^{\ast}_{\Delta s\gamma^{n}_{j-1/2}}
+ \mathcal{O}\left(\Delta s^{2}\right).
\label{h2}
\end{eqnarray}
Substituting \eqref{h1} and \eqref{h2} into \eqref{without_truncation_exact}, the exact variational integrator becomes
\begin{equation}
\begin{aligned}
&\Delta t\Delta s T_{e}^{\ast}L_{g^{n}_{j}}\frac{\delta\ell}{\delta g}\left(g^{n}_{j},\zeta^{n+\frac{1}{2}}_{j},\gamma^{n}_{j+\frac{1}{2}}\right)
- \Delta s
\frac{\delta\ell}{\delta\zeta}\left(g^{n}_{j},\zeta^{n+\frac{1}{2}}_{j},\gamma^{n}_{j+\frac{1}{2}}\right)
+ \frac{\Delta s}{2}\mathrm{ad}^{\ast}_{\tau^{-1}\left(\left(g^{n}_{j}\right)^{-1}g^{n+1}_{j}\right)}
\frac{\delta\ell}{\delta\zeta}\left(g^{n}_{j},\zeta^{n+\frac{1}{2}}_{j},\gamma^{n}_{j+\frac{1}{2}}\right)\\
&\quad
+ \Delta s\frac{\delta\ell}{\delta\zeta}\left(g^{n-1}_{j},\zeta^{n-1/2}_{j},\gamma^{n-1}_{j+1/2}\right)
+ \frac{\Delta s}{2}\mathrm{ad}^{\ast}_{\tau^{-1}\left(\left(g^{n-1}_{j}\right)^{-1}g^{n}_{j}\right)}\frac{\delta\ell}{\delta\zeta}\left(g^{n-1}_{j},\zeta^{n-1/2}_{j},\gamma^{n-1}_{j+1/2}\right)  \\
&\quad
- \Delta t\frac{\delta\ell}{\delta\gamma}\left(g^{n}_{j},\zeta^{n+1/2}_{j},\gamma^{n}_{j+1/2}\right)
+ \frac{\Delta t}{2}\mathrm{ad}^{\ast}_{\tau^{-1}\left(\left(g^{n}_{j}\right)^{-1}g^{n}_{j+1}\right)}
\frac{\delta\ell}{\delta\gamma}\left(g^{n}_{j},\zeta^{n+1/2}_{j},\gamma^{n}_{j+1/2}\right)  \\
&\quad 
+ \Delta t\frac{\delta\ell}{\delta\gamma}\left(g^{n}_{j-1},\zeta^{n+1/2}_{j-1},\gamma^{n}_{j-1/2}\right)
+ \frac{\Delta t}{2}\mathrm{ad}^{\ast}_{\tau^{-1}\left(\left(g^{n}_{j-1}\right)^{-1}g^{n}_{j}\right)}\frac{\delta\ell}{\delta\gamma}\left(g^{n}_{j-1},\zeta^{n+1/2}_{j-1},\gamma^{n}_{j-1/2}\right)
+ \mathcal{O}\left(\Delta t + \Delta s^{2}\right)
= 0.
\label{var_then_trun_app_vi}
\end{aligned}
\end{equation}
Keeping all terms except for the last term gives us an approximate integrator, which fully agrees with the approximate integrator \eqref{App_Integrator_1} obtained from the approximate variation formulae. Moreover, it is noteworthy that \eqref{var_then_trun_app_vi} implies that the approximate variational integrator \eqref{App_Integrator_1} is equivalent to the exact variational integrator up to the accuracy order $\mathcal{O}\left(\Delta t + \Delta s^{2}\right)$ (as shown in Fig.~\ref{Idea_Chap_4}).

\subsection{Discrete stochastic compatibility condition}\label{Sec_comp_cond}
As in the continuous-time formulation, to get a complete system one supplements the equations of motion with the compatibility condition. In this subsection, we first give a numerically \textit{exact discrete compatibility condition}, and then keep the leading-order terms in the series expansions up to a certain accuracy order to obtain an \textit{approximate compatibility condition}, which is numerically efficient to implement.
\subsubsection{Exact discrete stochastic compatibility condition}
The discrete counterpart of the stochastic compatibility condition \eqref{e9} comes from the idea of the continuous setting. Specifically, the spacetime transition $g^{n}_{j}\longrightarrow g^{n+1}_{j+1}$ could be carried out from either the space-then-time path
$$g^{n}_{j}\overset{\Delta s}{\longrightarrow} g^{n}_{j+1}\overset{\Delta t}{\longrightarrow}  g^{n+1}_{j+1}$$
or the time-then-space path
$$g^{n}_{j}\overset{\Delta t}{\longrightarrow}  g^{n+1}_{j}\overset{\Delta s}{\longrightarrow}  g^{n+1}_{j+1}.$$
That is, the following diagram commutes.
$$
\begin{tikzcd}[color=black]
	{g^{n}_{j}} & {g^{n}_{j+1}} \\
	{g^{n+1}_{j}} & {g^{n+1}_{j+1}}
	\arrow[from=1-1, to=1-2]
	\arrow[from=1-1, to=2-1]
	\arrow[from=1-2, to=2-2]
	\arrow[from=2-1, to=2-2]
\end{tikzcd}
$$
Explicitly, the diagram means that the following two equations
\begin{align*}
g^{n+1}_{j+1} = g^{n}_{j}
\tau\left(\Delta s\gamma^{n}_{j+1/2}\right)
\tau
\left(
\Delta t\zeta^{n+1/2}_{j+1} + \Theta^{n+1/2}_{j+1}
\left(W^{(M)}_{n+1,j+1} - W^{(M)}_{n,j+1}\right)
\right)
\end{align*}
and
\begin{eqnarray}
\begin{aligned}
g^{n+1}_{j+1} = g^{n}_{j}
\tau
\left(
\Delta t\zeta^{n+1/2}_{j} + \Theta^{n+1/2}_{j}
\left(W^{(M)}_{n+1,j} - W^{(M)}_{n,j}\right)
\right)
\tau\left(\Delta s\gamma^{n+1}_{j+1/2}\right)
\end{aligned}
\end{eqnarray}
are equal, which is equivalent to 
\begin{equation}
\begin{aligned}
&\quad\tau\left(\Delta s\gamma^{n}_{j+1/2}\right)
\tau
\left(
\Delta t\zeta^{n+1/2}_{j+1} + \Theta^{n+1/2}_{j+1}
\left(W^{(M)}_{n+1,j+1} - W^{(M)}_{n,j+1}\right)
\right)  \\
&=
\tau
\left(
\Delta t\zeta^{n+1/2}_{j} + \Theta^{n+1/2}_{j}
\left(W^{(M)}_{n+1,j} - W^{(M)}_{n,j}\right)
\right)
\tau\left(\Delta s\gamma^{n+1}_{j+1/2}\right).
\label{commu}
\end{aligned}
\end{equation}

\subsubsection{Approximate discrete compatibility condition}
The compatibility condition \eqref{commu} is numerically exact. However, it is computationally hard to implement in simulations due to the presence of the retraction map $\tau$, making equation \eqref{commu} live in the curved Lie group space and difficult to solve. Thus, we propose an approximate discrete compatibility condition without compromising the numerical accuracy up to $\mathcal{O}\left(\Delta t + \Delta s^{2}\right)$ as the following shows. 

Recall the Baker-Campbell-Hausdorff (BCH) formula \citep{Michor2008,hall2015lie}
$$
\tau(A)\tau(B) = 
\tau
\left(
A + B + \frac{1}{2}\left[A,B\right] 
+ \frac{1}{12}\left(\left[A,\left[A,B\right]\right] - \left[B,\left[B,A\right]\right]\right)
+ \cdots
\right),
$$
where $A$, $B\in\mathfrak{g}$ and $\tau$ is the exponential map. Then, using the BCH formula the left-hand-side of \eqref{commu} is
\begin{equation}
\begin{aligned}
&\quad\tau\left(\Delta s\gamma^{n}_{j+1/2}\right)
\tau
\left(
\Delta t\zeta^{n+1/2}_{j+1} + \Theta^{n+1/2}_{j+1}
\left(W^{(M)}_{n+1,j+1} - W^{(M)}_{n,j+1}\right)
\right)  \\ 
&= \tau
\left(
\Delta s\gamma^{n}_{j+1/2}
+
\Delta t\zeta^{n+1/2}_{j+1} + \Theta^{n+1/2}_{j+1}
\left(W^{(M)}_{n+1,j+1} - W^{(M)}_{n,j+1}\right)  \right.\\
&\qquad\quad
+ \frac{1}{2}
\left[\Delta s\gamma^{n}_{j+1/2},
\Delta t\zeta^{n+1/2}_{j+1} + \Theta^{n+1/2}_{j+1}
\left(W^{(M)}_{n+1,j+1} - W^{(M)}_{n,j+1}\right)
\right]  \\
&\qquad\quad\left.
+ \mathcal{O}\left(\Delta s\Delta t^{1/2}\right)
\right).
\label{lhs-com}
\end{aligned}
\end{equation}
Similarly, the right-hand side of \eqref{commu} is
\begin{equation}
\begin{aligned}
&\quad\tau
\left(
\Delta t\zeta^{n+1/2}_{j} + \Theta^{n+1/2}_{j}
\left(W^{(M)}_{n+1,j} - W^{(M)}_{n,j}\right)
\right)
\tau\left(\Delta s\gamma^{n+1}_{j+1/2}\right)  \\
&= \tau
\left(
\Delta t\zeta^{n+1/2}_{j} + \Theta^{n+1/2}_{j}
\left(W^{(M)}_{n+1,j} - W^{(M)}_{n,j}\right)
+ \Delta s\gamma^{n+1}_{j+1/2} \right.\\
&\qquad\quad
+ \frac{1}{2}
\left[
\Delta t\zeta^{n+1/2}_{j} + \Theta^{n+1/2}_{j}
\left(W^{(M)}_{n+1,j} - W^{(M)}_{n,j}\right),
\Delta s\gamma^{n+1}_{j+1/2}
\right]  \\
&\qquad\quad\left.
+ \mathcal{O}\left(\Delta s\Delta t^{1/2}\right)
\right).
\label{rhs-com}
\end{aligned}
\end{equation}
Substituting \eqref{lhs-com}-\eqref{rhs-com} into \eqref{commu}, the numerically exact stochastic compatibility condition \eqref{commu} is equivalent to
\begin{equation}
\begin{aligned}
&\Delta s\left(\gamma^{n+1}_{j+1/2} - \gamma^{n}_{j+1/2}\right)
- \Delta t\left(\zeta^{n+1/2}_{j+1} - \zeta^{n+1/2}_{j}\right) 
+ \Theta^{n+1/2}_{j}
\left(W^{(M)}_{n+1,j} - W^{(M)}_{n,j}\right)
-  \Theta^{n+1/2}_{j+1}
\left(W^{(M)}_{n+1,j+1} - W^{(M)}_{n,j+1}\right)  \\
&= \frac{1}{2}
\left[
\Delta s\gamma^{n+1}_{j+1/2},
\Delta t\zeta^{n+1/2}_{j} + \Theta^{n+1/2}_{j}
\left(W^{(M)}_{n+1,j} - W^{(M)}_{n,j}\right)
\right]  \\
&\qquad
+ \frac{1}{2}
\left[
\Delta s\gamma^{n}_{j+1/2},
\Delta t\zeta^{n+1/2}_{j+1} + \Theta^{n+1/2}_{j+1}
\left(W^{(M)}_{n+1,j+1} - W^{(M)}_{n,j+1}\right)
\right]
+ \mathcal{O}\left(\Delta s\Delta t^{1/2}\right).
\label{com_cond_50}
\end{aligned}
\end{equation}
By Young's inequality, we have
\begin{align*}
\Delta s\Delta t^{1/2}
\leq
\frac{\Delta s^{2}}{2} + \frac{\left(\Delta t^{1/2}\right)^{2}}{2}  
= \frac{1}{2}\left(\Delta s^{2} + \Delta t\right),
\end{align*}
which implies that the numerically exact stochastic compatibility condition \eqref{com_cond_50} reads
\begin{equation}
\begin{aligned}
&\Delta s\left(\gamma^{n+1}_{j+1/2} - \gamma^{n}_{j+1/2}\right)
- \Delta t\left(\zeta^{n+1/2}_{j+1} - \zeta^{n+1/2}_{j}\right) 
+ \Theta^{n+1/2}_{j}
\left(W^{(M)}_{n+1,j} - W^{(M)}_{n,j}\right)
-  \Theta^{n+1/2}_{j+1}
\left(W^{(M)}_{n+1,j+1} - W^{(M)}_{n,j+1}\right)  \\
&= \frac{1}{2}
\left[
\Delta s\gamma^{n+1}_{j+1/2},
\Delta t\zeta^{n+1/2}_{j} + \Theta^{n+1/2}_{j}
\left(W^{(M)}_{n+1,j} - W^{(M)}_{n,j}\right)
\right]  \\
&\qquad
+ \frac{1}{2}
\left[
\Delta s\gamma^{n}_{j+1/2},
\Delta t\zeta^{n+1/2}_{j+1} + \Theta^{n+1/2}_{j+1}
\left(W^{(M)}_{n+1,j+1} - W^{(M)}_{n,j+1}\right)
\right]
+ \mathcal{O}\left(\Delta s^{2} + \Delta t\right).
\label{com_exact}
\end{aligned}
\end{equation}
Hence, by \eqref{com_exact} the approximate discrete stochastic compatibility condition is
\begin{equation}
\begin{aligned}
&\Delta s\left(\gamma^{n+1}_{j+1/2} - \gamma^{n}_{j+1/2}\right)
- \Delta t\left(\zeta^{n+1/2}_{j+1} - \zeta^{n+1/2}_{j}\right) 
+ \Theta^{n+1/2}_{j}
\left(W^{(M)}_{n+1,j} - W^{(M)}_{n,j}\right)
-  \Theta^{n+1/2}_{j+1}
\left(W^{(M)}_{n+1,j+1} - W^{(M)}_{n,j+1}\right)  \\
&= \frac{1}{2}
\left[
\Delta s\gamma^{n+1}_{j+1/2},
\Delta t\zeta^{n+1/2}_{j} + \Theta^{n+1/2}_{j}
\left(W^{(M)}_{n+1,j} - W^{(M)}_{n,j}\right)
\right]  \\
&\qquad
+ \frac{1}{2}
\left[
\Delta s\gamma^{n}_{j+1/2},
\Delta t\zeta^{n+1/2}_{j+1} + \Theta^{n+1/2}_{j+1}
\left(W^{(M)}_{n+1,j+1} - W^{(M)}_{n,j+1}\right)
\right].
\label{d_sto_compatibility}
\end{aligned}
\end{equation}
In practical numerical implementations, we will use this approximate compatibility condition \eqref{d_sto_compatibility}. This simplifies the resulting algorithm and its implementation without affecting the quality of the algorithm up to the spatial accuracy order $\mathcal{O}\left(\Delta s^{2}\right)$ and temporal accuracy order $\mathcal{O}\left(\Delta t\right)$.

\begin{exam}[Stochastic wave equation, continued]\label{Exam_wave}
For the stochastic wave motion, the discrete action sum associated to \eqref{L_1} is
\begin{equation*}
A_{d} = 
\sum_{n=0}^{N-1}\sum_{j=0}^{J-1}
\frac{\Delta t\Delta s}{2}
\left[
\left(\zeta^{n+1/2}_{j}\right)^{2}
-
\left(\gamma^{n}_{j+1/2}\right)^{2}
\right].
\end{equation*}
Then, the Type-I stochastic integrator given by \eqref{dis_kin_eq}, \eqref{dis_spatial_equation}, \eqref{App_Integrator_1}, and \eqref{d_sto_compatibility} reads
\begin{equation}
\begin{aligned}
&
\Delta s\left(\zeta^{n-1/2}_{j} - \zeta^{n+1/2}_{j}\right)
-
\Delta t\left(\gamma^{n}_{j-1/2} - \gamma^{n}_{j+1/2}\right)
= 0, \\
&
\Delta s\left(\gamma^{n+1}_{j+\frac{1}{2}} - \gamma^{n}_{j+\frac{1}{2}}\right)
- \Delta t\left(\zeta^{n+\frac{1}{2}}_{j+1} - \zeta^{n+\frac{1}{2}}_{j}\right) 
+ \Theta^{n+\frac{1}{2}}_{j}
\left(W^{(M)}_{n+1,j} - W^{(M)}_{n,j}\right)
-  \Theta^{n+\frac{1}{2}}_{j+1}
\left(W^{(M)}_{n+1,j+1} - W^{(M)}_{n,j+1}\right)
= 0, \\
&
\phi^{n+1}_{j} - \phi^{n}_{j}
-
\Delta t \zeta^{n+1/2}_{j}
-
\Theta^{n+1/2}_{j}
\left(
W^{(M)}_{n+1,j} - W^{(M)}_{n,j}
\right)
= 0, \\
&\phi^{n}_{j+1} - \phi^{n}_{j}
-
\Delta s\gamma^{n}_{j+1/2}
= 0, \\
& \gamma^{n}_{1/2} = \gamma^{n}_{J+1/2} = 0,
\label{e11}
\end{aligned}
\end{equation}
which is a stochastic deformation of the deterministic staggered leapfrog scheme for the wave equation \citep{MoMa2005}. Moreover, if one interprets $\zeta$ and $\gamma$ as the electric and magnetic flux densities, respectively, equations \eqref{e11} become a stochastic extension of the well-known Yee's scheme \citep{Yee1966} in computational electromagnetism for the stochastic fields that depend on a single spatial variable.

Moreover, in this trivial Lie group case, the operations $\mathrm{Ad}^{*}$ and $\left[\mathrm{d}\tau^{-1}\right]^{*}$ specialize to identity maps and the adjoint action $\mathrm{ad}$ vanishes, which implies that the exact variation \eqref{var_zeta_2}-\eqref{var_gamma_2} are the same as the approximate variation formulae \eqref{app_variation_formula_1}. Thus, the approximate integrator \eqref{App_Integrator_1} coincides with the exact integrator \eqref{without_truncation_exact} in this case.
\end{exam}

\begin{exam}[Stochastic geometrically exact beam]
For the stochastic geometrically exact beam, the discrete action sum reads
\begin{equation*}
A_{d} = \sum_{n=0}^{N-1}\sum_{j=0}^{J-1}\Delta t\Delta s
\ell\left(g^{n}_{j},\zeta^{n+1/2}_{j},\gamma^{n}_{j+1/2}\right),
\end{equation*}
where $g^{n}_{j}\in SE(3)$ and $\zeta^{n+1/2}, \gamma^{n}_{j+1/2}\in\mathfrak{se}(3)$. The resulting Type-I approximate integrator reads
\begin{align*}
&\Delta s
\left(
\mathcal{K}\zeta^{n-\frac{1}{2}}_{j} - \mathcal{K}\zeta^{n+\frac{1}{2}}_{j}
\right)
-
\Delta t
\left(
\mathcal{P}\gamma^{n}_{j-\frac{1}{2}}
-
\mathcal{P}\gamma^{n}_{j+\frac{1}{2}}
\right)  
-
\frac{\Delta t}{2}
\mathrm{ad}^{*}_{
\Delta s\gamma^{n}_{j+\frac{1}{2}}
}
\mathcal{P}\gamma^{n}_{j+\frac{1}{2}}
-
\frac{\Delta t}{2}
\mathrm{ad}^{*}_{
\Delta s\gamma^{n}_{j-\frac{1}{2}}
}
\mathcal{P}\gamma^{n}_{j-\frac{1}{2}}  \\
& 
+
\frac{\Delta s}{2}
\mathrm{ad}^{*}_{
\Delta t\zeta^{n+\frac{1}{2}}_{j} 
+ 
\Theta^{n+\frac{1}{2}}_{j}
\left(W^{(M)}_{n+1,j} - W^{(M)}_{n,j}\right)
}
\mathcal{K}\zeta^{n+\frac{1}{2}}_{j}
+
\frac{\Delta s}{2}
\mathrm{ad}^{*}_{
\Delta t\zeta^{n-\frac{1}{2}}_{j} 
+ 
\Theta^{n-\frac{1}{2}}_{j}
\left(W^{(M)}_{n,j} - W^{(M)}_{n-1,j}\right)
}
\mathcal{K}\zeta^{n-\frac{1}{2}}_{j}
= 0,
\end{align*}
together with the discrete stochastic compatibility condition
\begin{align*}
&\Delta s\left(\gamma^{n+1}_{j+1/2} - \gamma^{n}_{j+1/2}\right)
- \Delta t\left(\zeta^{n+1/2}_{j+1} - \zeta^{n+1/2}_{j}\right) 
+ \Theta^{n+1/2}_{j}
\left(W^{(M)}_{n+1,j} - W^{(M)}_{n,j}\right)
-  \Theta^{n+1/2}_{j+1}
\left(W^{(M)}_{n+1,j+1} - W^{(M)}_{n,j+1}\right)  \\
&= \frac{1}{2}
\left[
\Delta s\gamma^{n+1}_{j+1/2},
\Delta t\zeta^{n+1/2}_{j} + \Theta^{n+1/2}_{j}
\left(W^{(M)}_{n+1,j} - W^{(M)}_{n,j}\right)
\right]  \\
&\qquad
+ \frac{1}{2}
\left[
\Delta s\gamma^{n}_{j+1/2},
\Delta t\zeta^{n+1/2}_{j+1} + \Theta^{n+1/2}_{j+1}
\left(W^{(M)}_{n+1,j+1} - W^{(M)}_{n,j+1}\right)
\right].
\end{align*}
\end{exam}

\subsection{Discrete stochastic Hamilton-Pontryagin principle (Type-II integrator)}\label{d_hph}
The second way of variational discretization is based on the stochastic potential and corresponds to the continuous-time variational principle proposed in Section~\ref{Hamel_HPH}, which results in the following so-called Type-II integrator. In this setting, one can either construct the \textit{exact} variational integrator or the \textit{approximate} integrator as discussed in Section~\ref{d_ham}. Since it has been clarified in detail the relation between the exact and approximate variational integrators, in this section we directly construct the Type-II approximate integrator, which will be used for numerical simulations in Section~\ref{NR}. The discrete temporal and spatial steps are $\Delta t>0$ and $\Delta s>0$, respectively. The discrete curve approximating $\left\{\left(g,\zeta,\mu,\gamma,\sigma\right)\right\}\subseteq G\times \left(\mathfrak{g}\times \mathfrak{g}^{\ast}\right)^{2}$ is
\begin{equation*}
c_{d} =
\left\{
\left(g^{0}_{j},\zeta^{1/2}_{j},\mu^{1/2}_{j},\gamma^{0}_{j+1/2},\sigma^{0}_{j+1/2},g^{1}_{j},\cdots,g^{N-1}_{j},\zeta^{N-1/2}_{j},\mu^{N-1/2}_{j},\gamma^{N-1}_{j+1/2},\sigma^{N-1}_{j+1/2},g^{N}_{j}\right)
\right\},
\end{equation*}
which is in the space
\begin{equation*}
C_{d}(N,J) = \left(G\times\left(\left(\mathfrak{g}\oplus \mathfrak{g}^{\ast}\right)\times\left(\mathfrak{g}\oplus \mathfrak{g}^{\ast}\right)\times G\right)^{N}\right)^{J},
\end{equation*}
where $g^{n}_{j}$ is defined by $g^{n}_{j} := g(t_{n}, s_{j})$, $\zeta^{n+1/2}_{j}$ and $\mu^{n+1/2}_{j}$ are approximations of $\zeta$ and $\mu$ at $\left(t_{n+1/2}, s_{j}\right) := \left(\frac{t_{n}+t_{n+1}}{2}, s_{j}\right)$, respectively, that is,
\begin{equation*}
\zeta^{n+1/2}_{j} \approx \zeta\left(\frac{t_{n}+t_{n+1}}{2}, s_{j}\right)
\quad
\text{and}
\quad
\mu^{n+1/2}_{j} \approx \mu\left(\frac{t_{n}+t_{n+1}}{2}, s_{j}\right),
\end{equation*}
and $\gamma^{n}_{j+1/2}$ and $\sigma^{n}_{j+1/2}$ are approximations of $\gamma$ and $\sigma$ at $\left(t_{n},s_{j+1/2}\right) := \left(t_{n}, \frac{s_{j} + s_{j+1}}{2}\right)$, respectively, i.e.,
\begin{equation*}
\gamma^{n}_{j+1/2} \approx \gamma\left(t_{n}, \frac{s_{j} + s_{j+1}}{2}\right)
\quad
\text{and}
\quad
\sigma^{n}_{j+1/2} \approx \sigma\left(t_{n}, \frac{s_{j} + s_{j+1}}{2}\right).
\end{equation*}
Define the discrete action sum that approximates the action functional \eqref{action_fun} by
\begin{equation}
\begin{aligned}
A_{d} &= \sum_{n=0}^{N-1}\sum_{j=0}^{J-1}
\left[
\Delta t\Delta s\ell\left(g^{n}_{j}, \zeta^{n+1/2}_{j}, \gamma^{n}_{j+1/2}\right)
+ \Delta s\left<\mu^{n+1/2}_{j}, 
\tau^{-1}\left(\left(g^{n}_{j}\right)^{-1}g^{n+1}_{j}\right)
- \Delta t\zeta^{n+1/2}_{j}\right>  \right.\\
& \left.\qquad\qquad\quad
+ \Delta t\left<\sigma^{n}_{j+1/2}, 
\tau^{-1}\left(\left(g^{n}_{j}\right)^{-1}g^{n}_{j+1}\right)
- \Delta s\gamma^{n}_{j+1/2}\right>
- \Delta s\Phi\left(g^{n}_{j}, \mu^{n+1/2}_{j}\right)
\left(
W^{(M)}_{n+1,j} - W^{(M)}_{n,j}
\right)
\right].
\end{aligned}
\label{action_sum}
\end{equation}

\begin{theorem}[Discrete stochastic Hamilton-Pontryagin principle]\label{Theorem_Type_II}
The following statements are equivalent:
\begin{itemize}
    \item [(i)] The discrete stochastic curve
\begin{equation}
c_{d} =
\left\{
\left(g^{0}_{j},\zeta^{1/2}_{j},\mu^{1/2}_{j},\gamma^{0}_{j+1/2},\sigma^{0}_{j+1/2},g^{1}_{j},\cdots,g^{N-1}_{j},\zeta^{N-1/2}_{j},\mu^{N-1/2}_{j},\gamma^{N-1}_{j+1/2},\sigma^{N-1}_{j+1/2},g^{N}_{j}\right)
\right\}
\label{dis_sto_curve_hp}
\end{equation}
satisfies the following discrete stochastic equations:
\begin{numcases}{}
\diamond\hspace{0.2em}\text{\sc Discrete stochastic temporal kinematic equation:} & \label{dis_tem_eqn_2}\hfill\\
\quad \delta\mu^{n+1/2}_{j}: \tau^{-1}\left(\left(g^{n}_{j}\right)^{-1}g^{n+1}_{j}\right)
- \Delta t\zeta^{n+1/2}_{j}
- \frac{\delta\Phi}{\delta\mu}\left(g^{n}_{j}, \mu^{n+1/2}_{j}\right)
\left(
W^{(M)}_{n+1,j} - W^{(M)}_{n,j}
 \right) = 0, & \notag\hfill\\
\diamond\hspace{0.2em}\text{\sc Discrete stochastic spatial kinematic equation:} & \hfill\\
\quad \delta\sigma^{n}_{j+1/2}: \tau^{-1}\left(\left(g^{n}_{j}\right)^{-1}g^{n}_{j+1}\right)
- \Delta s\gamma^{n}_{j+1/2} = 0, & \notag\hfill\\
\diamond\hspace{0.2em}\text{\sc Discrete stochastic temporal Legendre transform:} & \label{leg_tem} \hfill\\
\quad \delta \zeta^{n+1/2}_{j}: \mu^{n+1/2}_{j} = \frac{\delta\ell}{\delta\zeta}\left(g^{n}_{j}, \zeta^{n+1/2}_{j}, \gamma^{n}_{j+1/2}\right), &   \notag\hfill\\
\diamond\hspace{0.2em}\text{\sc Discrete stochastic spatial Legendre transform:} & \label{leg_spa}\hfill\\
\quad \delta \gamma^{n}_{j+1/2}: \sigma^{n}_{j+1/2} = \frac{\delta\ell}{\delta\gamma}\left(g^{n}_{j}, \zeta^{n+1/2}_{j}, \gamma^{n}_{j+1/2}\right), &  \notag\hfill\\
\diamond\hspace{0.2em}\text{\sc Discrete stochastic implicit field-theoretic equation:} & \label{imp_ham}\hfill\\
\quad \eta^{n}_{j}: \Delta s\left(\mu^{n-1/2}_{j} - \mu^{n+1/2}_{j}\right)
+ \Delta t\left(\sigma^{n}_{j-1/2} - \sigma^{n}_{j+1/2}\right)  & \notag\hfill\\
\qquad\quad
+ \frac{\Delta s}{2}
\mathrm{ad}^{*}_{
\tau^{-1}\left(\left(g^{n}_{j}\right)^{-1}g^{n+1}_{j}\right)
}\mu^{n+1/2}_{j}
+ \frac{\Delta s}{2}
\mathrm{ad}^{*}_{
\tau^{-1}\left(\left(g^{n-1}_{j}\right)^{-1}g^{n}_{j}\right)
}\mu^{n-1/2}_{j}  &  \notag\hfill  \\
\qquad\quad
+ \frac{\Delta t}{2}
\mathrm{ad}^{*}_{
\tau^{-1}\left(\left(g^{n}_{j}\right)^{-1}g^{n}_{j+1}\right)
}\sigma^{n}_{j+1/2}
+ \frac{\Delta t}{2}
\mathrm{ad}^{*}_{
\tau^{-1}\left(\left(g^{n}_{j-1}\right)^{-1}g^{n}_{j}\right)
}\sigma^{n}_{j-1/2}  & \notag\hfill\\
\qquad\quad
- \Delta s T_{e}L^{\ast}_{g^{n}_{j}}
\frac{\delta\Phi}{\delta g}\left(g^{n}_{j}, \mu^{n+1/2}_{j}\right)
\left(
W^{(M)}_{n+1,j} - W^{(M)}_{n,j}
\right)
+ \Delta t\Delta s T_{e}L^{\ast}_{g^{n}_{j}}
\frac{\delta\ell}{\delta g}\left(g^{n}_{j}, \zeta^{n+1/2}_{j}, \gamma^{n}_{j+1/2}\right)
= 0.  & \notag \hfill
\end{numcases}

    \item[(ii)] The discrete stochastic curve \eqref{dis_sto_curve_hp} is a critical point of the discrete action sum \eqref{action_sum}, that is, $\delta A_{d} = 0$, 
    with independent variations $\delta\zeta^{n+1/2}_{j}$, $\delta\mu^{n+1/2}_{j}$, $\delta\sigma^{n}_{j+1/2}$, $\delta\gamma^{n}_{j+1/2}$, $\eta^{n}_{j}$, 
    and variations of $\tau^{-1}\left(\left(g^{n}_{j}\right)^{-1}g^{n+1}_{j}\right)$ and $\tau^{-1}\left(\left(g^{n}_{j}\right)^{-1}g^{n}_{j+1}\right)$ approximated as
    \begin{equation}
    \begin{aligned}
    \delta \left[\tau^{-1}\left(\left(g^{n}_{j}\right)^{-1}g^{n+1}_{j}\right)\right]
    &=
    \eta^{n+1}_{j} - \eta^{n}_{j}
    + \mathrm{ad}_{
    \tau^{-1}\left(\left(g^{n}_{j}\right)^{-1}g^{n+1}_{j}\right)
    }\eta^{n+1/2}_{j},  \\
    \delta \left[\tau^{-1}\left(\left(g^{n}_{j}\right)^{-1}g^{n}_{j+1}\right)\right]
    &=
    \eta^{n}_{j+1} - \eta^{n}_{j}
    + \mathrm{ad}_{
    \tau^{-1}\left(\left(g^{n}_{j}\right)^{-1}g^{n}_{j+1}\right)
    }\eta^{n}_{j+1/2},  
    \label{approx_variation_Type_2}
    \end{aligned}
    \end{equation}
    together with the fixed point condition $\eta^{0}_{j} = \eta^{N}_{j} = 0$, $j=1,...,J$.
\end{itemize}
\end{theorem}

\begin{proof}
By taking the variation of \eqref{action_sum} and using \eqref{approx_variation_Type_2}, we have
\begin{align*}
\delta A_{d}
&= \sum_{n=1}^{N-1}\sum_{j=1}^{J-1}
\left[
\Delta t\Delta s\left<\frac{\delta\ell}{\delta g}\left(g^{n}_{j}, \zeta^{n+1/2}_{j}, \gamma^{n}_{j+1/2}\right), \delta g^{n}_{j}\right>
+
\Delta t\Delta s\left<\frac{\delta\ell}{\delta\zeta}\left(g^{n}_{j}, \zeta^{n+1/2}_{j}, \gamma^{n}_{j+1/2}\right), \delta\zeta^{n+1/2}_{j}\right>  \right.  \\
&\qquad
+
\Delta t\Delta s\left<\frac{\delta\ell}{\delta\gamma}\left(g^{n}_{j}, \zeta^{n+1/2}_{j}, \gamma^{n}_{j+1/2}\right), \delta\gamma^{n}_{j+1/2}\right>
+
\left<\delta\mu^{n+1/2}_{j}, 
\tau^{-1}\left(\left(g^{n}_{j}\right)^{-1}g^{n+1}_{j}\right)
- \Delta t\zeta^{n+1/2}_{j}\right>\Delta s  \\
&\qquad
+
\left<\mu^{n+1/2}_{j}, 
\eta^{n+1}_{j} - \eta^{n}_{j}
+ \left[
\tau^{-1}\left(\left(g^{n}_{j}\right)^{-1}g^{n+1}_{j}\right), \eta^{n+1/2}_{j}
\right]
- \Delta t\delta\zeta^{n+1/2}_{j}\right>\Delta s  \\
&\qquad
+
\left<\delta\sigma^{n}_{j}, 
\tau^{-1}\left(\left(g^{n}_{j}\right)^{-1}g^{n}_{j+1}\right)
- \Delta s\gamma^{n}_{j+1/2}\right>\Delta t  \\
&\qquad
+
\left<\sigma^{n}_{j+1/2}, \eta^{n}_{j+1} - \eta^{n}_{j}
+ \left[
\tau^{-1}\left(\left(g^{n}_{j}\right)^{-1}g^{n}_{j+1}\right), \eta^{n}_{j+1/2}
\right]
- \Delta s\delta\gamma^{n}_{j+1/2}\right>\Delta t  \\
&\qquad \left.
-
\Delta s
\left(
\left<\frac{\delta\Phi}{\delta g}\left(g^{n}_{j}, \mu^{n+1/2}_{j}\right), \delta g^{n}_{j}\right>
+ \left<\frac{\delta\Phi}{\delta \mu}\left(g^{n}_{j}, \mu^{n+1/2}_{j}\right), \delta \mu^{n+1/2}_{j}\right>
\right)
\left(
W^{(M)}_{n+1,j} - W^{(M)}_{n,j}
\right)
\right]  \\
& + \text{Boundary terms},
\end{align*}
where we have rearranged the indices $n$ and $j$. By the boundary condition $\delta g^{0}_{j} = \delta g^{N}_{j} = 0$, $j=1,...,J$, (or equivalently $\eta^{0}_{j} = \eta^{N}_{j} = 0$, $j=1,...,J$), and the arbitrariness of $\delta\zeta^{n+1/2}_{j}$, $\delta\gamma^{n}_{j+1/2}$, $\delta\mu^{n+1/2}_{j}$, and $\delta\sigma^{n}_{j+1/2}$, we obtain the desired result.
\end{proof}

The full dynamics is given by \eqref{dis_tem_eqn_2}-\eqref{imp_ham}, together with the discrete stochastic compatibility condition
\begin{equation}
\begin{aligned}
&\Delta s\left(\gamma^{n+1}_{j+1/2} - \gamma^{n}_{j+1/2}\right)
- \Delta t\left(\zeta^{n+1/2}_{j+1} - \zeta^{n+1/2}_{j}\right)  \\
&\quad
+ \frac{\delta\Phi}{\delta\mu}\left(g^{n}_{j},\mu^{n+1/2}_{j}\right)
\left(W^{(M)}_{n+1,j} - W^{(M)}_{n,j}\right)
-  \frac{\delta\Phi}{\delta\mu}\left(g^{n}_{j+1},\mu^{n+1/2}_{j+1}\right)
\left(W^{(M)}_{n+1,j+1} - W^{(M)}_{n,j+1}\right)  \\
&\quad
- \frac{1}{2}
\left[
\Delta s\gamma^{n+1}_{j+1/2},
\Delta t\zeta^{n+1/2}_{j} + \frac{\delta\Phi}{\delta\mu}\left(g^{n}_{j},\mu^{n+1/2}_{j}\right)
\left(W^{(M)}_{n+1,j} - W^{(M)}_{n,j}\right)
\right]  \\
&\quad
- \frac{1}{2}
\left[
\Delta s\gamma^{n}_{j+1/2},
\Delta t\zeta^{n+1/2}_{j+1} + \frac{\delta\Phi}{\delta\mu}\left(g^{n}_{j+1},\mu^{n+1/2}_{j+1}\right)
\left(W^{(M)}_{n+1,j+1} - W^{(M)}_{n,j+1}\right)
\right]
= 0.
\label{dis_compatibility_2}
\end{aligned}
\end{equation}

\begin{rmk}\label{equi_discrete}
By discretizing the stochastic Hamilton's principle, we have obtained the Type-I integrator (and its variants) in Section~\ref{d_ham}. On the other hand, by discretizing the stochastic Hamilton-Pontryagin principle, we have obtained the Type-II integrator in Section~\ref{d_hph}. It is natural to consider the relation between the obtained integrators. In fact, comparing the stochastic integrators in Theorem~\ref{Theorem_App_Int} and \ref{Theorem_Type_II} term by term, one finds that they are equivalent if and only if
\begin{equation*}
\frac{\delta\Phi}{\delta g}\left(g^{n}_{j},\mu^{n+1/2}_{j}\right) = 0
\qquad
\text{and}
\qquad
\frac{\delta\Phi}{\delta \mu}\left(g^{n}_{j},\mu^{n+1/2}_{j}\right)
= \Theta^{n+1/2}_{j},
\end{equation*}
which is equivalent to 
\begin{equation*}
    \Phi\left(g^{n}_{j},\mu^{n+1/2}_{j}\right) = \left<\mu^{n+1/2}_{j},\Theta^{n+1/2}_{j}\right> + C,
\end{equation*}
where $C\in\mathbb{R}$ is a constant, which is a discrete counterpart of the equivalence condition \eqref{equivalence}.
\end{rmk}

\begin{exam}[Discrete stochastic wave equations, continued]
The type-II stochastic integrator \eqref{dis_tem_eqn_2}-\eqref{dis_compatibility_2} for the stochastic wave system becomes
\begin{align*}
&
\Delta s \left(\mu^{n-1/2}_{j} - \mu^{n+1/2}_{j}\right)
+ \left(\sigma^{n}_{j-1/2} - \sigma^{n}_{j+1/2}\right)\Delta t
- \Delta s\frac{\delta\Phi}{\delta\phi}\left(\phi^{n}_{j},\mu^{n+1/2}_{j}\right)
\left(
W^{(M)}_{n+1,j} - W^{(M)}_{n,j}
\right)
= 0, \\
&\Delta s\left(\gamma^{n+1}_{j+1/2} - \gamma^{n}_{j+1/2}\right)
- \Delta t\left(\zeta^{n+1/2}_{j+1} - \zeta^{n+1/2}_{j}\right)  \\
&\quad
+ \frac{\delta\Phi}{\delta\mu}\left(\phi^{n}_{j},\mu^{n+1/2}_{j}\right)
\left(W^{(M)}_{n+1,j} - W^{(M)}_{n,j}\right)
-  \frac{\delta\Phi}{\delta\mu}\left(\phi^{n}_{j+1},\mu^{n+1/2}_{j+1}\right)
\left(W^{(M)}_{n+1,j+1} - W^{(M)}_{n,j+1}\right) = 0,  \\
&\phi^{n+1}_{j} - \phi^{n}_{j}
- \Delta t\zeta^{n+1/2}_{j}
- \frac{\delta\Phi}{\delta\mu}\left(\phi^{n}_{j},\mu^{n+1/2}_{j}\right)
\left(
W^{(M)}_{n+1,j} - W^{(M)}_{n,j}
\right) = 0, \\
&\phi^{n}_{j+1} - \phi^{n}_{j}
- \Delta s\gamma^{n}_{j+1/2} = 0, \\
&\mu^{n+1/2}_{j} = \zeta^{n+1/2}_{j}, \\
&\sigma^{n}_{j+1/2} = - \gamma^{n}_{j+1/2}, \\
&
\gamma^{n}_{1/2} = \gamma^{n}_{J+1/2} = 0.
\end{align*}
In particular, when the stochastic potential $\Phi$ is taken to be zero, the above integrator specializes to the well-known Yee's scheme \citep{Yee1966} in the deterministic setting.
\end{exam}

\subsection{Numerical updating scheme}
The proposed two types of stochastic integrators share a common numerical updating scheme, which enables parallel computation and space-time staggered propagation as the following Fig.~\ref{num_scheme} shows. The discrete configurations $g^{n}_{j}$, the velocities $\zeta^{n+1/2}_{j}$, and the strains $\gamma^{n}_{j+1/2}$ are represented by circles $\bigcirc$, squares $\square$, and triangles $\triangle$, respectively. It is seen in Fig.~\ref{num_scheme} that the update of the temporal velocity $\zeta^{n+1/2}_{j}$ only depends on adjacent states as implied by \eqref{App_Integrator_1} and \eqref{imp_ham}. Similarly, the update of the spatial stain $\gamma^{n}_{j+1/2}$ only depends on its adjacent states. This \textit{space-time staggered} scheme is important, because it facilitates parallel computation, which is critical for fast simulations.

\begin{figure}[pos=ht]
\centering
\vspace{0.5em}
\begin{overpic}[width = 0.75\textwidth]{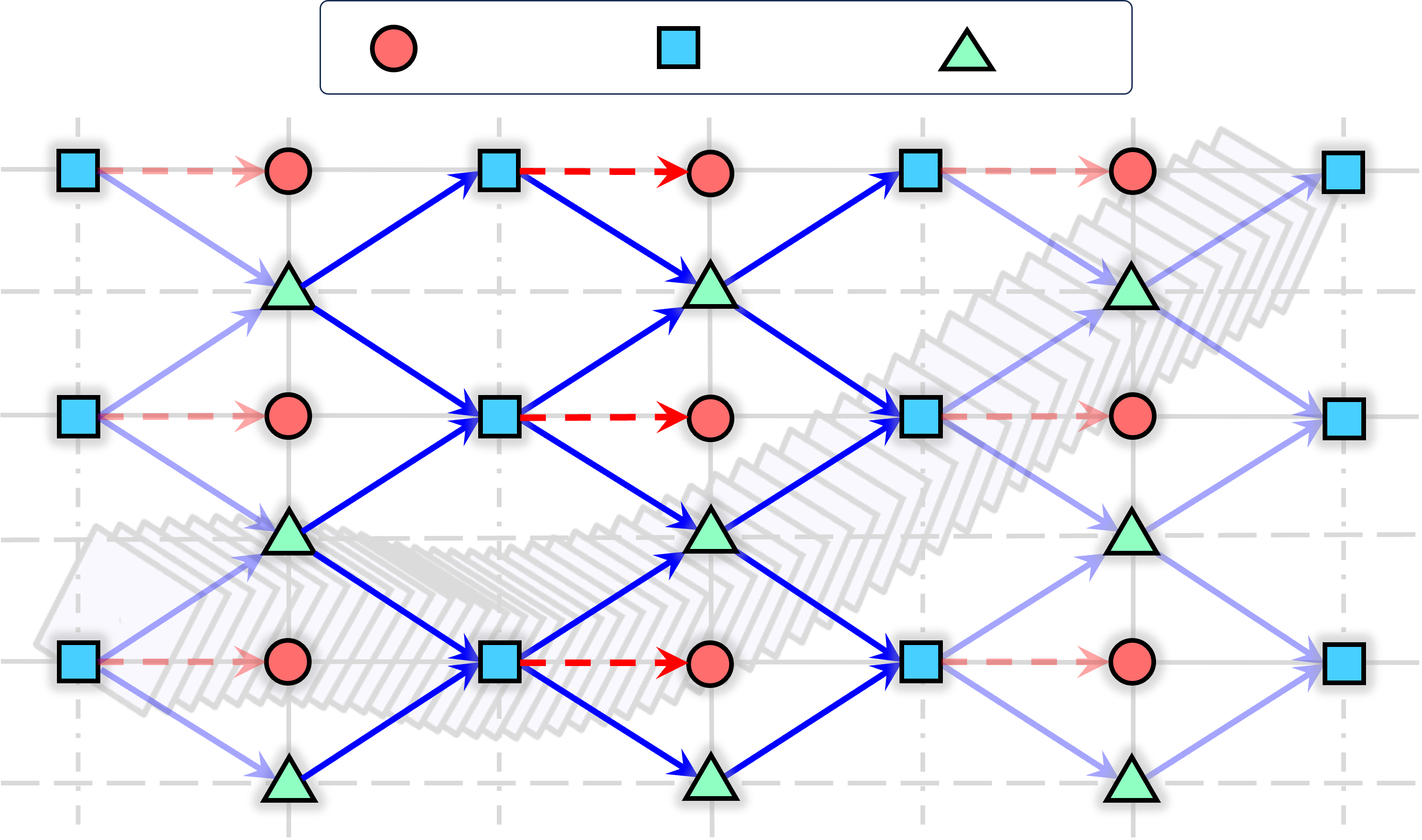}
\put(0,52){{\color{red}$\longrightarrow$} \footnotesize Time}
\put(0,52){$\bullet$}
\put(-0.2,54){{\color{blue}$\big\uparrow$}}
\put(-1,58){\footnotesize Space}
\put(30.5,55){\large $g^{n}_{j}$}
\put(50,55){\large $\zeta^{n+1/2}_{j}$}
\put(70.5,55){\large $\gamma^{n}_{j+1/2}$}
\put(-6,46.5){$j+1$}
\put(-6,38){$j+\frac{1}{2}$}
\put(-2,29.5){$j$}
\put(-6,20.5){$j-\frac{1}{2}$}
\put(-6,11.6){$j-1$}
\put(-6,3.5){$j-\frac{3}{2}$}
\put(37,48.8){$\zeta^{n-1/2}_{j+1}$}
\put(36.3,32.5){$\zeta^{n-1/2}_{j}$}
\put(63.5,32.8){$\zeta^{n+1/2}_{j}$}
\put(78.5,32.8){$g^{n+1}_{j}$}
\put(48.5,42){$\gamma^{n}_{j+1/2}$}
\put(48.5,25){$\gamma^{n}_{j-1/2}$}
\put(17.5,-2){$n-1$}
\put(31.5,-2){$n-1/2$}
\put(49.5,-2){$n$}
\put(61.5,-2){$n+1/2$}
\put(77,-2){$n+1$}
\end{overpic}
\vspace{0.5em}
\caption{Space-time staggered updating scheme of the proposed stochastic integrators.}
\label{num_scheme}
\vspace{-1em}
\end{figure}

\section{Numerical results}\label{NR}
In this section, we utilize the proposed stochastic integrators to simulate the stochastic geometrically exact beam and the stochastic wave equations. The proposed methods are implemented on the platform MATLAB$^{\circledR}$ R2022a with a computer equipped with 11th Gen Intel(R) Core(TM) i7-11800H $@$2.30 GHz processor and 32.0 GB RAM. More descriptions and accompanied videos can be found on the website: \url{https://tianzhi-li.github.io/Stochastic-Beam}

\subsection{Stochastic wave equations}\label{Sec_Sto_Wave}
For the example of wave equations, the simulation duration time is $[0,20]s$. The temporal and spatial steps are chosen as $\Delta t = 2\times 10^{-4}$ $s$ and $\Delta s = \frac{2\pi}{100}$ $m$, respectively. The total iteration number is $N=100000$. The system is initiated with the initial positions $\phi^{0}_{j} = \cos(j\Delta s)$ and the initial velocities  $\zeta^{1/2}_{j} = 0$, for $j=1,...,J$. The stochastic perturbations are described by the diffusivity vector fields 
\begin{align*}
\Theta = 5\times 10^{-4},
\quad
10^{-3},
\quad 
5\times 10^{-3},
\quad
\text{and}
\quad
10^{-2}\in\mathbb{R}
\end{align*}
for the Type-I stochastic integrator introduced in \S~\ref{d_ham}, and the stochastic potentials 
\begin{equation}
\Phi\left(g,\mu\right) = \int_{0}^{S}5\times10^{-4}\cdot\mu ds, 
\quad
\int_{0}^{S}10^{-3}\cdot\mu ds, 
\quad
\int_{0}^{S}5\times10^{-3}\cdot\mu ds, 
\quad
\mathrm{and}
\quad
\int_{0}^{S}10^{-2}\cdot\mu ds.
\label{sto_potential_wave}
\end{equation}
for the Type-II stochastic integrator introduced in \S~\ref{d_hph}. In this case, it follows from Remark~\ref{equi_discrete} that the Type-I and II integrators are equivalent.

\begin{figure*}[pos=t]
    \begin{subfigure}[t]{0.348\textwidth}
    \centering
    \begin{minipage}[t]{0.05\textwidth}  
        \vspace{0.95cm}
        \rotatebox{90}{\textbf{Deterministic}}
    \end{minipage}
    \begin{minipage}[t]{0.92\textwidth}   
        \centering
        \caption*{\hspace{1.8em}Energy conservation}
        \includegraphics[width=\textwidth]{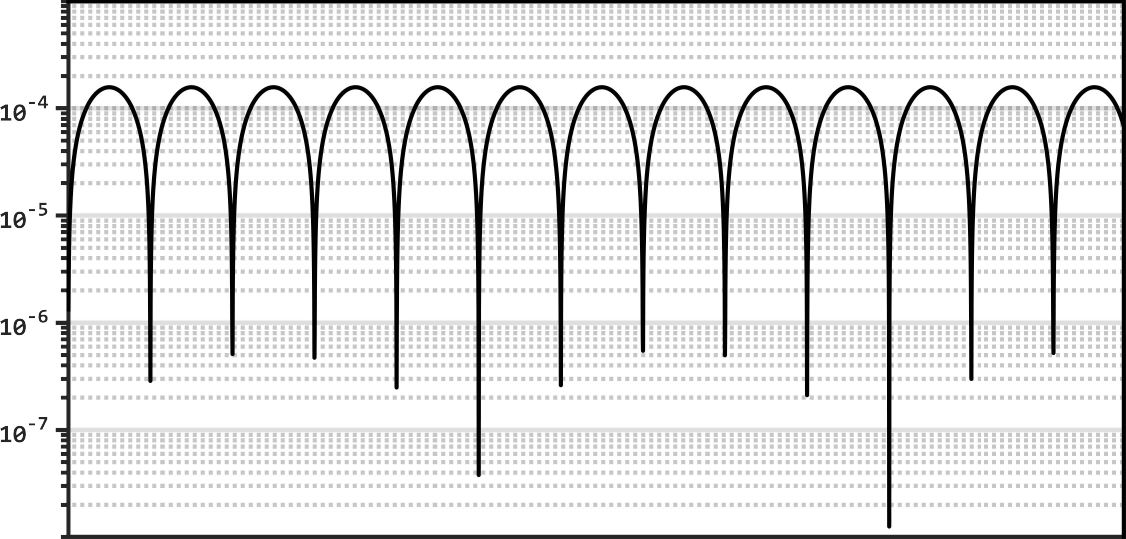}
    \end{minipage}
    \end{subfigure}
    \hfill
    \begin{subfigure}[t]{0.2999\textwidth}
        \caption*{Contour}
        \centering
        \includegraphics[width=\textwidth]{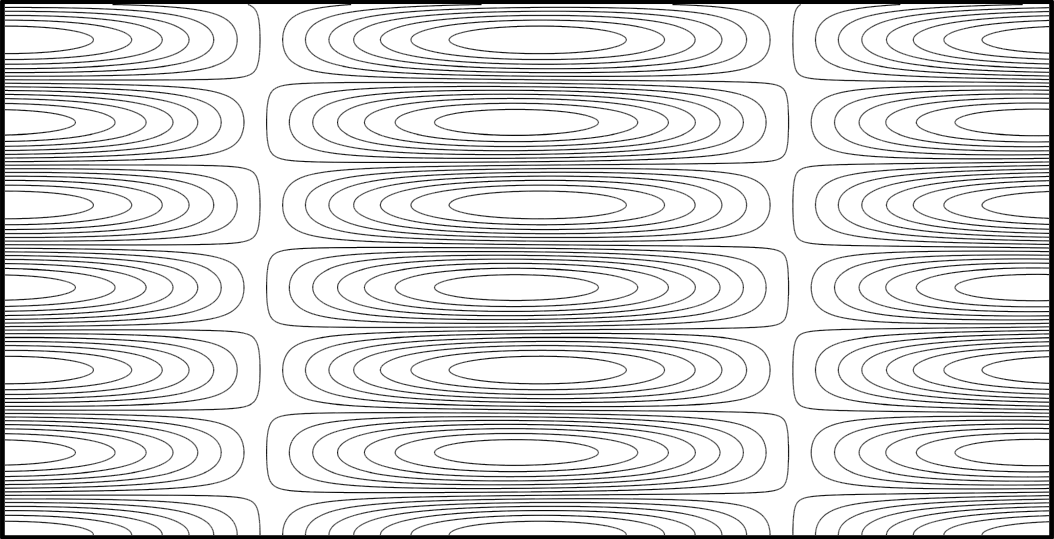}
    \end{subfigure}
    \hfill
    \begin{subfigure}[t]{0.2999\textwidth}
        \caption*{Wave value}
        \centering
        \includegraphics[width=\textwidth]{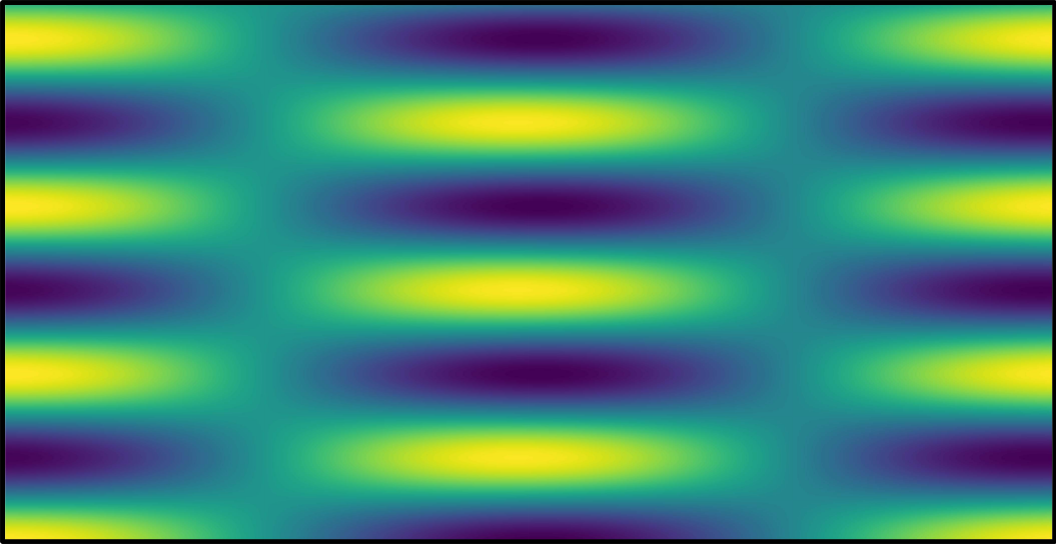}
    \end{subfigure}

    \vspace{0.8em}
    \begin{subfigure}[t]{0.348\textwidth}
    \centering
    \begin{minipage}[t]{0.05\textwidth}  
        \vspace{-2.3cm}
        \rotatebox{90}{$\boldsymbol{\Theta} = 5\times10^{-4}$}
    \end{minipage}
    \begin{minipage}[t]{0.92\textwidth}   
        \centering
        \includegraphics[width=\textwidth]{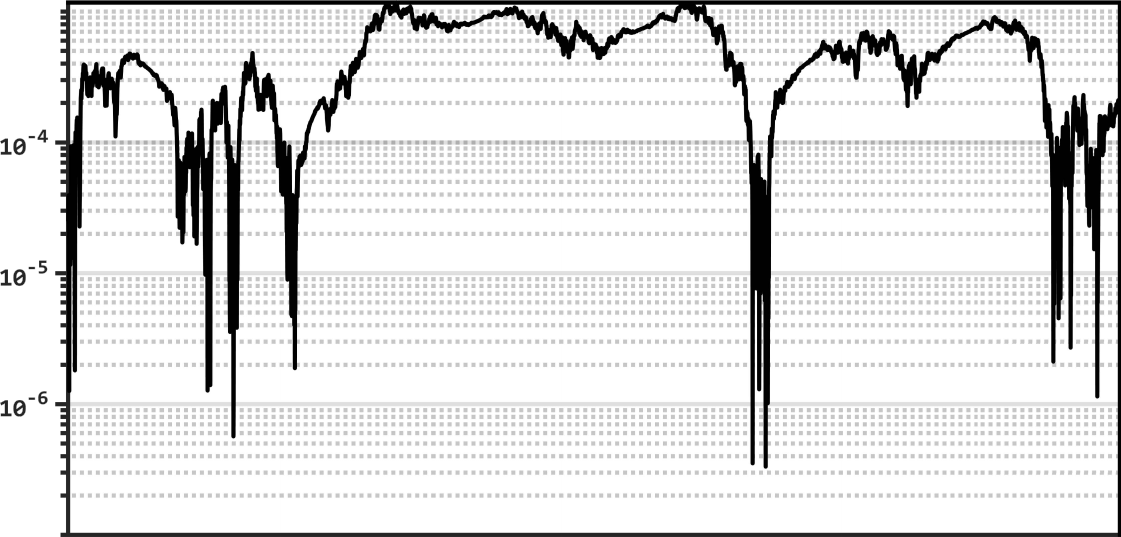}
    \end{minipage}
    \end{subfigure}
    \hfill
    \begin{subfigure}[t]{0.2999\textwidth}
        \centering
        \includegraphics[width=\textwidth]{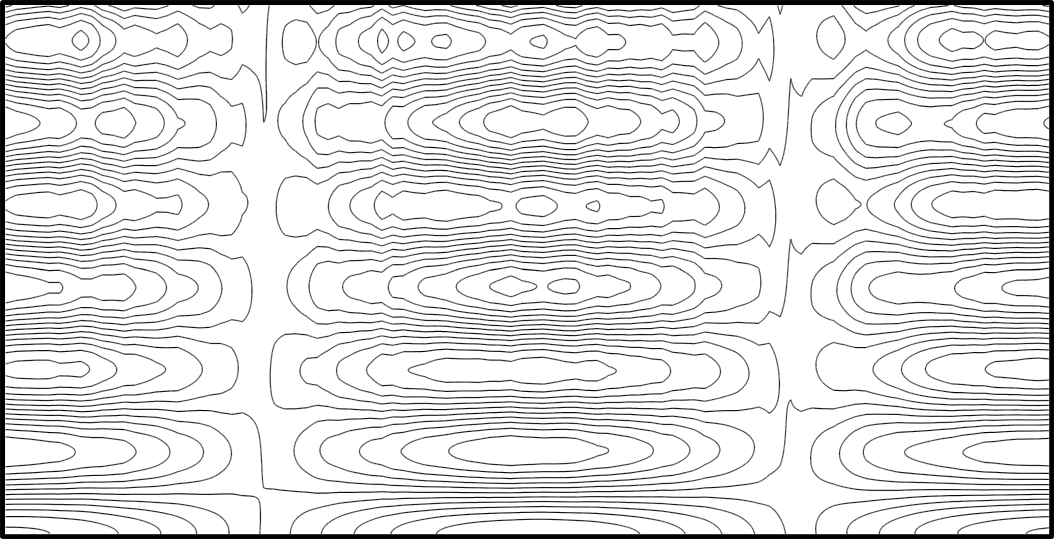}
    \end{subfigure}
    \hfill
    \begin{subfigure}[t]{0.2999\textwidth}
        \centering
        \includegraphics[width=\textwidth]{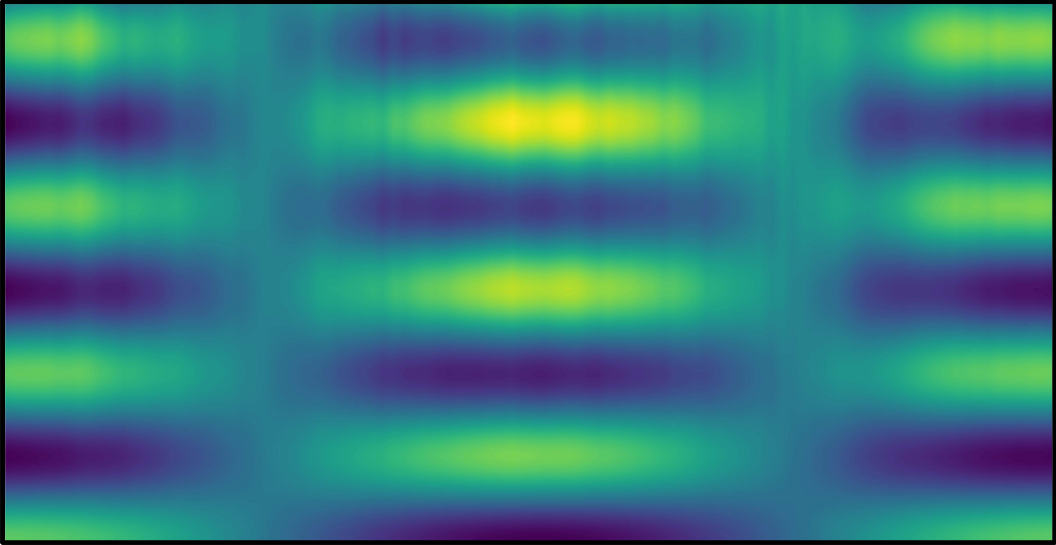}
    \end{subfigure}

    \vspace{0.8em}
    \begin{subfigure}[t]{0.348\textwidth}
    \centering
    \begin{minipage}[t]{0.05\textwidth}  
        \vspace{-1.95cm}
        \rotatebox{90}{$\boldsymbol{\Theta} = 10^{-3}$}
    \end{minipage}
    \begin{minipage}[t]{0.92\textwidth}   
        \centering
        \includegraphics[width=\textwidth]{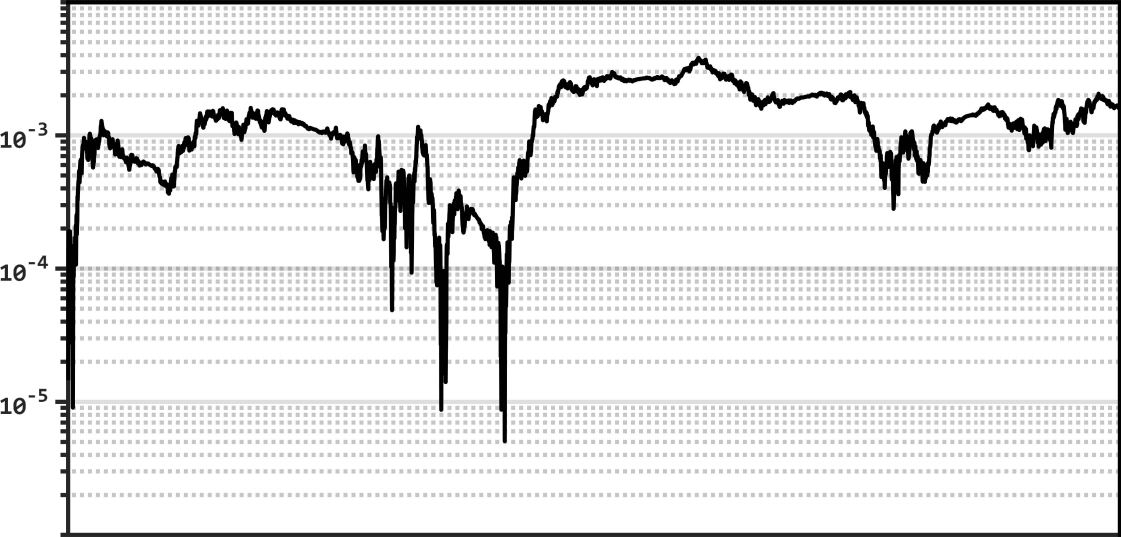}
    \end{minipage}
    \end{subfigure}
    \hfill
    \begin{subfigure}[t]{0.2999\textwidth}
        \centering
        \includegraphics[width=\textwidth]{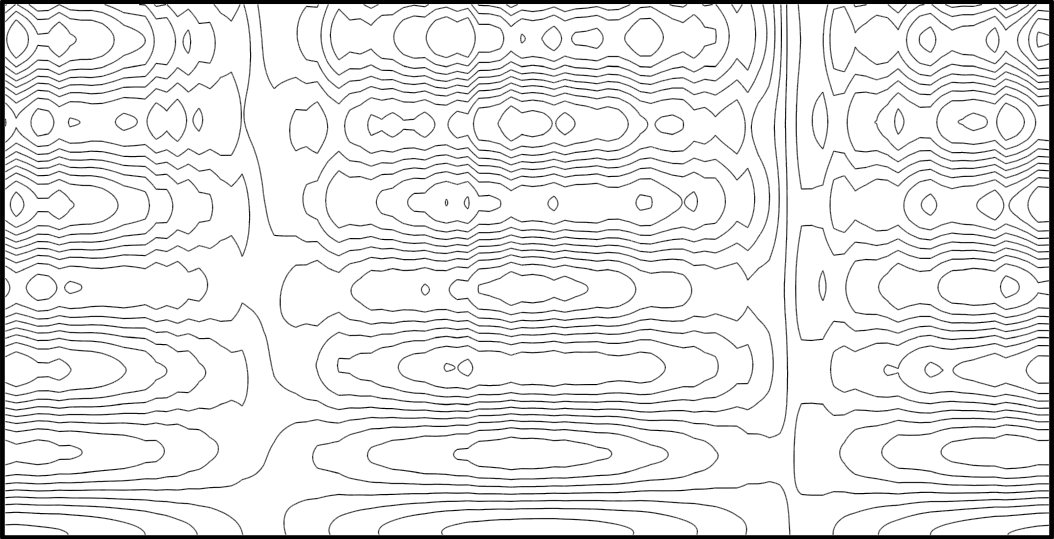}
    \end{subfigure}
    \hfill
    \begin{subfigure}[t]{0.2999\textwidth}
        \centering
        \includegraphics[width=\textwidth]{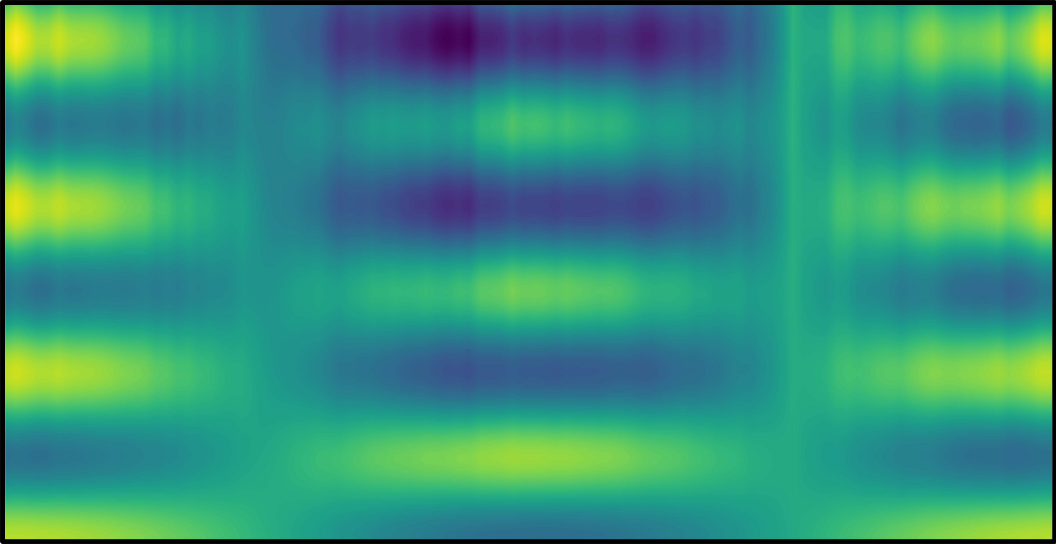}
    \end{subfigure}

    \vspace{0.8em}
    \begin{subfigure}[t]{0.348\textwidth}
    \centering
    \begin{minipage}[t]{0.05\textwidth}  
        \vspace{-2.3cm}
        \rotatebox{90}{$\boldsymbol{\Theta} = 5\times10^{-3}$}
    \end{minipage}
    \begin{minipage}[t]{0.92\textwidth}   
        \centering
        \includegraphics[width=\textwidth]{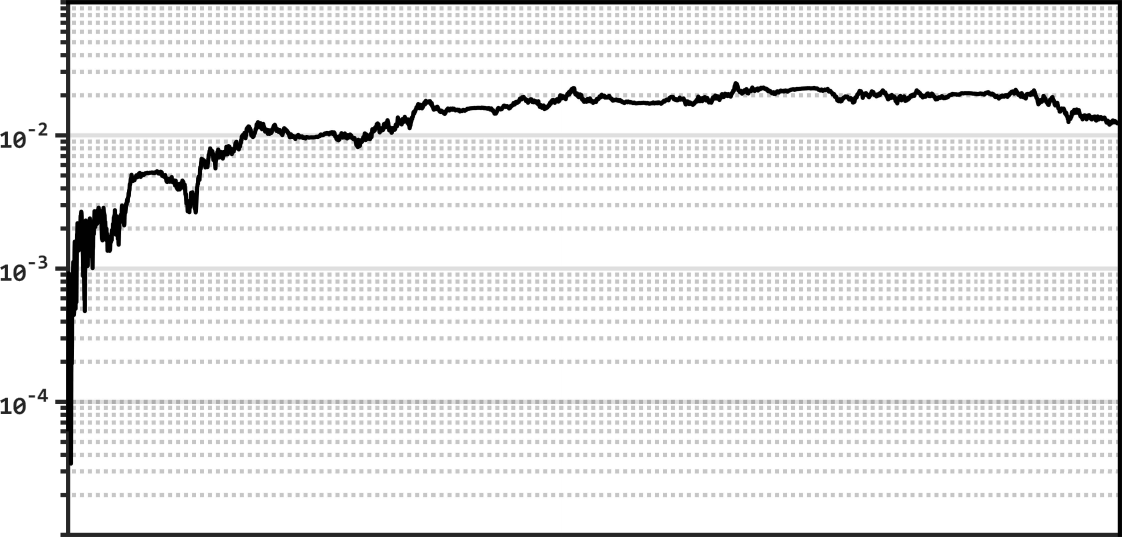}
    \end{minipage}
    \end{subfigure}
    \hfill
    \begin{subfigure}[t]{0.2999\textwidth}
        \centering
        \includegraphics[width=\textwidth]{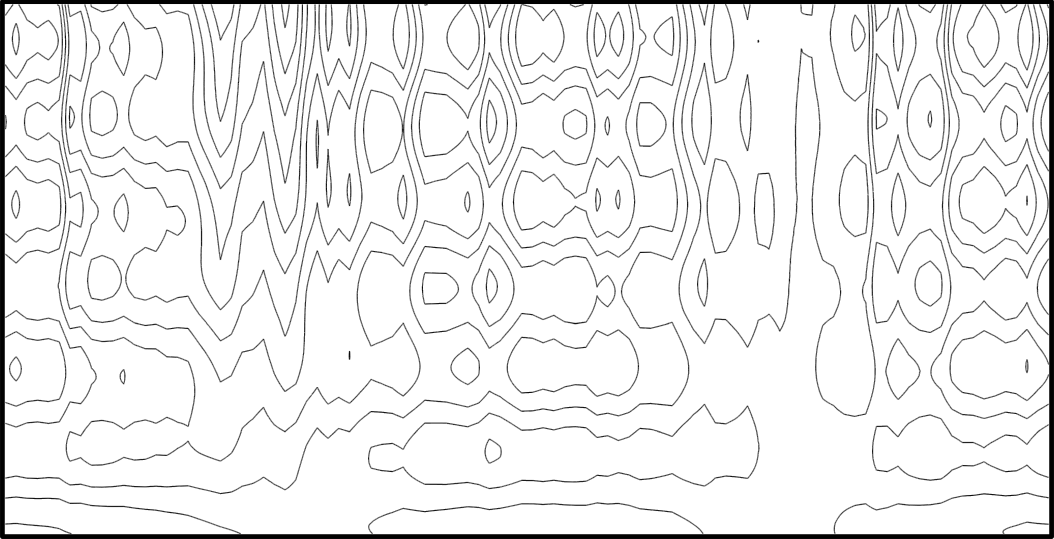}
    \end{subfigure}
    \hfill
    \begin{subfigure}[t]{0.2999\textwidth}
        \centering
        \includegraphics[width=\textwidth]{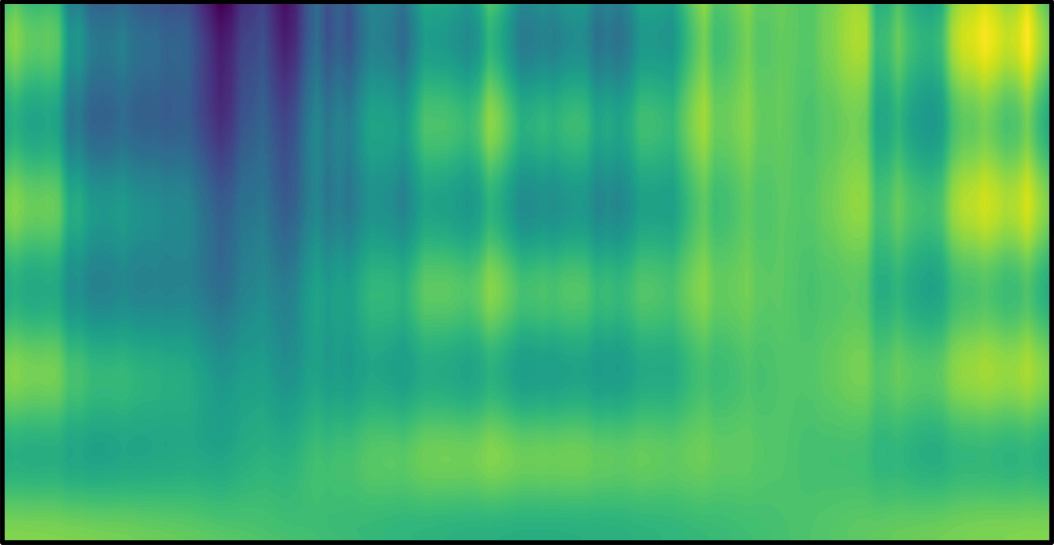}
    \end{subfigure}

    \vspace{0.8em}
    \begin{subfigure}[t]{0.348\textwidth}
    \centering
    \begin{minipage}[t]{0.05\textwidth}  
        \vspace{-1.95cm}
        \rotatebox{90}{$\boldsymbol{\Theta} = 10^{-2}$}
    \end{minipage}
    \begin{minipage}[t]{0.92\textwidth}   
        \centering
        \includegraphics[width=\textwidth]{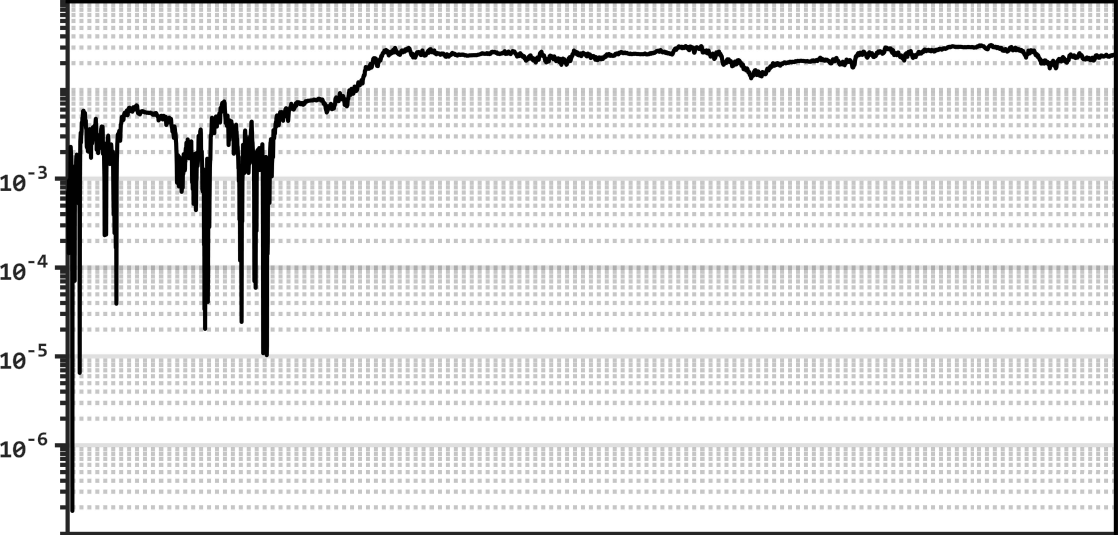}
    \end{minipage}
    \end{subfigure}
    \hfill
    \begin{subfigure}[t]{0.2999\textwidth}
        \centering
        \includegraphics[width=\textwidth]{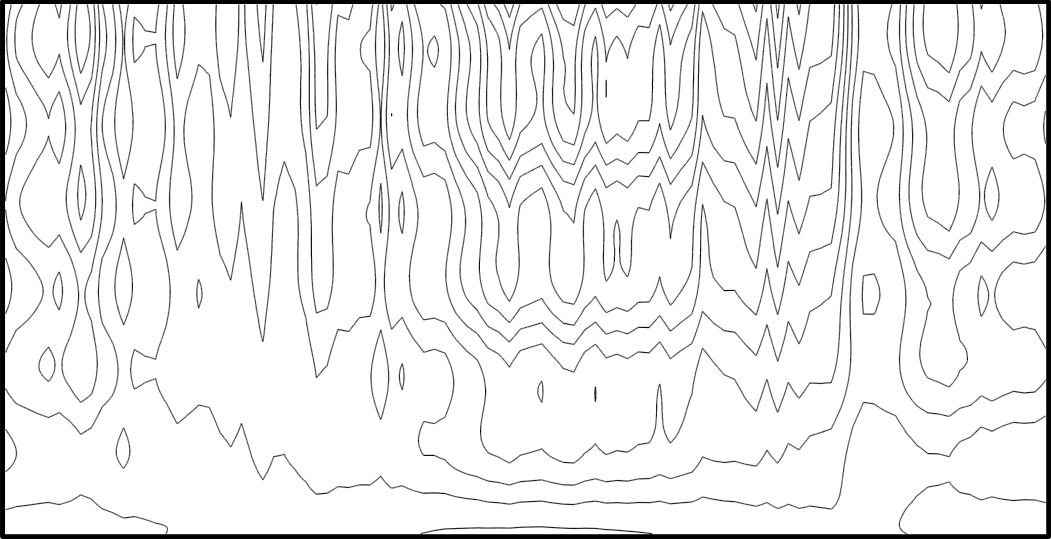}
    \end{subfigure}
    \hfill
    \begin{subfigure}[t]{0.2999\textwidth}
        \centering
        \includegraphics[width=\textwidth]{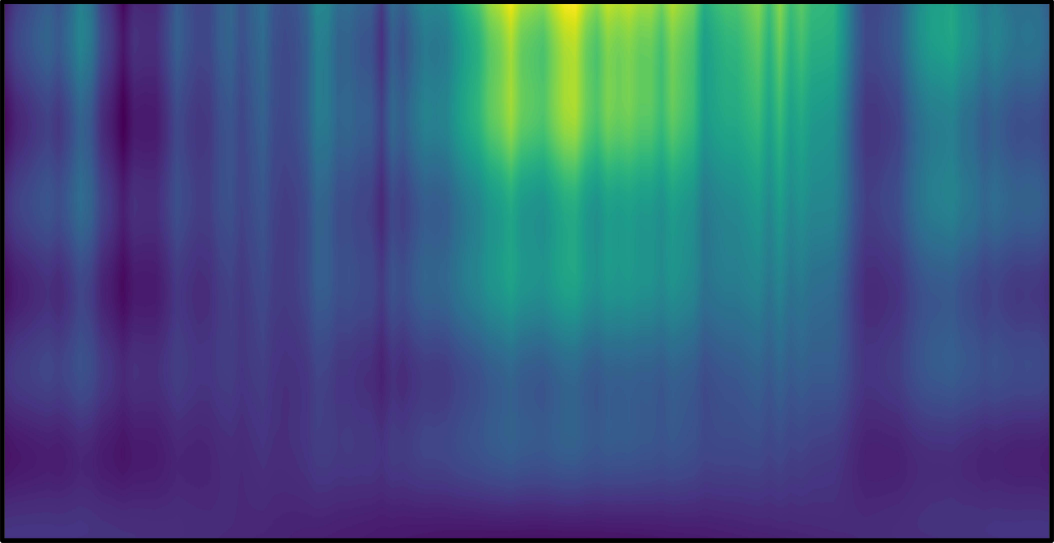}
    \end{subfigure}
    \caption{\textbf{Left}: Energy error performance. \textbf{Middle}: Contours of the wave equation under different perturbations. \textbf{Right}: Wave value at every spacetime node.}
    \label{Wave_Plot}
\end{figure*}
The resulting numerical evolution of the system is shown by contour maps and 2D-view plots in Fig.~\ref{Wave_Plot}. Furthermore, it is observed that the discrete total energy is nearly conserved by the deterministic integrators (the first row of Fig.~\ref{Wave_Plot}). The energy fluctuation depends on the size of time step. With time step $\Delta t = 10^{-4}$, $10^{-5}$, and $10^{-6}$, the energy fluctuation is of order $\mathcal{O}(10^{-3})$, $\mathcal{O}(10^{-4})$, and $\mathcal{O}(10^{-5})$, respectively, within the time horizon $T=20$ s, thereby following the typical energy behavior of symplectic numerical integrators.

In the stochastic setting, due to the existence of stochastic potential/diffusivity vector field, the total energy is in general not conserved by a stochastic system. But for systems that admits the energy conservation (system with a stochastic potential that Poisson-commutes with the Hamiltonian as discussed in Corollary~\ref{Coro_Energy}), the proposed integrators nearly preserves the energy for exponentially long time. Here in the example of stochastic wave equations, the stochastic potential \eqref{sto_potential_wave} Poisson-commutes with the reduced Hamiltonian, i.e., $\{h,\Phi\}$ = 0, which can be easily checked. Consequently, the discrete total energy 
$
E_{d} = 
\sum_{n=0}^{N-1}\sum_{j=0}^{J-1}
\frac{\Delta t\Delta s}{2}
\left[
\left(\zeta^{n+1/2}_{j}\right)^{2}
+
\left(\gamma^{n}_{j+1/2}\right)^{2}
\right]$
is numerically nearly preserved by the proposed integrators (the last four rows of Fig.~\ref{Wave_Plot}).

\subsection{Stochastic bending beam}\label{bending_beam}
\begin{figure*}[pos=htpb]
\centering
    \begin{subfigure}[pos=t]{0.29\textwidth}
    \centering
    \begin{minipage}[pos=t]{0.1\textwidth}  
        \vspace{1.5cm}
        \rotatebox{90}{Iter. $n=4500$}
    \end{minipage}
    \begin{minipage}[pos=t]{0.85\textwidth}   
        \centering
        \caption*{\hspace{1.8em}\textbf{Deterministic}}
        \vspace{-0.3em}
        \includegraphics[width=0.95\textwidth]{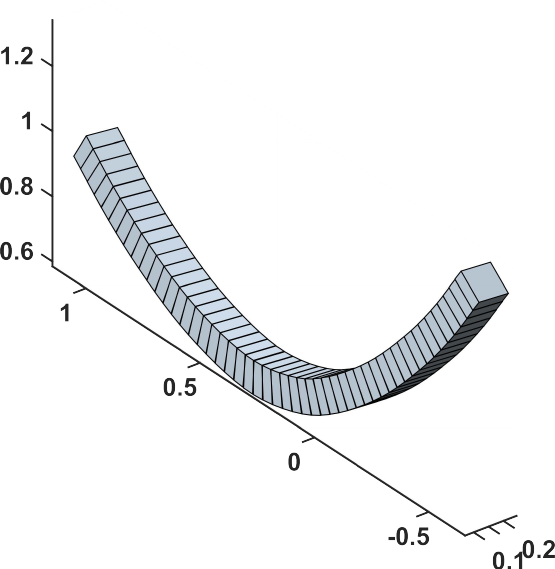}
    \end{minipage}
    \end{subfigure}
    \hspace{2em}
    \begin{subfigure}[pos=t]{0.2465\textwidth}
        \caption*{\textbf{Stochastic}: $\Theta = 5\times 10^{-3}\boldsymbol{1}_{6}$}
        \vspace{-0.3em}
        \centering
        \includegraphics[width=0.95\textwidth]{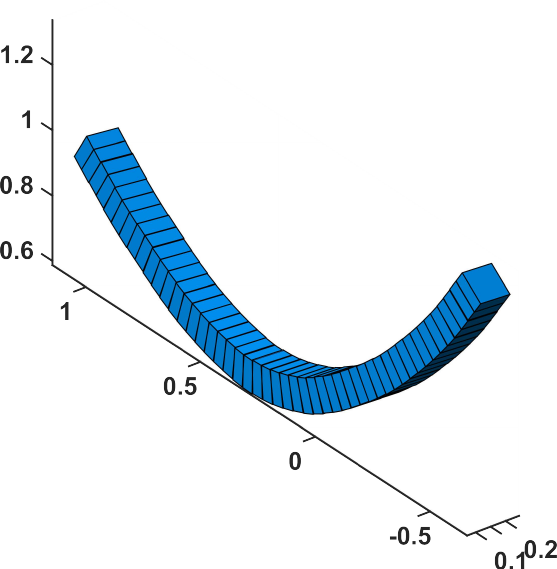}
    \end{subfigure}
    \hspace{2em}
    \begin{subfigure}[pos=t]{0.2465\textwidth}
        \caption*{\textbf{Stochastic}: $\Theta = 5\times 10^{-2}\boldsymbol{1}_{6}$}
        \vspace{-0.3em}
        \centering
        \includegraphics[width=0.95\textwidth]{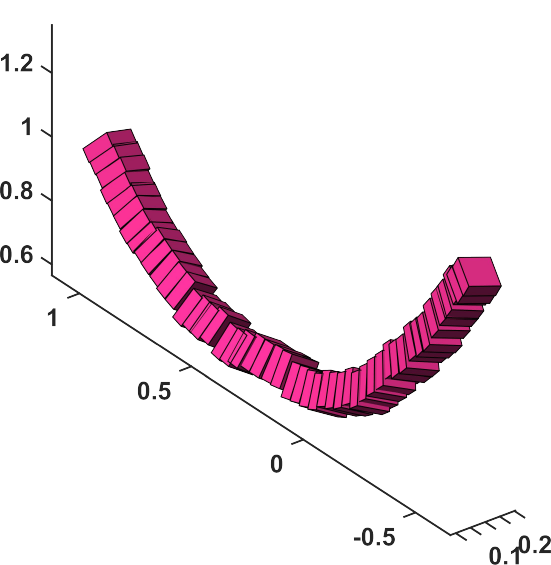}
    \end{subfigure}

    \begin{subfigure}[pos=t]{0.29\textwidth}
    \centering
    \begin{minipage}[pos=t]{0.1\textwidth}  
        \vspace{-3cm}
        \rotatebox{90}{$n=9000$}
    \end{minipage}
    \begin{minipage}[pos=t]{0.85\textwidth}   
        \centering
        \includegraphics[width=0.95\textwidth]{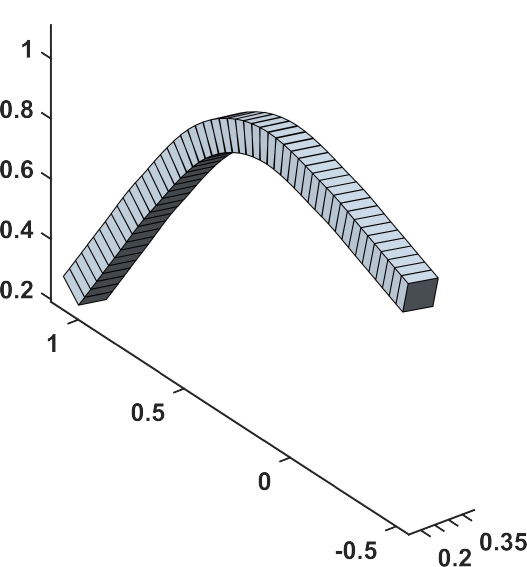}
    \end{minipage}
    \end{subfigure}
    \hspace{2em}
    \begin{subfigure}[pos=t]{0.2465\textwidth}
        \centering
        \includegraphics[width=0.95\textwidth]{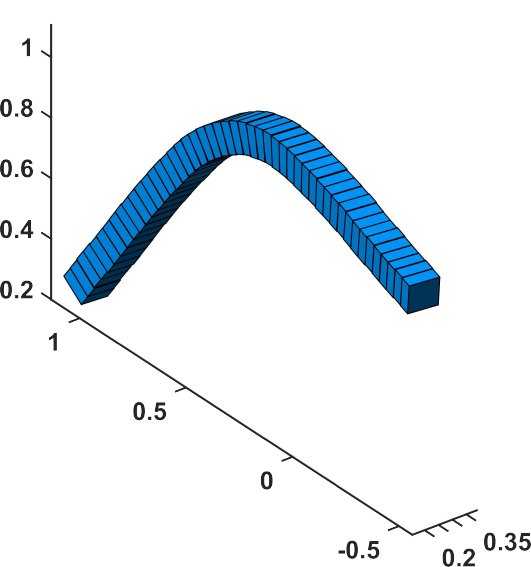}
    \end{subfigure}
    \hspace{2em}
    \begin{subfigure}[pos=t]{0.2465\textwidth}
        \centering
        \includegraphics[width=0.95\textwidth]{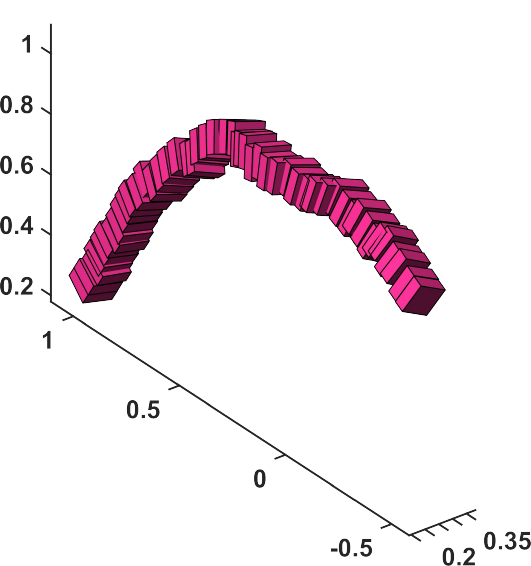}
    \end{subfigure}

    \begin{subfigure}[pos=t]{0.29\textwidth}
    \centering
    \begin{minipage}[pos=t]{0.1\textwidth}  
        \vspace{-3.5cm}
        \rotatebox{90}{$n=13500$}
    \end{minipage}
    \begin{minipage}[pos=t]{0.85\textwidth}   
        \centering
        \includegraphics[width=0.95\textwidth]{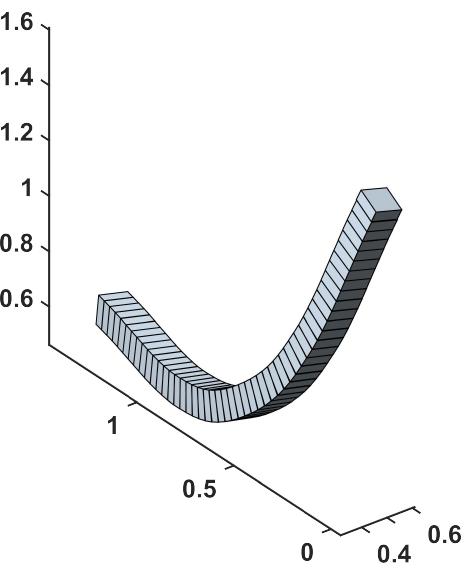}
    \end{minipage}
    \end{subfigure}
    \hspace{2em}
    \begin{subfigure}[pos=t]{0.2465\textwidth}
        \centering
        \includegraphics[width=0.95\textwidth]{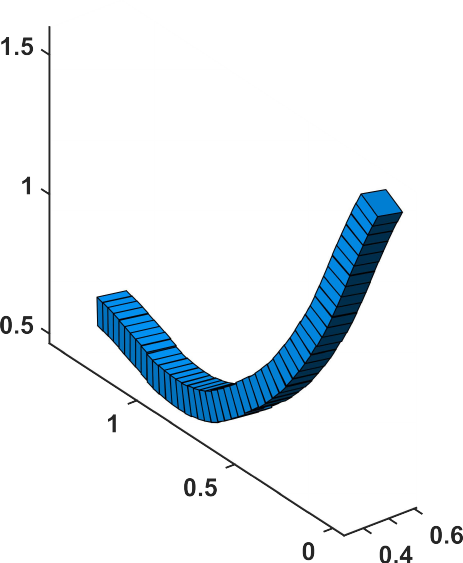}
    \end{subfigure}
    \hspace{2em}
    \begin{subfigure}[pos=t]{0.2465\textwidth}
        \centering
        \includegraphics[width=0.95\textwidth]{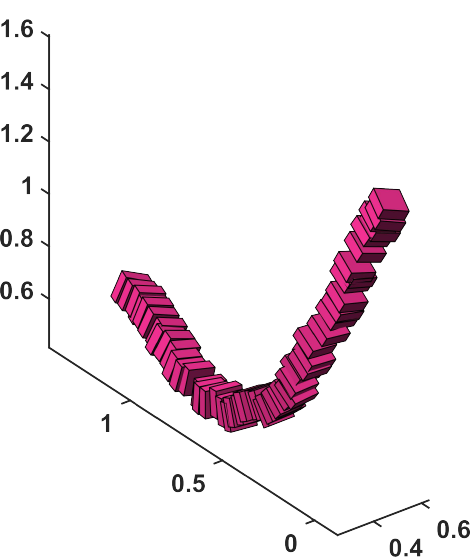}
    \end{subfigure}

    \begin{subfigure}[pos=t]{0.29\textwidth}
    \centering
    \begin{minipage}[pos=t]{0.1\textwidth}  
        \vspace{-4.5cm}
        \rotatebox{90}{$n=18000$}
    \end{minipage}
    \begin{minipage}[pos=t]{0.85\textwidth}   
        \centering
        \includegraphics[width=0.95\textwidth]{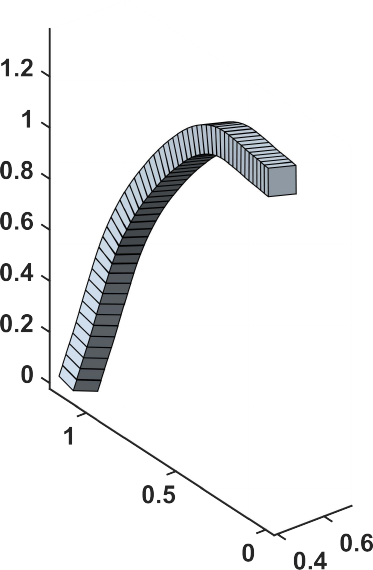}
    \end{minipage}
    \end{subfigure}
    \hspace{2em}
    \begin{subfigure}[pos=t]{0.2465\textwidth}
        \centering
        \includegraphics[width=0.95\textwidth]{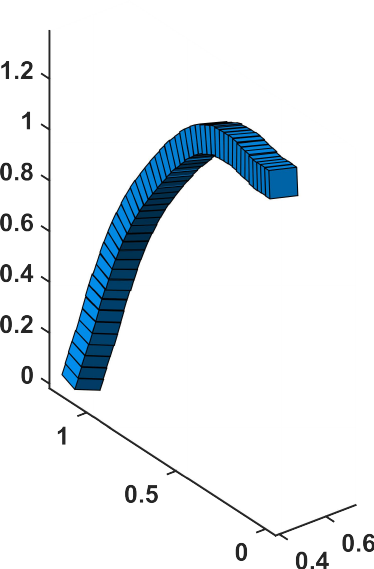}
    \end{subfigure}
    \hspace{2em}
    \begin{subfigure}[pos=t]{0.2465\textwidth}
        \centering
        \includegraphics[width=0.95\textwidth]{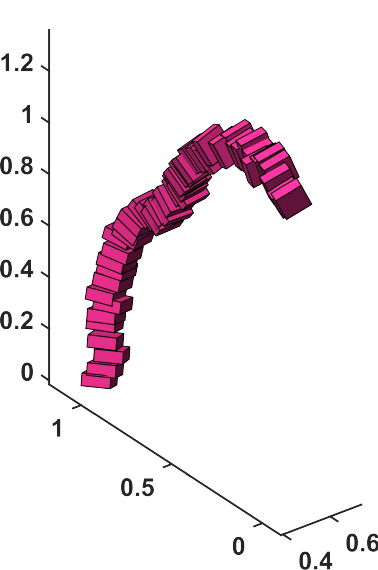}
    \end{subfigure}
    \caption{Snapshots of the beam at different iteration steps. \textbf{Left}: deterministic beam simulated by the numerical method in \cite{An2020}. \textbf{Middle}: stochastic beam simulated by the proposed stochastic variational integrators with $\Theta = 5\times 10^{-3}\boldsymbol{1}_{6\times 1}$ or $\Phi\left(g,\mu\right) = \int_{0}^{S}\left<\mu,5\times 10^{-3}\mathbf{1}_{6\times 1}\right> ds$. \textbf{Right}: stochastic beam simulated by the proposed stochastic variational integrators with $\Theta = 5\times 10^{-2}\boldsymbol{1}_{6\times 1}$ or $\Phi\left(g,\mu\right) = \int_{0}^{S}\left<\mu, 5\times 10^{-2}\mathbf{1}_{6\times 1}\right> ds$.}
    \label{Beam_Snapshots}
\end{figure*}
For the numerical simulation of a stochastic bending beam, the following parameters are chosen (see Table~\ref{beam_para}). These beam parameters have been introduced by \cite{Simo1985} and used in \cite{Simo1986}, \cite{Demoures2015}, and \cite{An2020}. 

\begin{table}[width=.685\linewidth,cols=5,pos=h]
\caption{Parameters of the beam.}
\label{beam_para}
\begin{tabular*}{\tblwidth}{@{} ccccc@{} }
\toprule
Mass density  & Section edge length & Poisson ratio   &  Length      & Young's modulus      \\
$\rho$        & $a$                 & $\nu$                 & $S$        & $E$           \\
\midrule
1000 kg/m     & 0.1 m               & 0.35       & 2$\pi$/3 m    & $10^{7}$ Pa     \\
\bottomrule
\end{tabular*}
\end{table}
The spatial discretization number of the beam is $J=50$, and therefore the spatial discretization step is $\Delta s = S/J = \frac{2\pi}{3J}$ $m$. The temporal discretization step is $\Delta t = 5\times 10^{-5}$ $s$. The total iteration number is $N=20000$. The beam is initiated with the positions of cross sections 
\begin{equation*}
b^{0}_{j}
=
\left(
0,
\cos\left(\frac{\pi}{6}+j\Delta s\right),
\sin\left(\frac{\pi}{6}+j\Delta s\right)
\right),
\quad 
j=1,...,J,
\end{equation*}
the initial velocities
\begin{equation*}
\zeta^{1/2}_{j}
=
\left(
0,0,0,\sin\frac{j\Delta s\pi}{S},1,0
\right),
\quad
j=1,...,J,
\end{equation*}
and the boundary conditions $\gamma^{n}_{j+1/2} = 0$, $j=1,...,J$.
\begin{figure*}[pos=htpb]
    \begin{subfigure}[t]{0.36\textwidth}
    \centering
    \begin{minipage}[t]{0.1\textwidth}  
        \vspace{1cm}
        \rotatebox{90}{Energy performance}
    \end{minipage}
    \begin{minipage}[t]{0.85\textwidth}   
        \centering
        \caption*{\hspace{1.8em}\textbf{Deterministic}}
        \includegraphics[width=\textwidth]{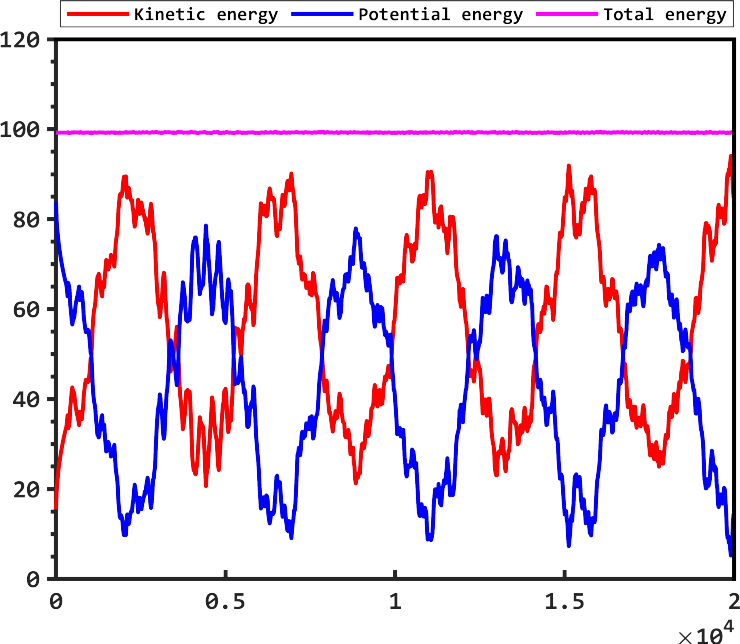}
    \end{minipage}
    \end{subfigure}
    \hfill
    \begin{subfigure}[t]{0.306\textwidth}
        \caption*{\textbf{Stochastic}: $\Theta = 5\times 10^{-3}\cdot\boldsymbol{1}_{6\times 1}$}
        \centering
        \includegraphics[width=\textwidth]{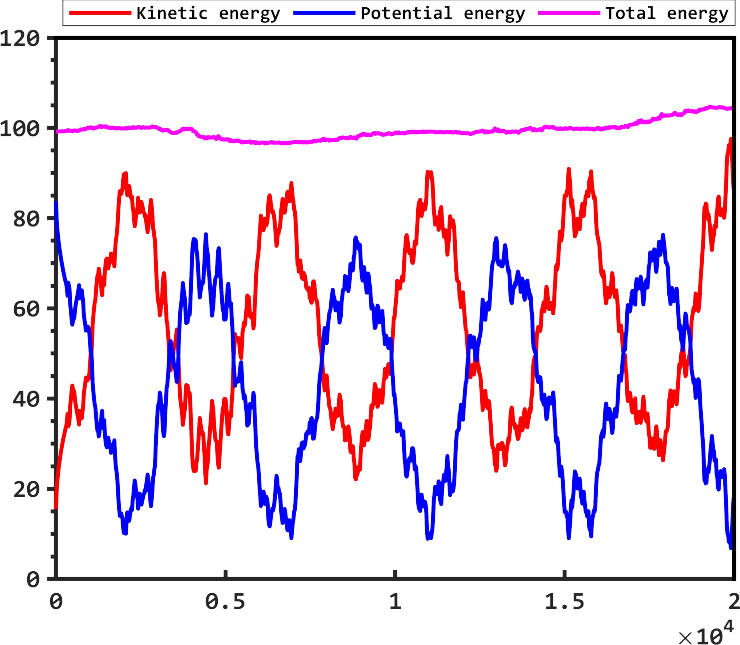}
    \end{subfigure}
    \hfill
    \begin{subfigure}[t]{0.306\textwidth}
        \caption*{\textbf{Stochastic}: $\Theta = 5\times 10^{-2}\cdot\boldsymbol{1}_{6\times 1}$}
        \centering
        \includegraphics[width=\textwidth]{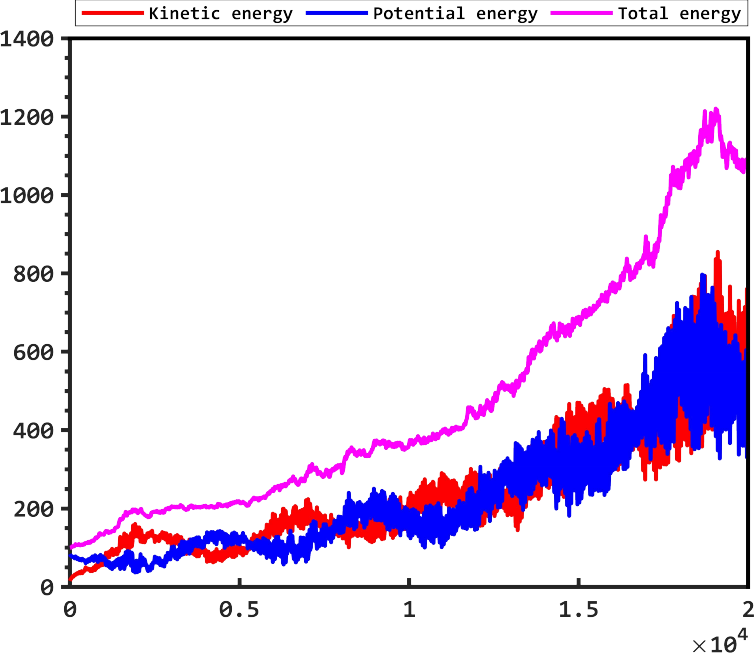}
    \end{subfigure}

    \begin{subfigure}[t]{0.36\textwidth}
    \centering
    \begin{minipage}[t]{0.1\textwidth}  
        \vspace{-3.85cm}
        \rotatebox{90}{Angular momentum}
    \end{minipage}
    \begin{minipage}[t]{0.85\textwidth}   
        \centering
        \includegraphics[width=\textwidth]{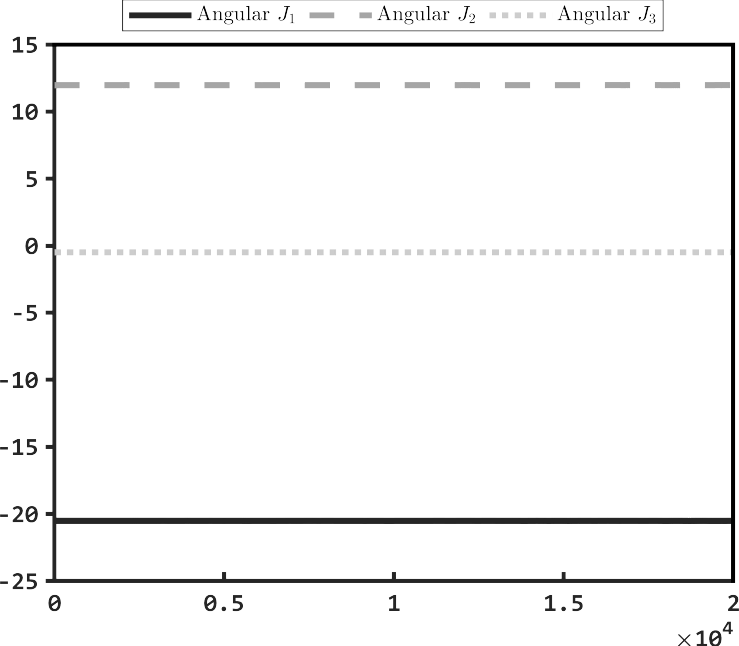}
    \end{minipage}
    \end{subfigure}
    \hfill
    \begin{subfigure}[t]{0.306\textwidth}
        \centering
        \includegraphics[width=\textwidth]{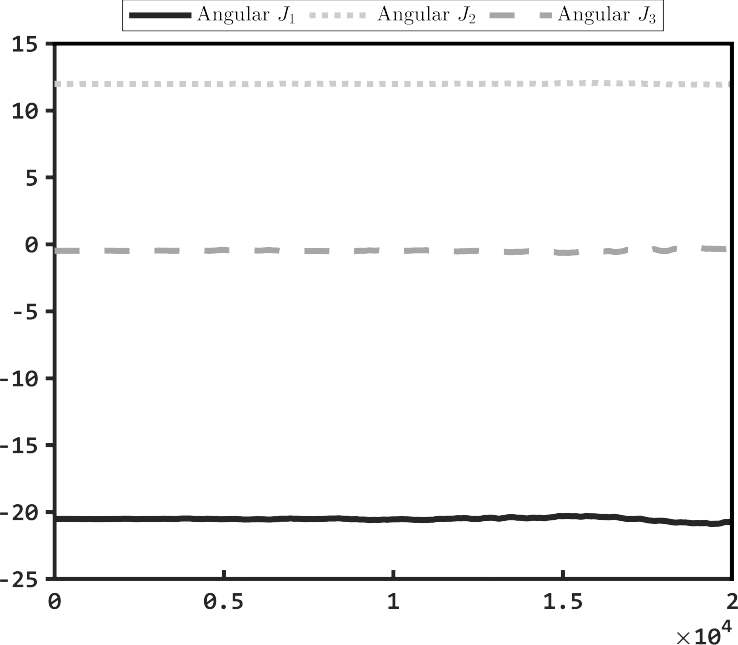}
    \end{subfigure}
    \hfill
    \begin{subfigure}[t]{0.306\textwidth}
        \centering
        \includegraphics[width=\textwidth]{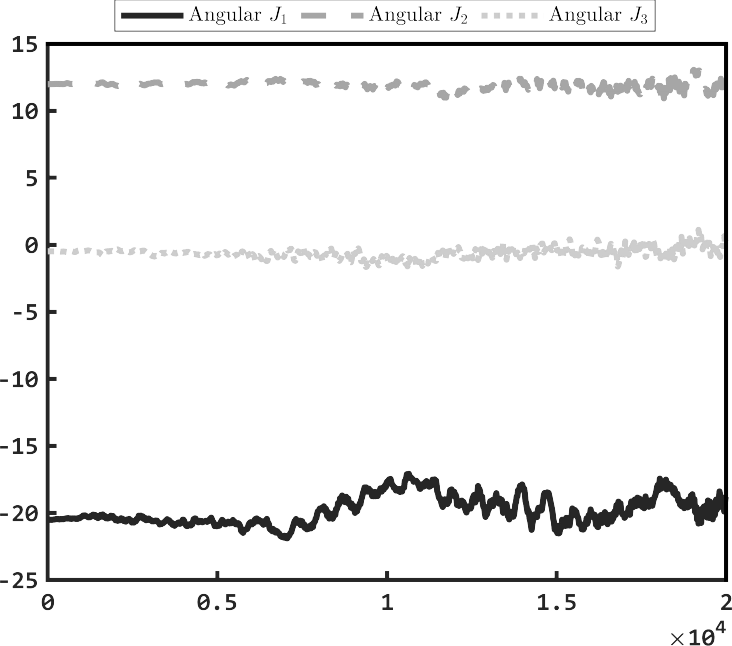}
    \end{subfigure}

    \begin{subfigure}[t]{0.36\textwidth}
    \centering
    \begin{minipage}[t]{0.1\textwidth}  
        \vspace{-3.8cm}
        \rotatebox{90}{Linear momentum}
    \end{minipage}
    \begin{minipage}[t]{0.85\textwidth}   
        \centering
        \includegraphics[width=\textwidth]{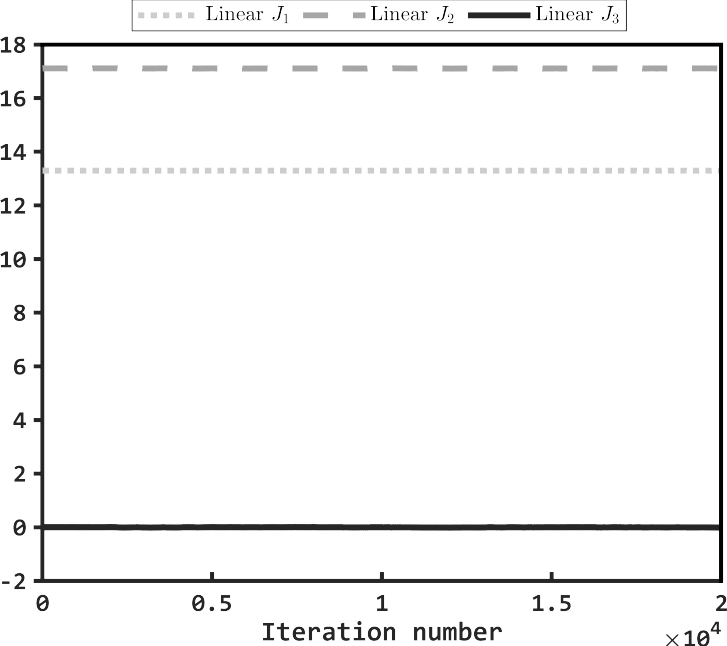}
    \end{minipage}
    \end{subfigure}
    \hfill
    \begin{subfigure}[t]{0.306\textwidth}
        \centering
        \includegraphics[width=\textwidth]{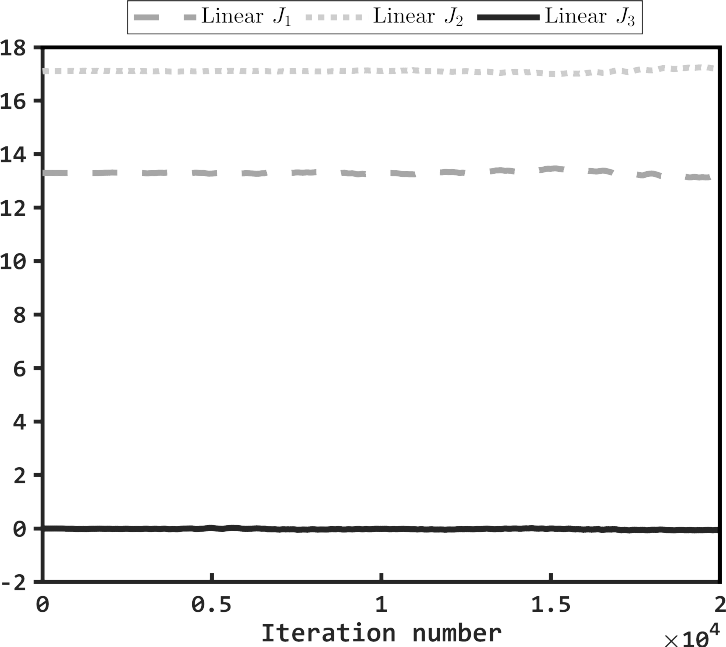}
    \end{subfigure}
    \hfill
    \begin{subfigure}[t]{0.306\textwidth}
        \centering
        \includegraphics[width=\textwidth]{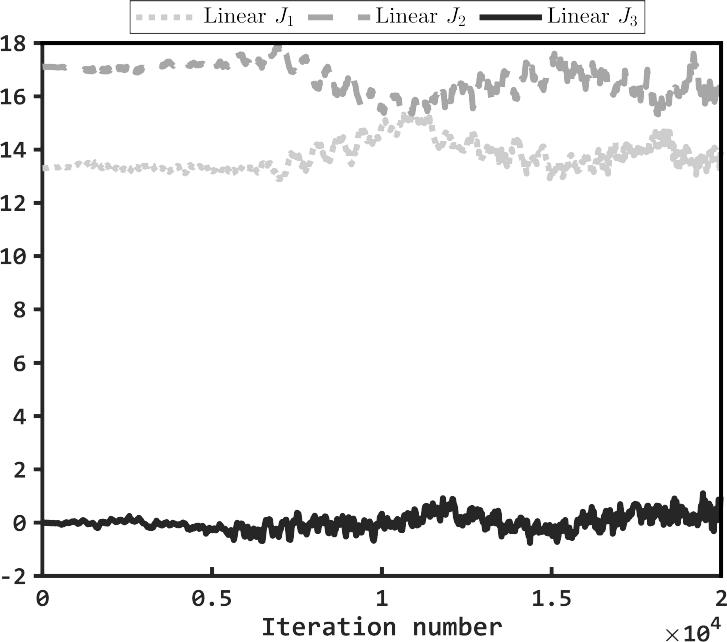}
    \end{subfigure}
    \caption{\textbf{Top}: Energy: the total, kinetic, and potential energy. \textbf{Middle}: Angular momentum. \textbf{Bottom}: Linear momentum.}
    \label{Ene_Mom}
\end{figure*}
The stochastic perturbations are described by the diffusivity vector fields 
\begin{align*}
\Theta = 0.005\boldsymbol{1}_{6\times 1}\in\mathfrak{se}(3)
\quad
\text{and}
\quad
\Theta = 0.05\boldsymbol{1}_{6\times 1}\in\mathfrak{se}(3)
\end{align*}
for the Type-I stochastic integrator introduced in \S~\ref{d_ham}, and the stochastic potentials 
\begin{align}
\Phi\left(g,\mu\right) = \int_{0}^{S}\left<\mu, 0.005\mathbf{1}_{6\times 1}\right> ds
\quad
\text{and}
\quad
\Phi\left(g,\mu\right) = \int_{0}^{S}\left<\mu, 0.05\mathbf{1}_{6\times 1}\right> ds
\label{potentials_bending_beam}
\end{align}
for the Type-II stochastic integrator introduced in \S~\ref{d_hph}.

The simulated trajectories of the stochastic bending beam are shown in Fig.~\ref{Beam_Snapshots}. When we choose $\Theta = 0$ and $\Phi = 0$, the proposed integrators specialize to the deterministic ones. In this case, it is seen from the first column of Fig.~\ref{Ene_Mom} that the discrete total energy
\begin{equation*}
E\left(\zeta^{n+1/2}_{j}, \gamma^{n}_{j+1/2}\right) =
\sum_{j=1}^{J}\frac{\Delta s}{2}\left(\left<\zeta^{n+1/2}_{j},\mathcal{K}\zeta^{n+1/2}_{j}\right>
+ \left<\gamma^{n}_{j+1/2},\mathcal{P}\gamma^{n}_{j+1/2}\right>\right)
\end{equation*}
is nearly preserved by the proposed variational discretizations. The proposed integrators in \S~\ref{Exact_VI_1} and \S~\ref{d_hph} belong to the class of \textit{approximate variational integrators}, which not only approximate the action integral by applying a discrete variational principle as \textit{variational integrators} do, but also additionally approximate variations. Hence, not being fully variational, the approximate variational integrators in general do not inherit the exact preservation in momentum
\begin{equation*}
M\left(g^{n}_{j},\zeta^{n+1/2}_{j}\right) =
\sum_{j=1}^{J}\Delta s\mathrm{Ad}^{*}_{\left(g^{n}_{j}\right)^{-1}}\mathcal{K}\zeta^{n+1/2}_{j}.
\end{equation*}
However, it is a pleasant surprise that a good \textit{numerical} momentum preservation has been observed in simulations (the last two figures in the first column of Fig.~\ref{Ene_Mom}). In the stochastic setting, the momentum of the stochastic beam varies over time, which exactly reflects the fact that the system is influenced by the stochasticity (the last two rows of Fig.~\ref{Ene_Mom}). Furthermore, since the potentials \eqref{potentials_bending_beam} do not Poisson-commute with the Hamiltonian, the total energy is perturbed as shown in the first row of Fig.~\ref{Ene_Mom}.

\subsection{Stochastic flying spaghetti}
Here we give an interesting example of a flying stochastic geometrically exact beam, which is a stochastic extension of the well-known deterministic flying spaghetti \citep{JS_LV_1988}. The initial configuration of the spaghetti is given by the first plot in Fig.~\ref{initial_spaghetti}. The spaghetti is initiated by an impulse described by the generalized force $\left\{0,M_{y},M_{z},F_{x},0,0\right\}$, where $F_{x} = M_{z}/10$, $M_{y} = M_{z}/2$, and $M_{z}$ is given by the function shown in the second plot in Fig.~\ref{initial_spaghetti}.
\begin{figure*}[pos=htpb]
    \centering
    \begin{overpic}[width = 0.2\textwidth]{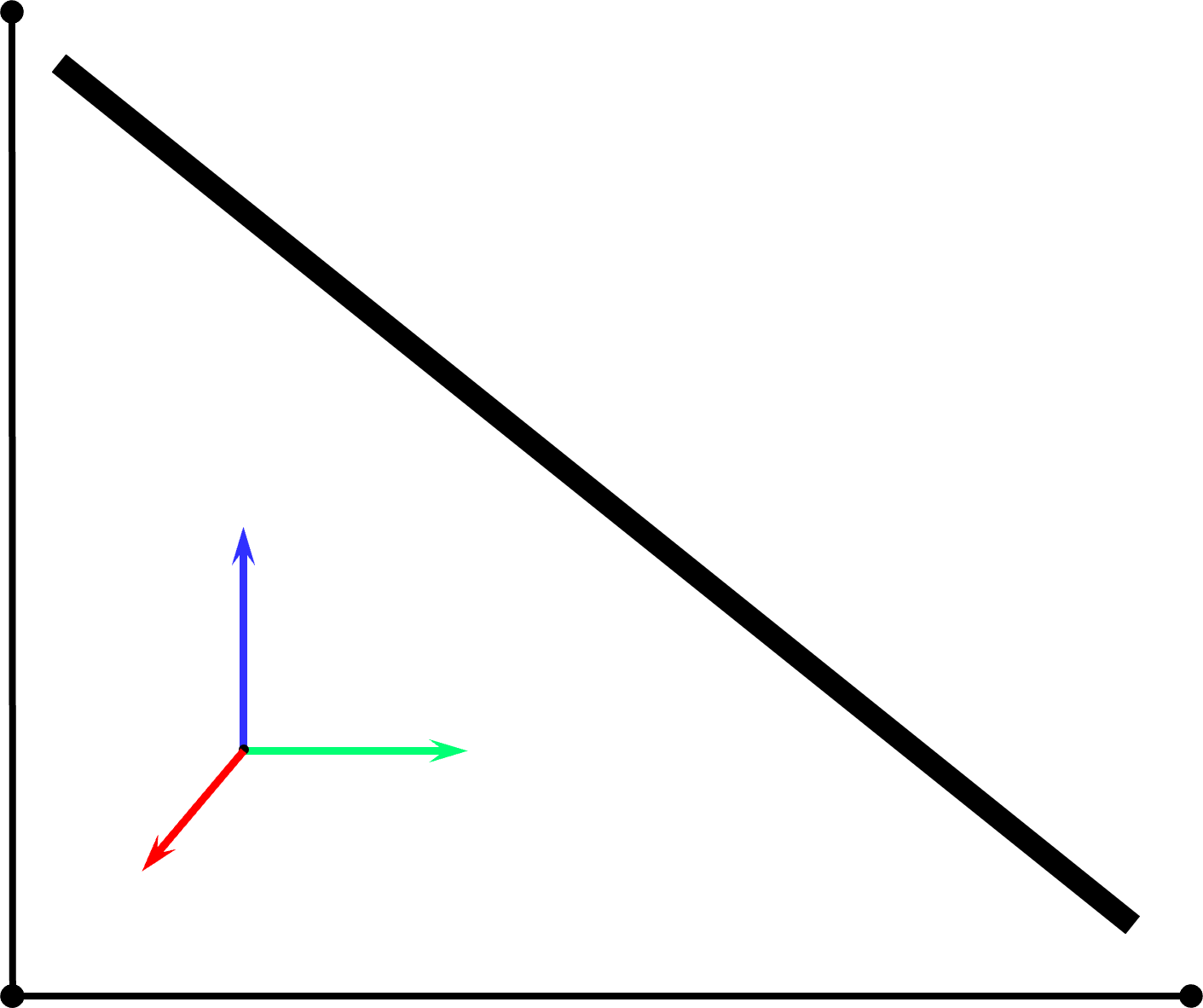}
        \put(8,6){$z$}
        \put(18,43){$y$}
        \put(39,19.5){$x$}
        \put(39,-8.5){$x=6$}
        \put(-25,33){$y=8$}
    \end{overpic}
    \hspace{5em}
    \begin{subfigure}{0.3\textwidth}
        \includegraphics[width=\textwidth]{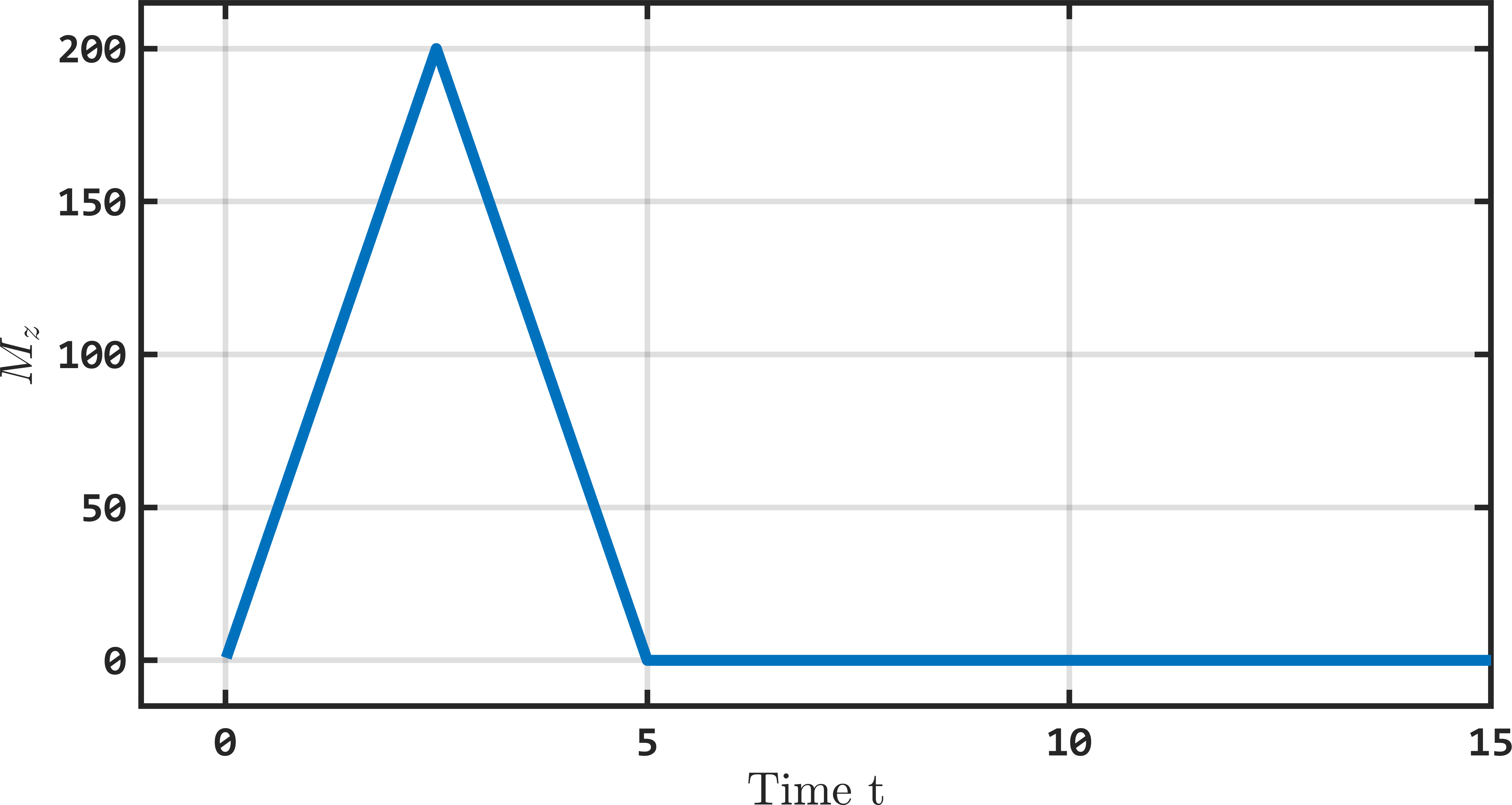}
    \end{subfigure}
    \vspace{0.5em}
    \caption{\textbf{Left}: Initial configuration of the spaghetti. \textbf{Right}: Function $M_z$.}
    \label{initial_spaghetti}
\end{figure*}

The stochasticity is described by the following diffusivity vector fields, respectively,
\begin{equation*}
\Theta = 10^{-2}\cdot\boldsymbol{1}_{6\times 1}\in\mathfrak{se}(3), 
\quad
3\times10^{-2}\cdot\boldsymbol{1}_{6\times 1}\in\mathfrak{se}(3), 
\quad
\text{and}
\quad
5\times10^{-2}\cdot\boldsymbol{1}_{6\times 1}\in\mathfrak{se}(3)
\end{equation*}
for the Type-I stochastic integrator, or equivalently by the stochastic potentials
\begin{equation*}
\Phi(g,\mu) = \int_{0}^{S}\left<\mu, 10^{-2}\cdot\boldsymbol{1}_{6\times 1}\right>ds,
\quad
\int_{0}^{S}\left<\mu, 3\times 10^{-2}\cdot\boldsymbol{1}_{6\times 1}\right> ds,
\quad
\text{and}
\quad
\int_{0}^{S}\left<\mu, 5\times 10^{-2}\cdot\boldsymbol{1}_{6\times 1}\right> ds
\end{equation*}
for the Type-II stochastic integrator. The resulting motion of the spaghetti is illustrated in the first row of Fig.~\ref{flying-spaghetti}, showing large deformations and rotations. The second row of Fig.~\ref{flying-spaghetti} shows that the proposed integrators exhibit excellent performance in preserving the Lie group structure of the beam in both deterministic and stochastic settings, with the accuracy order of $\mathcal{O}(10^{-13})$.

\begin{figure*}[pos=htpb]
    \begin{subfigure}[t]{0.55\textwidth}
    \centering
    \begin{minipage}[t]{0.05\textwidth}  
        \vspace{1.1cm}
        \rotatebox{90}{\textbf{Deterministic}}
    \end{minipage}
    \begin{minipage}[t]{0.92\textwidth}   
        \centering
        \caption*{\hspace{1.8em}Trajectories of the flying spaghetti}
        \includegraphics[width=\textwidth]{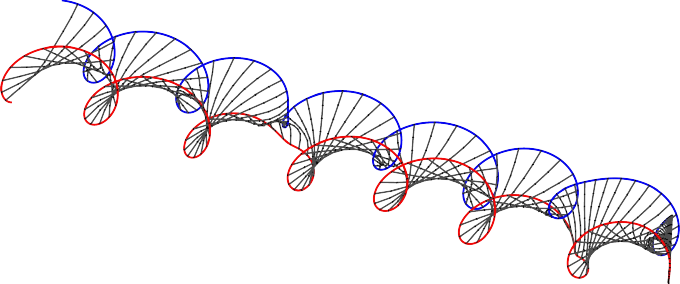}
    \end{minipage}
    \end{subfigure}
    \hfill
    \begin{subfigure}[t]{0.4\textwidth}
        \caption*{Preservation in Lie group SE(3) structure}
        \centering
        \includegraphics[width=\textwidth]{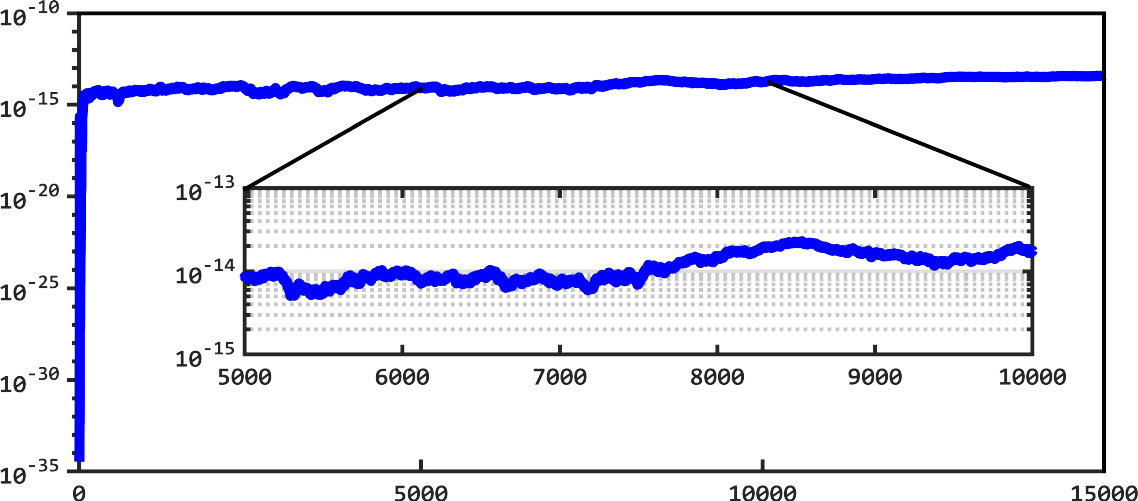}
    \end{subfigure}

    \vspace{0.8em}
    \begin{subfigure}[t]{0.55\textwidth}
    \centering
    \begin{minipage}[t]{0.05\textwidth}  
        \vspace{-3.4cm}
        \rotatebox{90}{\textbf{Sto.:} $\boldsymbol{\Theta} = 10^{-2}\cdot\boldsymbol{1}_{6\times 1}$}
    \end{minipage}
    \begin{minipage}[t]{0.92\textwidth}   
        \centering
        \includegraphics[width=\textwidth]{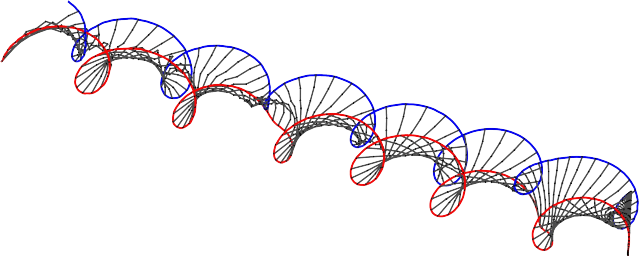}
    \end{minipage}
    \end{subfigure}
    \hfill
    \begin{subfigure}[t]{0.4\textwidth}
        \centering
        \includegraphics[width=\textwidth]{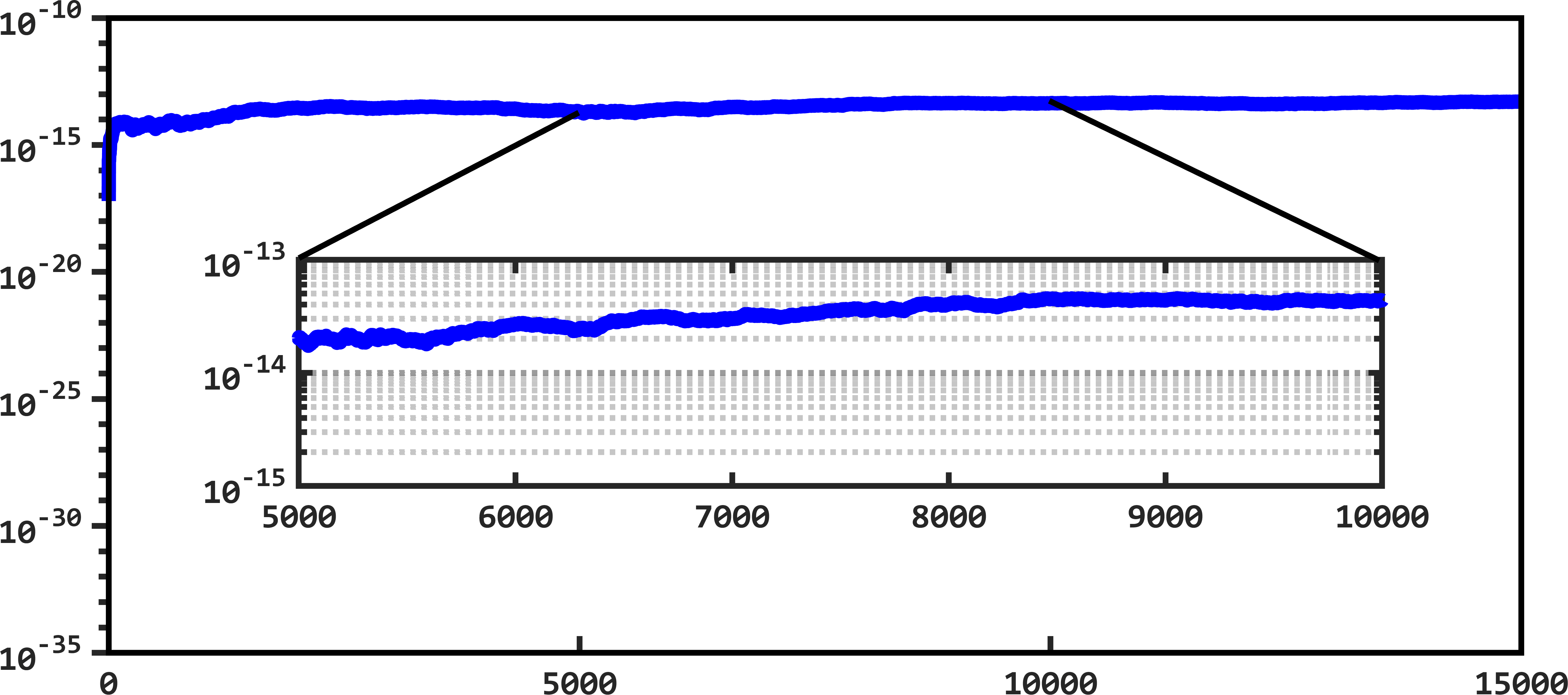}
    \end{subfigure}

    \vspace{0.8em}
    \begin{subfigure}[t]{0.55\textwidth}
    \centering
    \begin{minipage}[t]{0.05\textwidth}  
        \vspace{-3.6cm}
        \rotatebox{90}{\textbf{Sto.:} $\boldsymbol{\Theta} = 3\times 10^{-2}\cdot\boldsymbol{1}_{6\times 1}$}
    \end{minipage}
    \begin{minipage}[t]{0.92\textwidth}   
        \centering
        \includegraphics[width=\textwidth]{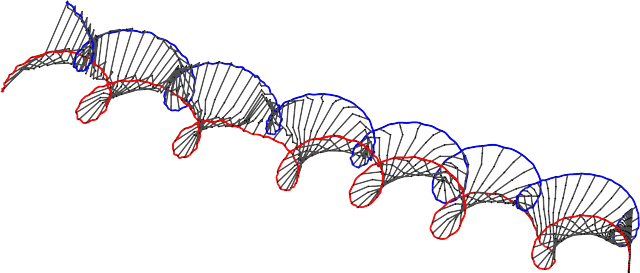}
    \end{minipage}
    \end{subfigure}
    \hfill
    \begin{subfigure}[t]{0.4\textwidth}
        \centering
        \includegraphics[width=\textwidth]{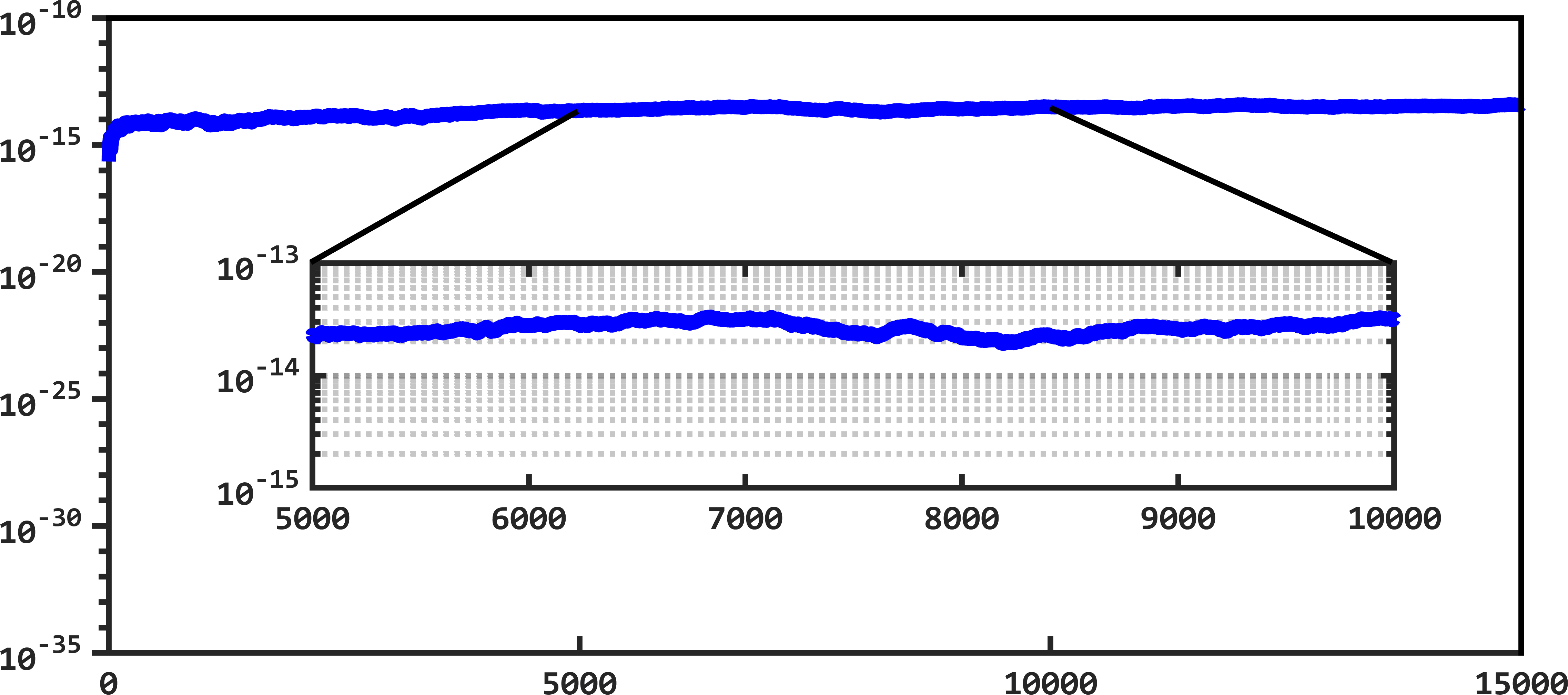}
    \end{subfigure}

    \vspace{0.8em}
    \begin{subfigure}[t]{0.55\textwidth}
    \centering
    \begin{minipage}[t]{0.05\textwidth}  
        \vspace{-3.5cm}
        \rotatebox{90}{\textbf{Sto.:} $\boldsymbol{\Theta} = 5\times 10^{-2}\cdot\boldsymbol{1}_{6\times 1}$}
    \end{minipage}
    \begin{minipage}[t]{0.92\textwidth}   
        \centering
        \includegraphics[width=\textwidth]{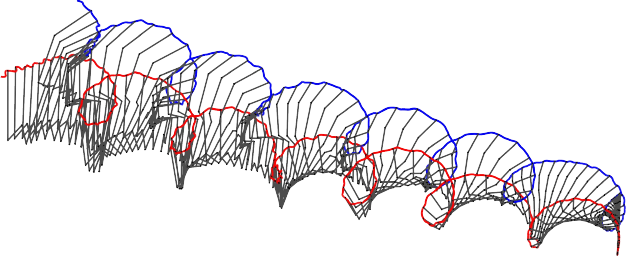}
    \end{minipage}
    \end{subfigure}
    \hfill
    \begin{subfigure}[t]{0.4\textwidth}
        \centering
        \includegraphics[width=\textwidth]{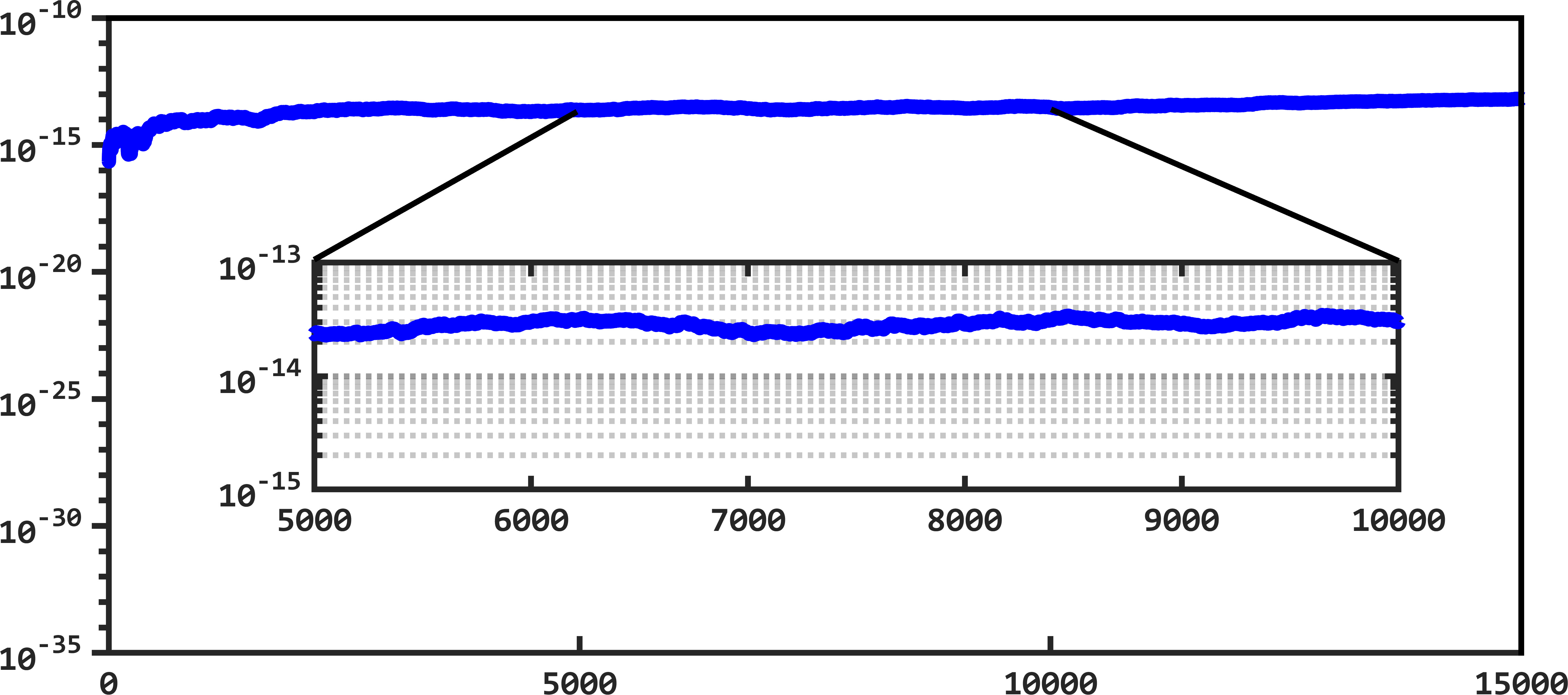}
    \end{subfigure}
    
    \caption{\textbf{Left}: Trajectories of the stochastic flying spaghetti. \textbf{Right}: Structure-preserving property of the proposed stochastic integrators.}
    \label{flying-spaghetti}
\end{figure*}

\section{Conclusions}
This work presents variational formulations for modeling stochastic geometrically exact beams, in both continuous- and discrete-time settings. The key ingredient of this approach is to first unify perturbations and system dynamics from a variational perspective, and then construct the associated stochastic geometric integrators by discretizing variational principles. By doing so, the obtained two types of stochastic integrators exhibit structure-preserving properties, such as nearly energy conservation, momentum preservation, and Lie group structure preservation. The effectiveness of the proposed approach is demonstrated with various examples. Furthermore, the proposed integrators are shown to be stochastic extensions of the well-known staggered leapfrog method and Yee's scheme in specific cases. However, there are limitations of the proposed approach. In their current form, the proposed integrators are applicable to stochastic unconstrained field-theoretic systems with a single spatial variable. Future research will focus on removing the above limitations, extending the framework to stochastic integrators for higher-dimensional field-theoretic systems, and incorporating constraints to account for material perturbations of constrained deformable bodies.

\printcredits

\appendix
\section{Proof of Variations \eqref{var_zeta_2} and \eqref{var_gamma_2}}\label{App_var}
\begin{proof}
It follows from \eqref{dis_kin_eq} that
\begin{align*}
\Delta t\zeta^{n+1/2}_{j} + \Theta^{n+1/2}_{j}
\left(W^{(M)}_{n+1,j} - W^{(M)}_{n,j}\right)
=
\tau^{-1}\left(\left(g^{n}_{j}\right)^{-1}g^{n+1}_{j}\right).
\end{align*}
Taking variations on both sides and using the definition of the right trivialized derivative yield
\begin{align*}
\Delta t\zeta^{n+1/2}_{j}
&=
\delta \left[
\tau^{-1}\left(\left(g^{n}_{j}\right)^{-1}g^{n+1}_{j}\right)
\right]  \\
&= D\tau^{-1}\left(\left(g^{n}_{j}\right)^{-1}g^{n+1}_{j}\right)\cdot
\delta\left(\left(g^{n}_{j}\right)^{-1}g^{n+1}_{j}\right)  \\
&= \mathrm{d}_{\tau^{-1}\left(\left(g^{n}_{j}\right)^{-1}g^{n+1}_{j}\right)}\tau^{-1}
\left(
\delta\left(\left(g^{n}_{j}\right)^{-1}g^{n+1}_{j}\right)\cdot
\left(g^{n+1}_{j}\right)^{-1}g^{n}_{j}
\right),
\end{align*}
with
\begin{align*}
\delta\left(\left(g^{n}_{j}\right)^{-1}g^{n+1}_{j}\right)\cdot
\left(g^{n+1}_{j}\right)^{-1}g^{n}_{j}
&= - \left(g^{n}_{j}\right)^{-1}\delta g^{n}_{j}
+ \mathrm{Ad}^{*}_{\left(g^{n}_{j}\right)^{-1}g^{n+1}_{j}}\left(g^{n+1}_{j}\right)^{-1}\delta g^{n+1}_{j}  \\
&= - \eta^{n}_{j}
+ \mathrm{Ad}^{*}_{\left(g^{n}_{j}\right)^{-1}g^{n+1}_{j}}\eta^{n+1}_{j},
\end{align*}
where the last equality follows from the definition of $\eta^{n}_{j}$. Discrete variation \eqref{var_gamma_2} can be proved in a similar way.
\end{proof}

\section{Proof of Theorem~\ref{theorem_exact_integrator}}\label{proof_exact_integrator}
\begin{proof}
Taking the variation of the discrete action sum \eqref{discrete_act_sum_1} and using Prop.~\ref{exact_variations} yield
\begin{align*}
\delta A_{d} &= \Delta t\Delta s\sum_{n=0}^{N-1}\sum_{j=0}^{J-1}
\left[
\left<\frac{\delta\ell}{\delta g}\left(g^{n}_{j},\zeta^{n+1/2}_{j},\gamma^{n}_{j+1/2}\right),
\delta g^{n}_{j}\right>
+ \left<\frac{\delta\ell}{\delta \zeta}\left(g^{n}_{j},\zeta^{n+1/2}_{j},\gamma^{n}_{j+1/2}\right), \delta \zeta^{n+1/2}_{j}\right>  \right.\\
&\qquad\qquad\qquad\qquad\left.
+ \left<\frac{\delta\ell}{\delta \gamma}\left(g^{n}_{j},\zeta^{n+1/2}_{j},\gamma^{n}_{j+1/2}\right), \delta \gamma^{n}_{j+1/2}\right>
\right]  \\
&= \sum_{n=0}^{N-1}\sum_{j=0}^{J-1}
\left[
\Delta s\Delta t\left<T_{e}^{\ast}L_{g^{n}_{j}}\frac{\delta\ell}{\delta g}\left(g^{n}_{j},\zeta^{n+1/2}_{j},\gamma^{n}_{j+1/2}\right),
\eta^{n}_{j}\right>  \right.\\
&\qquad\qquad\quad
+ \Delta s\left<\frac{\delta\ell}{\delta \zeta}\left(g^{n}_{j},\zeta^{n+1/2}_{j},\gamma^{n}_{j+1/2}\right), 
\mathrm{d}_{\tau^{-1}\left(\left(g^{n}_{j}\right)^{-1}g^{n+1}_{j}\right)}\tau^{-1}
\cdot
\left(
-\eta^{n}_{j} + \mathrm{Ad}_{\left(g^{n}_{j}\right)^{-1}g^{n+1}_{j}}\eta^{n+1}_{j}
\right)\right>  \\
&\qquad\qquad\quad\left.
+
\Delta t\left<\frac{\delta\ell}{\delta \gamma}\left(g^{n}_{j},\zeta^{n+1/2}_{j},\gamma^{n}_{j+1/2}\right), 
\mathrm{d}_{\tau^{-1}\left(\left(g^{n}_{j}\right)^{-1}g^{n}_{j+1}\right)}\tau^{-1}
\cdot
\left(
-\eta^{n}_{j} + \mathrm{Ad}_{\left(g^{n}_{j}\right)^{-1}g^{n}_{j+1}}\eta^{n}_{j+1}
\right)\right>
\right]  \\
&= \sum_{n=1}^{N-1}\sum_{j=1}^{J-1}
\left[
\Delta s\Delta t\left<
T_{e}^{*}L_{g^{n}_{j}}\frac{\delta\ell}{\delta g}\left(g^{n}_{j},\zeta^{n+1/2}_{j},\gamma^{n}_{j+1/2}\right),
\eta_{j}^{n}
\right>\right.  \\
&\qquad\qquad\quad
+ \Delta s\left<
-\left[\mathrm{d}_{\tau^{-1}\left(\left(g^{n}_{j}\right)^{-1}g^{n+1}_{j}\right)}\tau^{-1}\right]^{*}
\frac{\delta\ell}{\delta \zeta}\left(g^{n}_{j},\zeta^{n+1/2}_{j},\gamma^{n}_{j+1/2}\right),
\eta_{j}^{n}
\right>  \\
&\qquad\qquad\quad
+ \Delta s\left<
\mathrm{Ad}^{*}_{\left(g^{n-1}_{j}\right)^{-1}g^{n}_{j}}
\left[\mathrm{d}_{\tau^{-1}\left(\left(g^{n-1}_{j}\right)^{-1}g^{n}_{j}\right)}\tau^{-1}\right]^{*}
\frac{\delta\ell}{\delta \zeta}\left(g^{n-1}_{j},\zeta^{n-1/2}_{j},\gamma^{n-1}_{j+1/2}\right),
\eta_{j}^{n}
\right>  \\
&\qquad\qquad\quad
+ \Delta t\left<
-\left[\mathrm{d}_{\tau^{-1}\left(\left(g^{n}_{j}\right)^{-1}g^{n}_{j+1}\right)}\tau^{-1}\right]^{*}
\frac{\delta\ell}{\delta \gamma}\left(g^{n}_{j},\zeta^{n+1/2}_{j},\gamma^{n}_{j+1/2}\right),
\eta_{j}^{n}
\right>  \\
&\qquad\qquad\quad\left.
+ \Delta t\left<
\mathrm{Ad}^{*}_{\left(g^{n}_{j-1}\right)^{-1}g^{n}_{j}}
\left[\mathrm{d}_{\tau^{-1}\left(\left(g^{n}_{j-1}\right)^{-1}g^{n}_{j}\right)}\tau^{-1}\right]^{*}
\frac{\delta\ell}{\delta \zeta}\left(g^{n}_{j-1},\zeta^{n+1/2}_{j-1},\gamma^{n}_{j-1/2}\right),
\eta_{j}^{n}
\right>  
\right]  \\
&\quad
+\sum_{j=0}^{J-1}
\Delta s\left<
-\left[\mathrm{d}_{\tau^{-1}\left(\left(g^{0}_{j}\right)^{-1}g^{1}_{j}\right)}\tau^{-1}\right]^{*}
\frac{\delta\ell}{\delta \zeta}\left(g^{0}_{j},\zeta^{1/2}_{j},\gamma^{0}_{j+1/2}\right),
\eta_{j}^{0}
\right>  \\
&\quad
+ \sum_{j=0}^{J-1}
\Delta s\left<
\mathrm{Ad}^{*}_{\left(g^{N-1}_{j}\right)^{-1}g^{N}_{j}}
\left[\mathrm{d}_{\tau^{-1}\left(\left(g^{N-1}_{j}\right)^{-1}g^{N}_{j}\right)}\tau^{-1}\right]^{*}
\frac{\delta\ell}{\delta \zeta}\left(g^{N-1}_{j},\zeta^{N-1/2}_{j},\gamma^{N-1}_{j+1/2}\right),
\eta_{j}^{N}
\right>  \\
&\quad
+\sum_{n=1}^{N-1}
\Delta t\left<
-\left[\mathrm{d}_{\tau^{-1}\left(\left(g^{n}_{0}\right)^{-1}g^{n}_{1}\right)}\tau^{-1}\right]^{*}
\frac{\delta\ell}{\delta\gamma}\left(g^{n}_{0},\zeta^{n+1/2}_{0},\gamma^{n}_{1/2}\right),
\eta^{n}_{0}
\right>  \\
&\quad
+ \sum_{n=1}^{N-1}
\Delta t\left<
\mathrm{Ad}^{*}_{\left(g^{n}_{J-1}\right)^{-1}g^{n}_{J}}
\left[\mathrm{d}_{\tau^{-1}\left(\left(g^{n}_{J-1}\right)^{-1}g^{n}_{J}\right)}\tau^{-1}\right]^{*}
\frac{\delta\ell}{\delta\gamma}\left(g^{n}_{J-1},\zeta^{n+1/2}_{J-1},\gamma^{n}_{J-1/2}\right),
\eta_{J}^{n}
\right>.
\end{align*}
From the fixed point condition $\eta^{0}_{j} = \eta^{N}_{j} = 0$ for $j=1,\dots,J$, we obtain the discrete equations \eqref{without_truncation_exact}
and the discrete boundary conditions \eqref{Exact_bc}.
\end{proof}

\section{Proof of Proposition~\ref{app_var_formulae}}\label{app_variation}
\begin{proof}
Recall that the adjoint action map admits a series expansion \citep{hall2015lie}
\begin{equation*}
\mathrm{Ad}_{\tau(X)} = I + \mathrm{ad}_{X} + \frac{1}{2}\mathrm{ad}^{2}_{X} + \mathcal{O}(|X|^{3}),
\end{equation*}
where $\tau$ is the exponential map.
Thus, in the exact variation formula \eqref{var_zeta_2} the term
\begin{align*}
&-\eta^{n}_{j} + \mathrm{Ad}_{\left(g^{n}_{j}\right)^{-1}g^{n+1}_{j}}\eta^{n+1}_{j}  \\
&= -\eta^{n}_{j} + \mathrm{Ad}_{\tau\left(\Delta t\zeta^{n+1/2}_{j} + \Theta^{n+1/2}_{j}
\left(
W^{(M)}_{n+1,j} - W^{(M)}_{n,j}
\right)
\right)}\eta^{n+1}_{j}  \\
&= -\eta^{n}_{j} + \eta^{n+1}_{j} 
+ \mathrm{ad}_{\Delta t\zeta^{n+1/2}_{j} + \Theta^{n+1/2}_{j}
\left(
W^{(M)}_{n+1,j} - W^{(M)}_{n,j}
\right)}\eta^{n+1}_{j}
+ \frac{1}{2}\mathrm{ad}^{2}_{\Delta t\zeta^{n+1/2}_{j} + \Theta^{n+1/2}_{j}
\left(
W^{(M)}_{n+1,j} - W^{(M)}_{n,j}
\right)}\eta^{n+1}_{j}
+ \mathcal{O}\left(\Delta t^{3/2}\right),
\end{align*}
where the first equality uses \eqref{dis_kin_eq}. On the other hand, since the inverse right trivialized tangent admits an expansion \citep{Iserles2000lie,hairer2006,Bou_Rabee_Marsden_2009}
\begin{equation*}
\mathrm{d}_{X}\tau^{-1} = I - \frac{1}{2}\mathrm{ad}_{X} + \mathcal{O}\left(|X|^{2}\right),
\end{equation*} 
the exact variation formula \eqref{var_zeta_2} becomes
\begin{align*}
\delta \left[\tau^{-1}\left(\left(g^{n}_{j}\right)^{-1}g^{n+1}_{j}\right)\right]
&= \delta\zeta^{n+1/2}_{j}\Delta t 
= \mathrm{d}_{\tau^{-1}\left(\left(g^{n}_{j}\right)^{-1}g^{n}_{j+1}\right)}\tau^{-1}
\cdot
\left(
-\eta^{n}_{j} + \mathrm{Ad}_{\left(g^{n}_{j}\right)^{-1}g^{n+1}_{j}}\eta^{n+1}_{j}
\right)  \\
\nonumber &=: \mathrm{d}_{V}\tau^{-1}\cdot A  \\
\nonumber &= A - \frac{1}{2}\mathrm{ad}_{V}A + \mathcal{O}\left(\Delta t^{3/2}\right)  \\
\nonumber &= A - \frac{1}{2}\mathrm{ad}_{V}
\left(
- \eta^{n}_{j} + \eta^{n+1}_{j} + \mathrm{ad}_{V}\eta^{n+1}_{j}
\right)
+ \mathcal{O}(\Delta t^{3/2})  \\
&= \eta^{n+1}_{j} - \eta^{n}_{j} + \mathrm{ad}_{V}\eta^{n+1/2}_{j} + \mathcal{O}\left(\Delta t^{3/2}\right).
\label{d_tau_inv_2}
\end{align*}
Similarly, for the spatial part we have
\begin{align*}
\delta \left[\tau^{-1}\left(\left(g^{n}_{j}\right)^{-1}g^{n}_{j+1}\right)\right] 
&= 
\delta\gamma^{n}_{j+1/2}\Delta s
=
\mathrm{d}_{\tau^{-1}\left(\left(g^{n}_{j}\right)^{-1}g^{n}_{j+1}\right)}\tau^{-1}
\cdot
\left(
-\eta^{n}_{j} + \mathrm{Ad}_{\left(g^{n}_{j}\right)^{-1}g^{n}_{j+1}}\eta^{n}_{j+1}
\right)  \\
&= \eta^{n}_{j+1} - \eta^{n}_{j}
+ \mathrm{ad}_{\tau^{-1}\left(\left(g^{n}_{j}\right)^{-1}g^{n}_{j+1}\right)}\eta^{n}_{j+1/2} 
+ \mathcal{O}\left(\Delta s^{2}\right).
\end{align*}
The proof is completed.
\end{proof}

\section{Proof of Theorem~\ref{Theorem_App_Int}}\label{Proof_App_Int}
\begin{proof}
Taking the variation of the discrete action sum \eqref{discrete_act_sum_1} and using the approximate variation formulae \eqref{app_variation_formula_1}, we have
\begin{align*}
\delta A_{d} &= \sum_{n=0}^{N-1}\sum_{j=0}^{J-1}
\left[
\Delta t\Delta s\left<
\frac{\delta\ell}{\delta g}\left(g^{n}_{j},\zeta^{n+1/2}_{j},\gamma^{n}_{j+1/2}\right),
\delta g^{n}_{j}
\right>
\right. \\
&\qquad\qquad\quad
+ \Delta s\left<
\frac{\delta\ell}{\delta\zeta}\left(g^{n}_{j},\zeta^{n+1/2}_{j},\gamma^{n}_{j+1/2}\right),
\eta^{n+1}_{j} - \eta^{n}_{j}
+ \mathrm{ad}_{
\tau^{-1}\left(\left(g^{n}_{j}\right)^{-1}g^{n+1}_{j}\right)
}
\eta^{n+1/2}_{j}
\right>  \\
&\qquad\qquad\quad \left.
+ \Delta t\left<
\frac{\delta\ell}{\delta\gamma}\left(g^{n}_{j},\zeta^{n+1/2}_{j},\gamma^{n}_{j+1/2}\right),
\eta^{n}_{j+1} - \eta^{n}_{j}
+ \mathrm{ad}_{
\tau^{-1}\left(\left(g^{n}_{j}\right)^{-1}g^{n}_{j+1}\right)
}
\eta^{n}_{j+1/2}
\right>
\right]  \\
&= \sum_{n=1}^{N-1}\sum_{j=1}^{J-1}
\left[
\Delta t\Delta s\left<
T_{e}L_{g^{n}_{j}}^{\ast}
\frac{\delta\ell}{\delta g}\left(g^{n}_{j},\zeta^{n+1/2}_{j},\gamma^{n}_{j+1/2}\right),
\eta^{n}_{j}
\right>
\right.\\
&\qquad\qquad\quad
+ \Delta s\left<
\frac{\delta\ell}{\delta\zeta}\left(g^{n-1}_{j},\zeta^{n-1/2}_{j},\gamma^{n-1}_{j+1/2}\right)
- \frac{\delta\ell}{\delta\zeta}\left(g^{n}_{j},\zeta^{n+1/2}_{j},\gamma^{n}_{j+1/2}\right),
\eta^{n}_{j}
\right> \\
&\qquad\qquad\quad
+ \frac{\Delta s}{2}
\left<
\mathrm{ad}^{*}_{
\tau^{-1}\left(\left(g^{n}_{j}\right)^{-1}g^{n+1}_{j}\right)
} 
\frac{\delta\ell}{\delta\zeta}\left(g^{n}_{j},\zeta^{n+1/2}_{j},\gamma^{n}_{j+1/2}\right)
,\eta^{n}_{j}
\right> \\
&\qquad\qquad\quad
+ \frac{\Delta s}{2}
\left<
\mathrm{ad}^{*}_{
\tau^{-1}\left(\left(g^{n-1}_{j}\right)^{-1}g^{n}_{j}\right)
}
\frac{\delta\ell}{\delta\zeta}\left(g^{n-1}_{j},\zeta^{n-1/2}_{j},\gamma^{n-1}_{j+1/2}\right)
,\eta^{n}_{j}
\right> \\
&\qquad\qquad\quad
+
\Delta t
\left<
\frac{\delta\ell}{\delta\gamma}\left(g^{n}_{j-1},\zeta^{n+1/2}_{j-1},\gamma^{n}_{j-1/2}\right)
-
\frac{\delta\ell}{\delta\gamma}\left(g^{n}_{j},\zeta^{n+1/2}_{j},\gamma^{n}_{j+1/2}\right)
,\eta^{n}_{j}
\right> \\
&\qquad\qquad\quad
+
\frac{\Delta t}{2}
\left<
\mathrm{ad}^{*}_{
\tau^{-1}\left(\left(g^{n}_{j}\right)^{-1}g^{n}_{j+1}\right)
}
\frac{\delta\ell}{\delta\gamma}\left(g^{n}_{j},\zeta^{n+1/2}_{j},\gamma^{n}_{j+1/2}\right)
,\eta^{n}_{j}
\right> \\
&\qquad\qquad\quad \left. 
+
\frac{\Delta t}{2}
\left<
\mathrm{ad}^{*}_{
\tau^{-1}\left(\left(g^{n}_{j-1}\right)^{-1}g^{n}_{j}\right)
}
\frac{\delta\ell}{\delta\gamma}\left(g^{n}_{j-1},\zeta^{n+1/2}_{j-1},\gamma^{n}_{j-1/2}\right)
,\eta^{n}_{j}
\right>
\right]\\
&\quad
+ \sum_{j=0}^{J-1}\Delta s
\left(
\left<
\frac{\delta\ell}{\delta\zeta}\left(g^{N-1}_{j},\zeta^{N-1/2}_{j},\gamma^{N-1}_{j+1/2}\right),
\eta^{N}_{j}
\right>
-
\left<
\frac{\delta\ell}{\delta\zeta}\left(g^{0}_{j},\zeta^{1/2}_{j},\gamma^{0}_{j+1/2}\right),
\eta^{0}_{j}
\right>
\right)  \\
&\quad
+ \sum_{j=0}^{J-1}
\frac{\Delta s}{2}
\left<
\mathrm{ad}^{*}_{
\tau^{-1}\left(\left(g^{0}_{j}\right)^{-1}g^{1}_{j}\right)
}
\frac{\delta\ell}{\delta\zeta}\left(g^{0}_{j},\zeta^{1/2}_{j},\gamma^{0}_{j+1/2}\right)
,\eta^{0}_{j}
\right>  \\
&\quad
+ \sum_{j=0}^{J-1}
\frac{\Delta s}{2}
\left<
\mathrm{ad}^{*}_{
\tau^{-1}\left(\left(g^{N-1}_{j}\right)^{-1}g^{N}_{j}\right)
}
\frac{\delta\ell}{\delta\zeta}\left(g^{N-1}_{j},\zeta^{N-1/2}_{j},\gamma^{N-1}_{j+1/2}\right)
,\eta^{N}_{j}
\right>  \\
& \quad
+ \sum_{n=0}^{N-1}
\Delta t
\left(
\left<
\frac{\delta\ell}{\delta\gamma}\left(g^{n}_{J-1},\zeta^{n+1/2}_{J-1},\gamma^{n}_{J-1/2}\right),
\eta^{n}_{J}
\right>
-
\left<
\frac{\delta\ell}{\delta\gamma}\left(g^{n}_{0},\zeta^{n+1/2}_{0},\gamma^{n}_{1/2}\right),
\eta^{n}_{0}
\right>
\right)  \\
&\quad
+ \sum_{n=0}^{N-1}
\frac{\Delta t}{2}\left<
\mathrm{ad}^{*}_{
\tau^{-1}\left(\left(g^{n}_{0}\right)^{-1}g^{n}_{1}\right)
}
\frac{\delta\ell}{\delta\gamma}\left(g^{n}_{0},\zeta^{n+1/2}_{0},\gamma^{n}_{1/2}\right)
,\eta^{n}_{0}
\right>  \\
&\quad
+ \sum_{n=0}^{N-1}
\frac{\Delta t}{2}
\left<
\mathrm{ad}^{*}_{
\tau^{-1}\left(\left(g^{n}_{J-1}\right)^{-1}g^{n}_{J}\right)
}
\frac{\delta\ell}{\delta\gamma}\left(g^{n}_{J-1},\zeta^{n+1/2}_{J-1},\gamma^{n}_{J-1/2}\right)
,\eta^{n}_{J}
\right>,
\end{align*}
where we have rearranged the indices $n$ and $j$ in the second equality. By the fixed point condition $\eta^{0}_{j} = \eta^{N}_{j} = 0$ for $j=1,\dots,J$, we obtain the desired stochastic integrator \eqref{App_Integrator_1} and the discrete boundary condition \eqref{App_boun_con_1}.
\end{proof}

\section{Proof of Proposition~\ref{ene_evo}}\label{app_proof}
In order to prove Proposition~\ref{ene_evo}, we first recall a lemma on Lie-Poisson bracket.
\begin{lemma}[\cite{ChenShi2025}]\label{ChenShi2025}
Given a field-theoretic mechanical system with a $G$-invariant Lagrangian $\ell\left(\zeta,\gamma\right) = L(g,g\zeta,g\gamma)$, the Lie-Poisson bracket is given by
\begin{equation*}
\left\{f,g\right\}(\gamma,\mu)
=
\left<
\frac{\delta f}{\delta\gamma},
\frac{\partial}{\partial s}\frac{\delta g}{\delta\mu} 
+ \mathrm{ad}_{\gamma}\frac{\delta g}{\delta\mu}
\right>
+
\left<
\mathrm{ad}^{*}_{\frac{\delta g}{\delta\mu}}\mu
+ \frac{\partial}{\partial s}\frac{\delta g}{\delta\gamma}
- \mathrm{ad}^{*}_{\gamma}\frac{\delta g}{\delta\gamma},
\frac{\delta f}{\delta\mu}
\right>.
\end{equation*}
\end{lemma}

\paragraph{\textbf{Proof of Proposition~\ref{ene_evo}.}} Using the Stratonovich chain rule and \eqref{sto_Hamel_Ham}-\eqref{sto_com_Ham}, we have
\begin{align*}
d_{t}f(\gamma,\mu) &= \left<\frac{\delta f}{\delta\gamma},d_{t}\gamma\right> + \left<d_{t}\mu, \frac{\delta f}{\delta\mu}\right> \\
&=
\left<
\frac{\delta f}{\delta\gamma},
\left(
\frac{\partial}{\partial s}\frac{\delta h}{\delta\mu}
+ \mathrm{ad}_{\gamma}\frac{\delta h}{\delta\mu}
\right)dt
+ \left(
\mathrm{ad}_{\gamma}\frac{\delta\Phi}{\delta\mu}
+ \frac{\partial}{\partial s}\frac{\delta\Phi}{\delta\mu}
\right)\circ dW
\right>  \\
&\quad
+
\left<
\left(
\mathrm{ad}^{*}_{\frac{\delta h}{\delta\mu}}\mu
+ \frac{\partial}{\partial s}\frac{\delta h}{\delta\gamma}
- \mathrm{ad}^{*}_{\gamma}\frac{\delta h}{\delta\gamma}
\right)dt
+ \mathrm{ad}^{*}_{\frac{\delta\Phi}{\delta\mu}}\mu\circ dW,
\frac{\delta f}{\delta\mu}
\right>  \\
&=
\left(
\left<
\frac{\delta f}{\delta\gamma},
\frac{\partial}{\partial s}\frac{\delta h}{\delta\mu} 
+ \mathrm{ad}_{\gamma}\frac{\delta h}{\delta\mu}
\right>
+
\left<
\mathrm{ad}^{*}_{\frac{\delta h}{\delta\mu}}\mu
+ \frac{\partial}{\partial s}\frac{\delta h}{\delta\gamma}
- \mathrm{ad}^{*}_{\gamma}\frac{\delta h}{\delta\gamma},
\frac{\delta f}{\delta\mu}
\right>
\right)dt  \\
&\quad
+
\left(
\left<
\frac{\delta f}{\delta\gamma},
\mathrm{ad}_{\gamma}\frac{\delta\Phi}{\delta\mu}
+ \frac{\partial}{\partial s}\frac{\delta\Phi}{\delta\mu}
\right>
+
\left<
\mathrm{ad}^{*}_{\frac{\delta\Phi}{\delta\mu}}\mu,
\frac{\delta f}{\delta\mu}
\right>
\right)\circ dW  \\
&=
\left\{f,h\right\}dt + \left\{f,\Phi\right\}\circ dW,
\end{align*}
where the last equality follows from Lemma~\ref{ChenShi2025}.

To prove the converse, we take $f=\gamma$ and $f=\mu$ in \eqref{SLP}, respectively. Then, it is straightforward to obtain the stochastic field-theoretic equation \eqref{sto_Hamel_Ham} and the stochastic compatibility condition \eqref{sto_com_Ham} in Hamiltonian forms.  \hfill $\Box$

\subsubsection*{Acknowledgments}
We are grateful to Dmitry Zenkov (NCSU) for fruitful discussions during this work. M. Du and D. Shi were partially supported by NSFC 12272037 and 12232009. F. Gay-Balmaz was supported by a startup grant from Nanyang Technological University.

\bibliographystyle{cas-model2-names}
\bibliography{cas-refs}






\end{document}